\documentclass[oneside,11pt]{amsart}
\usepackage{amsthm}
\usepackage{amssymb}
\usepackage{mathtools}
\usepackage{color}
\usepackage{float}
\usepackage{enumitem}
\usepackage[margin=1.5cm]{caption}

\renewcommand{\hom}{\operatorname{Hom}}

\newcommand{\N}{\mathbb{N}}
\newcommand{\R}{\mathbb{R}}
\newcommand{\Z}{\mathbb{Z}}

\newcommand{\D}{\mathbb{D}}

\newcommand{\K}{\mathbb{K}}

\newcommand{\rom}[1]{\uppercase\expandafter{\romannumeral #1\relax}}

\newcommand{\ad}{\mathrm{ad}\,}

\newcommand{\Aut}{\mathrm{Aut}}

\newcommand{\charK }{\mathrm{char}\,\K }

\newcommand{\chg }[1]{\widehat{#1}}
\newcommand{\const }{\mathrm{const}}
\newcommand{\cC}{\mathcal{C}}
\newcommand{\cF }{\mathcal{F}}
\newcommand{\cE }{\mathcal{E}}
\newcommand{\cH }{\mathcal{H}}
\newcommand{\cS }{\mathcal{S}}
\newcommand{\cO }{\mathcal{O}}
\newcommand{\cW }{\mathcal{W}}
\newcommand{\cX }{\mathcal{X}}
\newcommand{\cR }{\mathcal{R}}
\newcommand{\cD }{\mathcal{D}}
\newcommand{\cY }{\mathcal{Y}}
\newcommand{\End}{\mathrm{End}}
\newcommand{\Ggen}{\mathcal{E}}
\newcommand{\id}{\mathrm{id}}

\newcommand{\NA}{\mathcal{B}}

\newcommand{\ot }{\otimes }

\newcommand{\Res}{\mathrm{Res}}
\newcommand{\supp}{\mathrm{supp}\,}

\newcommand{\trid}{\triangleright}
\newcommand{\ydG}{\prescript{G}{G}{\mathcal{YD}}}
\newcommand{\ydH}{\prescript{H}{H}{\mathcal{YD}}}
\newcommand{\roots}{\mathbf{\Delta}}
\newcommand{\lroots}{\mathbf{\Delta}_+^{\operatorname{long}}}
\newcommand{\sroots}{\mathbf{\Delta}_+^{\operatorname{short}}}
\newcommand{\rroots}[1]{\mathbf{\Delta}^{\operatorname{re}\,#1}}

\newcounter{xposi}
\newcounter{yposi}

\newcommand{\DDvertexone}[2]{\put(#1,#2){\circle*{1.5}}}
\newcommand{\DDvertextwo}[2]{\multiput(#1,#2)(0,4)2{\circle*{1.5}}}
\newcommand{\DDvertexsq}[2]{
  \multiput(#1,#2)(0,4)2{\circle*{1.5}}
  \setcounter{xposi}{#1}
	\addtocounter{xposi}4
	\multiput(\value{xposi},#2)(0,4)2{\circle*{1.5}}
}
\newcommand{\DDvertexT}[2]{
  \setcounter{xposi}{#1}
  \setcounter{yposi}{#2}
	\put(\value{xposi},\value{yposi}){\circle*{1.5}}
	\addtocounter{xposi}{-2}
	\addtocounter{yposi}{2}
	\put(\value{xposi},\value{yposi}){\circle*{1.5}}
	\addtocounter{xposi}{4}
	\put(\value{xposi},\value{yposi}){\circle*{1.5}}
	\addtocounter{xposi}{-2}
	\addtocounter{yposi}{-4}
	\put(\value{xposi},\value{yposi}){\circle*{1.5}}
}
\newcommand{\DDvertexthree}[2]{
  \multiput(#1,#2)(3,2)2{\circle*{1.5}}
  \multiput(#1,#2)(3,-2)2{\circle*{1.5}}
}

\newcommand{\comment}[1]{}

\makeatletter
\numberwithin{equation}{section}
\numberwithin{figure}{section}

\theoremstyle{plain}
\newtheorem{thm}{Theorem}[section]
\newtheorem*{thm*}{Theorem}
\newtheorem*{classification*}{Classification theorem}
\newtheorem{lem}[thm]{Lemma}
\newtheorem{cor}[thm]{Corollary}

\newtheorem*{notation*}{Notation}
\newtheorem{pro}[thm]{Proposition}

\newtheorem{defn}[thm]{Definition}

\theoremstyle{remark}
\newtheorem{rem}[thm]{Remark}

\makeatother

\begin{document}

\title[Nichols algebras over non-abelian groups]{A classification
of Nichols algebras\\
of semi-simple Yetter-Drinfeld modules\\
over non-abelian groups}

\author{I. Heckenberger}
\address{Philipps-Universit\"at Marburg\\ 
FB Mathematik und Informatik \\
Hans-Meerwein-Stra\ss e\\
35032 Marburg, Germany}
\email{heckenberger@mathematik.uni-marburg.de}

\author{L. Vendramin}
\address{Departamento de Matem\'atica, FCEN, Universidad de Buenos Aires,
Pabell\'on 1, Ciudad Universitaria (1428), Buenos Aires, Argentina}
\email{lvendramin@dm.uba.ar}

\begin{abstract}
	Over fields of arbitrary characteristic we classify all braid-indecomposable
	tuples of at least
	two absolutely simple Yetter-Drinfeld modules over non-abelian groups such
	that the group is generated by the support of the tuple and the Nichols
	algebra of the tuple is finite-dimensional. Such tuples are
	classified in terms of analogs of Dynkin diagrams
	which encode
	much information about the Yetter-Drinfeld modules. We also	compute
	the dimensions of these finite-dimensional Nichols algebras. Our proof
	uses essentially
	the Weyl groupoid of a tuple of simple Yetter-Drinfeld modules
	and our previous result on pairs.
\end{abstract}

\maketitle

\setcounter{tocdepth}{1}
\tableofcontents{}

\section*{Introduction}

Let $\K $ be a field and let $G$ be a group. The $G$-graded $\K G$-modules (also
known as Yetter-Drinfeld modules) form a braided vector space. To any braided
vector space $V$ there exists up to isomorphism a unique connected graded
braided Hopf algebra $\NA (V)$ generated by $V$, such that the generators have
degree $1$ and all primitive elements are in $V$. This braided Hopf algebra is
known as the \textit{Nichols algebra of} $V$. It is a fundamental problem in
Hopf algebra theory to understand the structure of Nichols algebras, see for
example \cite{MR1913436} and \cite{MR1907185}.  Besides their applications to
quantum groups and Hopf algebras, Nichols algebras have many other interesting
applications such as Schubert calculus \cite{MR2209265}, Lie superalgebras, see
\cite[Example 5.2]{ICM}, and logarithmic conformal field theories
\cite{MR3171534,MR2965674,MR3146017}. 

First definitions and structural results on Nichols algebras were obtained by
Nichols \cite{MR0506406}. Nichols algebras were rediscovered later by Woronowicz
\cite{MR901157,MR994499} and Majid \cite{MR2106930}, and they were used as a
basic object in the lifting method of Andruskiewitsch and Schneider
\cite{MR1659895} to classify (finite-dimensional) pointed Hopf algebras
\cite{MR1780094,MR2108213,MR2630042}. Nowadays there exist generalizations of
this method to other classes of Hopf algebras \cite{MR3032811}.  A common step
in these methods is to determine all Nichols algebras satisfying a finiteness or
a moderate growth condition. Whereas Nichols was only able to determine Nichols
algebras over very small abelian groups, the theory of Weyl groupoids
\cite{MR2207786} lead to satisfactory classification results for arbitrary
finite abelian groups. Among the results in this direction we
mention the classification of finite-dimensional Nichols algebras of
diagonal type of \cite{MR2294771,MR2295200,MR2379892, MR2500361,MR2462836} and
\cite{MR3313687}, and the results related to presentations of such Nichols
algebras \cite{MR3181554,MR3420518}.  

Based on the successful experience with Weyl groupoids related to
Yetter-Drinfeld modules over abelian groups, the theory was extended
to arbitrary Hopf algebras with bijective antipode and semi-simple
Yetter-Drinfeld modules over them \cite{MR2766176}.  It turned out that Weyl
groupoids provide very strong information on the growth and on combinatorial
properties of Nichols algebras in the case of several direct summands. It is
remarkable that this theory is also very useful for studying Nichols algebras of
simple Yetter-Drinfeld modules. Indeed, the only known tool today to study Nichols
algebras over such Yetter-Drinfeld modules
is to look at braided subspaces which can be viewed as semi-simple
Yetter-Drinfeld modules over another Hopf algebra, see for example
\cite{MR2786171,MR2745542}.  From this point of view, the study of Nichols
algebras of semi-simple Yetter-Drinfeld modules is also crucial and has  
several potential applications.

\medskip
Let us explain the main results of this paper and the strategy of the proof.  
Let $G_0$ be a group and let $V$
be a finite-dimensional Yetter-Drinfeld module over $G_0$.  By restriction of
the module structure, one can view $V$ as a Yetter-Drinfeld module over the
subgroup $G$ of $G_0$ generated by the support of $V$. Moreover, under some
assumptions on $G$ and the field $\K$ one can decompose $V$ into the direct sum
of absolutely simples.
Motivated by this setting, we study tuples $M=(M_1,\dots ,M_\theta)$
of absolutely simple Yetter-Drinfeld modules over a non-abelian group $G$, where
$\theta \in \N$, such that $G$ is generated by the support of
$V=\oplus_{i=1}^\theta M_i$.

Let us add here a side remark.
	The reflection theory and the Weyl groupoid exist for tuples of
	simples. However, allowing simple Yetter-Drinfeld modules would lead to a
	discussion of group representations depending heavily on the field. Further, one
would loose essential parts of the combinatorics of the Weyl groupoid: In the
worst case one has only one simple summand over the base field
instead of several absolutely simples over an extended field.
On the other hand, field extensions of Nichols algebras are well understood. Therefore,
in general it is more promising to extend the field appropriately before
studying a Nichols algebra.

Since $V$ is a braided vector space, it admits a
Nichols algebra $\NA (V)$.
The general theory implies that in some cases, containing those with $\dim \NA
(V)<\infty $,
one can attach to $M$ a connected finite Cartan graph of rank $\theta $, see
Section~\ref{section:preliminaries} for the definitions.
If $\theta =1$, then the Cartan graph contains no information about $\NA (V)$.
Therefore we restrict our attention to the case $\theta \ge 2$.
Our aim is now to provide a classification of $\theta$-tuples $M$ of absolutely
simple Yetter-Drinfeld modules over non-abelian groups such that the group is
generated by the support of $\oplus_{i=1}^\theta M_i$, and $M$ has an
indecomposable finite Cartan graph.  We record that the indecomposability
assumption on the finite Cartan graph is merely of technical nature. It allows
us to exclude the components of the Cartan graph of rank one. Since the
classification in the case of two simple summands was performed in
\cite{MR3276225,MR3272075,rank2}, here we consider the case $\theta\geq3$. 
To write our classification theorem, we introduce two concepts:

\medskip
\noindent\emph{Braid-indecomposability.} The braid-indecomposability of the
tuple of Yetter-Drinfeld modules records the indecomposability assumption on the
finite Cartan graph, see Definition~\ref{def:braidindecomposable}.  

\medskip
\noindent\emph{Skeletons.} To describe the structure of the 
Yetter-Drinfeld modules involved, we make use of 
diagrams which are analogs of Dynkin diagrams of finite type.
We call them \textit{skeletons of finite type}.
See Definition~\ref{def:skeleton} 
for the definition of a skeleton and Figure~\ref{fig:fdNA} for skeletons of
finite type.

\medskip
A consequence of our main result is the following classification, see Theorem~\ref{thm:main}.

\begin{classification*}
	Let $\theta\geq3$ and let $M=(M_1,\dots,M_\theta)$ be a braid-indecomposable
	tuple of Yetter-Drinfeld modules over a non-abelian group $G$ such that 
	the support of $M_1\oplus\cdots\oplus M_\theta$ generates $G$. Then the Nichols
	algebra $\NA(M_1\oplus\cdots\oplus M_\theta)$ is finite-dimensional if and
	only if $M$ has a skeleton of finite type.
\end{classification*}

We remark that no assumption on the characteristic of the field $\K$ is needed.
The theorem is the culmination of several theorems stated in
Section~\ref{section:main_results}. These theorems contain the dimensions and
the Hilbert series for the Nichols algebras of the classification. Almost all of
the examples appearing in our classification admit a standard (classical) root
system.  The dimensions of the Nichols algebras admiting a standard root system
are shown in Table~\ref{tab:standard}. 

\begin{table}[h]
	\centering
	\caption{Finite-dimensional Nichols algebras with a standard root system.}
	\begin{tabular}{|c|c|c|c|c|c|c|c|c|}
		\hline dimension & \rule{0pt}{3ex} $2^{\theta(\theta+1)}$ & $4^{\theta
		(\theta-1)}3^{2\theta }$ & $2^{2\theta^2-\theta}$ & $4^{\theta (\theta -1)}$
		& $4^{36}$ & $4^{63}$ & $4^{120}$ & $4^{18}$ \\
		\hline
		root system & \rule{0pt}{2.5ex}  $A_\theta $ & $B_\theta $ & $C_\theta $ &
		$D_\theta $ & $E_6$ & $E_7$ & $E_8$ & $F_4$\\
		\hline
		$\charK $ & & 3 & $\ne 2$ & & & & & $\ne 2$\\
		\hline
	\end{tabular}
    \label{tab:standard}
\end{table}

\comment{
\begin{table}[H]
	\centering
	\begin{tabular}{|c|c|c|}
		\hline 
		dimension & root system & $\charK$\tabularnewline
		\hline 
		$4^{\theta(\theta+1)/2}$ \rule{0pt}{3ex} & $A_{\theta}$ & \tabularnewline
		\hline 
		$2^{2\theta(\theta-1)}3^{2\theta}$ \rule{0pt}{3ex} & $B_{\theta}$ & $3$\tabularnewline
		\hline 
		$2^{2\theta^{2}-\theta}$ \rule{0pt}{3ex} & $C_{\theta}$ & $\ne2$ \tabularnewline
		\hline 
		$4^{\theta(\theta-1)}$ \rule{0pt}{3ex} & $D_{\theta}$ & \tabularnewline
		\hline 
		$4^{36}$ & $E_{6}$ \rule{0pt}{3ex} & \tabularnewline
		\hline 
		$4^{63}$ & $E_{7}$ \rule{0pt}{3ex} & \tabularnewline
		\hline 
		$4^{120}$ & $E_{8}$ \rule{0pt}{3ex} & \tabularnewline
		\hline 
		$2^{36}$ & $F_{4}$ \rule{0pt}{3ex} & $2$\tabularnewline
		\hline 
	\end{tabular}
	\caption*{Finite-dimensional Nichols algebras with a standard root system.}
\end{table}
}
It is remarkable that in characteristic zero and in the case where $\theta\geq
4$, these finite-dimensional Nichols algebras appear in the work of Lentner
\cite{MR3253277} related to coverings of Nichols algebras.

In the case of three simple summands, one has an additional family of examples
which admits a non-standard root system. The dimensions of these Nichols
algebras are shown in Table~\ref{tab:nonstandard}.

\begin{table}[H]
	\centering
	\caption{Finite-dimensional Nichols algebras with a non-standard root
	system of rank three.}
	\begin{tabular}{|c|c|c|c|}
		\hline dimension & \rule{0pt}{3ex} $3^6\,12^7$ & $2^{12}\,12^4$ &
		$6^6\,12^7$ \\
		\hline
		$\charK $ & 2 & 3 & $\ne 2,3$\\
		\hline
	\end{tabular}
    \label{tab:nonstandard}
\end{table}

\comment{
\begin{table}[h]
	\centering
	\begin{tabular}{|c|c|}
		\hline 
		dimension & $\charK$\tabularnewline
		\hline 
		$3^{6}12^{7}$ \rule{0pt}{3ex}& $2$\tabularnewline
		\hline 
		$2^{12}12^{4}$ \rule{0pt}{3ex}& $3$\tabularnewline
		\hline 
		$6^{6}12^{7}$\rule{0pt}{3ex} & $\ne2,3$\tabularnewline
		\hline 
	\end{tabular}
	\caption*{Finite-dimensional Nichols algebras with a non-standard root
	system of rank three.}
\end{table}
}

The proof of Theorem~\ref{thm:main} is based on a general PBW-type theorem on
certain Nichols algebras from \cite[Thm.\,2.6]{MR2732989}, see
Theorem~\ref{thm:HS}, on the classification in the case
$\theta=2$ \cite{rank2}, and on the classification of connected indecomposable
finite Cartan graphs of rank three \cite{MR2869179}. In fact, we only need
Lemma~\ref{lem:noA} from \cite{MR2869179}, for the proof of which we had to use
the main result in \cite{MR2869179}.  In order to simplify our approach further,
we prove the following theorem, see Theorem~\ref{thm:finitetype}. 

\begin{thm*}
Any connected indecomposable finite Cartan graph has an object with a Cartan
matrix of finite type. 
\end{thm*}

This result is of independent interest and its proof does not use the
classification of finite Cartan graphs \cite{MR2869179,MR3341467}.  At an early stage
of our work we had a proof of our main classification theorem without using the
classification in \cite{MR2869179}, but it was much more technical than the
present work. 

\medskip
The paper is organized as follows.
Notations, terminology and 
a review of the theory of Weyl groupoids of tuples of
simple Yetter-Drinfeld modules over groups is given 
in Section~\ref{section:preliminaries}.
In Section~\ref{section:main_results} we
state the main result of the paper, a classification of finite-dimensional
Nichols algebras of semi-simple Yetter-Drinfeld modules over groups
in terms
of skeletons of finite type. This section also contains the
Hilbert series of each of the Nichols algebras appearing in
our classification, see Theorems~\ref{thm:ADE}, \ref{thm:C}, \ref{thm:B1},
\ref{thm:B2}, and \ref{thm:F4}.
In Sections~\ref{section:Cartan_rank3}
and~\ref{section:Cartan} we collect useful facts about finite Cartan graphs. In
particular, in Theorem~\ref{thm:finitetype} we prove that
every finite connected indecomposable Cartan graph contains a point
with a Cartan matrix of finite type. Section~\ref{section:helpful}
contains several useful lemmas related to the structure of Yetter-Drinfeld
modules over arbitrary groups.
Sections~\ref{section:ADE}--\ref{section:F4} are devoted
to prove the structure theorems in cases $ADE$, $C$, $B$ and $F_4$. The main
theorem, Theorem~\ref{thm:main}, is then proved in Section~\ref{section:main}.
The paper contains two appendices.  Appendix~\ref{appendix:reflections} is
devoted to the structure theory of $(\ad V)(W)$ for particular Yetter-Drinfeld
modules $V$ and $W$. Some known results are cited and some new results, which
are needed in the paper, are obtained using known methods.
Appendix~\ref{appendix:rank2} reviews the main results of \cite{rank2}, where
finite-dimensional Nichols algebras of direct sums of two absolutely simple
Yetter-Drinfeld modules were studied. 

\section{Preliminaries}
\label{section:preliminaries}

\subsection{}
As usual, $\N$ is the
set of positive integers, $\N_0=\N\cup\{0\}$, $\Z$ is the set of integers, $\K$
is an arbitrary field of characteristic $\charK$ and
$\K^\times=\K\setminus\{0\}$. 

For a set $X$ we write $|X|$ for the cardinality of $X$. 

For a group $G$ we write $\widehat{G}$ for the set of linear characters of $G$
and $Z(G)$ for the center of $G$. For $g\in G$, we write $G^g$ for the
centralizer of $g$ in $G$. The conjugacy class of $g$ will be denoted by $g^G$.
For any $g,h\in G$ we sometimes write $g\triangleright h$ for $ghg^{-1}$.
If $X\subseteq G$ is a subset, then $\langle X\rangle$ denotes the subgroup of $G$
generated by $X$. 

The category of Yetter-Drinfeld modules over $G$ will be
denoted by $\ydG$. 
Recall that a Yetter-Drinfeld module over $G$, also called a $G$-graded $\K
G$-module, is
a $\K G$-module $V=\oplus_{g\in G}V_{g}$ such that $hV_{g}\subseteq
V_{hgh^{-1}}$ for all $g,h\in G$. It is a braided vector space with braiding
$c\colon V\otimes V\to V\otimes V$ defined by $c(u\otimes v)=gv\otimes u$ for
all $u\in V_g$, $v\in V$.  The \textbf{support} of $V$ is
\[
\supp V=\{g\in G:V_g\ne0\}.
\]
We say that $V$ is \textbf{absolutely simple} if $V\ne 0$ and if for any field
extension $\mathbb{L}$ of $\K$ the only Yetter-Drinfeld submodules of
$\mathbb{L}\otimes _\K V$ over $\mathbb{L}G$ are $\{0\}$ and $\mathbb{L}\otimes
_\K V$. (Absolutely) simple
Yetter-Drinfeld modules over $G$ are parametrized by pairs
$(g^G,\rho)$, where $g^G$ is a conjugacy class of $G$ and $\rho:\K G^g\to
\mathrm{End}(W)$ is an (absolutely) irreducible
representation of the centralizer $G^g$.  The (absolutely)
simple Yetter-Drinfeld modules over $G$ are
\[
        M(g^G,\rho)=\mathrm{Ind}_{G^g}^G\rho
\]
with the induced action $y(x\otimes w)=yx\otimes w$ for $x,y\in G$ and $w\in W$,
and the coaction $\delta :M(g^G,\rho )\to \K G\otimes M(g^G,\rho )$
is given by $\delta(x\otimes w)=xgx^{-1}\otimes(x\otimes w)$
for all $w\in W$, $x\in G$.
One also says that $x\otimes w$ has $G$-degree $xgx^{-1}$.

For a $G$-graded $\K G$-module $V$ we write $\NA(V)$ for the Nichols algebra
of $V$. 
Nichols algebras are connected strictly $\N_0$-graded braided Hopf algebras with
$V$ as degree one part. The Hilbert series of an $\N_0$-graded
algebra $R=\oplus_{n\in\N_0}R_n$ is $\sum_{n\geq0}(\dim R_n)t^n\in\Z[[t]]$.  
For all $k\in\N_0$ and $t\in \K $ let $(k)_t=1+t+\cdots+t^{k-1}$ be the usual
$t$-number. 

Many examples of finite-dimensional Nichols algebras of pairs of absolutely
simple Yetter-Drinfeld modules are related to the groups $\Gamma _n$ for $n\in
\{2,3,4\}$ defined in \cite{MR2732989}: For all $n\in \N _{\ge 2}$,
the group $\Gamma _n$ is defined by the generators $a,b,\nu $ and relations
$$ ba=\nu ab,\quad \nu a=a\nu ^{-1},\quad \nu b=b\nu ,\quad \nu ^n=1. $$

\subsection{Weyl groupoids and root systems}
\label{subsection:Weyl}

We review the basics of the theory of Weyl groupoids of tuples of simple
Yetter-Drinfeld modules over groups. We refer to \cite{MR2766176} and
\cite{MR2734956,MR2732989} for details and proofs.  We use the terminology
introduced in \cite{MR3313687} after several discussions with Andruskiewitsch
and Schneider.

Let $\theta\in\N$ and let $I=\{1,\dots,\theta\}$.
Let $\cX$ be a non-empty set and for each $X\in\cX$ let $A^{X}=(a_{ij}^X)_{1\leq
i,j\leq\theta}$ be a generalized Cartan matrix. For all
$i\in I$ let $r_i\colon\cX\to\cX$ be a map. 
The quadruple 
\[
	\cC=\cC(I,\cX,r,A),
\]
where $r=(r_i)_{i\in I}$ and $A=(A^X)_{X\in\cX}$,
is called a \textbf{semi-Cartan graph} if $r_i^2=\id_X$ for all $i\in I$, and
$a_{ij}^X=a_{ij}^{r_i(X)}$ for all $X\in\cX$ and $i,j\in I$. We say that a
semi-Cartan graph $\cC$ is \textbf{connected} if there is no proper non-empty
subset $\cY\subset\cX$ such that $r_i(Y)\in\cY$ for all $i\in I$ and $Y\in\cY$.

Let $\cC=\cC(I,\cX,r,A)$ be a semi-Cartan graph. 
There exists a unique category $\cD(\cX,I)$ with $\cX$ as its set of objects and
with morphisms
$$\hom(X,Y)=\{(Y,f,X):f\in\End(\Z^\theta)\}$$
for $X,Y\in\cX$ with the
composition defined by 
\[(Z,g,Y)\circ (Y,f,X)=(Z,gf,X)\]
for all $X,Y,Z\in\cX$ and
$f,g\in\End(\Z^\theta)$.

We write $\alpha_1,\dots,\alpha_\theta$ for the standard basis of $\Z^{\theta}$.

For each $X\in\cX$ and $i\in I$ let
\[
	s_i^X\in\Aut(\Z^\theta),\quad
	s_i^X\alpha_j=\alpha_j-a_{ij}^X\alpha_i
\]
for all $j\in I$. Let $\cW(\cC)$ be the subcategory of
$\cD(\cX,I)$ generated by the morphisms $(r_i(X),s_i^X,X)$, where $i\in I$
and $X\in\cX$. Then $\cW(\cC)$ is a groupoid. For any $X,Y\in \cX $ and $f\in
\Aut (\Z^\theta)$ with $w=(Y,f,X)\in \hom (X,Y)$ and for any $\alpha \in
\Z^\theta $ we also write $w\alpha $ for $f\alpha $.
For all $k\in \N _0$, $i_1,\dots ,i_k\in I$, $X_0,X_1,\dots ,X_k\in \cX$ with
$r_{i_m}(X_m)=X_{m-1}$ for all $1\le m\le k$ let
$$\id_{X_0}s_{i_1}\cdots s_{i_k} = s_{i_1}^{X_1}s_{i_2}^{X_2}\cdots
s_{i_k}^{X_k}\in \hom (X_k,X_0).$$

For each $X\in\cX$ the set of \textbf{real roots} of $\cC$ at $X$ is 
\[
	\rroots X=\{w\alpha_i:w\in \cup _{Y\in \cX}\hom(Y,X)\}\subseteq\Z^\theta.
\]
The sets of \textbf{positive real roots} and \textbf{negative real roots} are
\[
\rroots X_+=\rroots X\cap\N_0^I,
\quad
\rroots X_-=\rroots X\cap-\N_0^I,
\]
respectively. The semi-Cartan graph $\cC$ is \textbf{finite} if its set of real
roots at $X$ is finite for all $X\in\cX$. The semi-Cartan graph $\cC$ is a
\textbf{Cartan graph} if the following hold:
\begin{enumerate}
	\item For each $X\in\cX$ the set $\rroots X$
		consists of positive and negative roots.
	\item Let $X\in\cX$ and $i,j\in I$. If
		$t_{ij}^X=|\rroots X\cap(\N_0\alpha_i+\N_0\alpha_j)|<\infty$
		then $(r_ir_j)^{t_{ij}^X}(X)=X$.
\end{enumerate}
If $\cC$ is a Cartan graph, the groupoid $\cW(\cC)$ is the \textbf{Weyl groupoid}
of $\cC$. 

For all points $X\in \cX $ of the semi-Cartan graph $\cC$
let $\roots ^X\subseteq \Z^\theta $.
We say that $\cR=\cR(\cC,(\roots^X)_{X\in\cX})$ is a \textbf{root system}
of type $\cC$ if the following conditions hold:
\begin{enumerate}
	\item $\roots^X=(\roots^X\cap\N_0^I)\cup
		-(\roots^X\cap\N_0^I)$ for all $X\in \cX$.
	\item $\roots^X\cap\Z\alpha_i=\{\alpha_i,-\alpha_i\}$ for all $i\in
		I$, $X\in \cX $.
	\item $s_i^X(\roots^X)=\roots^{r_i(X)}$ for all $i\in I$, $X\in \cX $. 
	\item $(r_ir_j)^{m_{ij}^X}(X)=X$ for all $i,j\in I$ with $i\ne j$ and all
		$X\in \cX $, where
		$m_{ij}^X=|\roots^X\cap(\N_0\alpha_i+\N_0\alpha_j)|$ is finite. 
\end{enumerate}
Note that (4) is similar to the condition (2) of a Cartan graph, but here
$\roots $ is involved instead of the set of real roots. Axiom (4) is necessary for
the Coxeter relations of the Weyl groupoid.
For any finite Cartan graph $\cC$ the family $(\rroots X)_{X\in \cX }$ defines
the unique root system of type $\cC$, see \cite[Props. 2.9 and 2.12]{MR2498801}.

A connected semi-Cartan graph is \textbf{indecomposable} if there exists
$X\in \cX $ such that
the Cartan matrix $A^X$ is indecomposable, that is, there are no disjoint subsets
$I_1,I_2\subset I$ such that $I_1,I_2\not=\emptyset$, $I_1\cup I_2=I$, and
$a_{ij}^X=0$ for all $i\in I_1$, $j\in I_2$.
It is known by \cite[Prop.~4.6]{MR2498801}
that if a connected finite Cartan graph $\cC $ is indecomposable, then
$A^X$ is an indecomposable Cartan matrix for all points $X$ of $\cC $.

A semi-Cartan graph $\cC$ is \textbf{standard} if $A^X=A^Y$ for all $X,Y\in\cX$.  In
this case the real roots form the set of real roots of the Weyl group attached
to the Cartan matrix, and hence $\cC $ is a Cartan graph.
We then say that the Weyl groupoid $\cW(\cC)$ is standard.
If $\cR$ is a root system of a standard Cartan graph $\cC$, then we say that $\cR$ is
standard. The terminology is based on \cite{MR2742620}. 

Let us review the connections between Cartan graphs and Nichols algebras.  Let
$G$ be a group and $\ydG$ be the category of Yetter-Drinfeld
modules over $G$. We write $\cF_\theta^G$ for the set of $\theta$-tuples of
finite-dimensional absolutely simple objects in $\ydG$ and $\cX_\theta$ for the
$\theta$-tuples of isomorphism classes of finite-dimensional absolutely simple
objects in $\ydG$.

For any Yetter-Drinfeld module $U$ over $G$ and any $x\in U$, $y\in \NA (U)$ we
write
$(\ad x)(y)$ for $\mathrm{mult}(\id -c)(x\otimes y)$, where $\mathrm{mult}$
denotes the multiplication map in $\NA (U)$ and $c$ is the braiding of $\NA
(U)$.  Then for any two subsets $U'\subseteq U$ and $U''\subseteq \NA (U)$ we
write $(\ad U')(U'')$ for the linear span of the elements $(\ad x)(y)$ with
$x\in U'$, $y\in U''$.

Let $\theta \in \N $ and let $I=\{1,\dots ,\theta \}$.
For $M=(M_1,\dots,M_\theta)\in\cF_\theta^G$ let
$[M]=([M_1],\dots,[M_\theta])\in\cX_\theta$ be the corresponding $\theta$-tuple
of isomorphism classes. For all $i\in I$ and $j\in I\setminus \{i\}$ let
$a_{ij}^M=-\infty $ if $(\ad M_i)^m(M_j)\ne 0$ for all $m\ge 0$ and let
\[
	a_{ij}^M=-\sup\{m\in\N_0:(\ad M_i)^m(M_j)\ne0\}
\]
otherwise. Moreover, let $a_{ii}^M=2$ for all $i\in I$.
Then $A^M=(a^M_{ij})_{i,j\in I}$ is called the \textbf{Cartan matrix of} $M$.
Clearly, $A^M$ depends only on the isomorphism class of $M$ and hence we also
write $A^{[M]}$ for $A^M$.

For all $i\in I$ the reflection map
$R_i\colon\cF_\theta^G\to\cF_\theta^G$ is defined by $R_i(N)=N$ if $a^N_{ij}=-\infty
$ for some $j\in I$, and by $R_i(N)=(N_1',\dots,N_\theta')$, where 
\[
N_j'=\begin{cases}
	(\ad N_i)^{-a_{ij}^N}(N_j) & \text{if $j\ne i$},\\
	N_i^* & \text{if $j=i$},
\end{cases}
\]
otherwise.

Since $[R_i(M)]=[R_i(N)]$ in $\cX_\theta$ for all $M,N\in\cF_\theta^G$ with
$[M]=[N]$ and all $i\in I$, we may define
$r_i\colon\cX_\theta\to\cX_\theta$ by $r_i([N])=[R_i(N)]$ for all
$i\in I$. We then define
\begin{align*}
	&\cF_\theta^G(M)=\{R_{i_1}\cdots R_{i_k}(M)\in\cF_\theta^G: k\in\N_0,\,
	i_1,\dots,i_k\in I\},\\
	&\cX_\theta(M)=\{r_{i_1}\cdots r_{i_k}([M])\in\cX_\theta: k\in\N_0,\,
	i_1,\dots,i_k\in I\}.
\end{align*}

A tuple $M\in \cF_\theta^G$ \textbf{admits all reflections} if $a^N_{ij}\in \Z $ for all
$N\in\cF_\theta^G(M)$ and all $i,j\in I$.

For all $M\in\cF_\theta^G$ let $\NA(M)=\NA(M_1\oplus\cdots\oplus M_\theta)$.
Following the terminology in \cite{MR2734956} we say that
a Nichols algebra $\NA(M)$ is \textbf{decomposable} if there exists a totally
ordered set $L$ and a sequence $(W_l)_{l\in L}$ of finite-dimensional absolutely
simple $\N_0^{\theta}$-graded objects in $\ydG$ such that 
\[
	\NA(M)\simeq \otimes_{l\in L}\NA(W_l).
\]
In this case, the isomorphism classes of the $W_l$ and the $\Z^{\theta}$-degrees
are uniquely determined and hence one may define the set $\roots _+^{[M]}$ of
\textbf{positive roots} and the set $\roots^{[M]}$ of \textbf{roots} of
$[M]$:
\[
\roots _+^{[M]}=\{\deg W_l:l\in L\},
\quad
\roots ^{[M]}=\roots _+^{[M]}\cup -\roots _-^{[M]}.
\]

There are several results that imply the decomposability of a Nichols algebra.
For example, Kharchenko proved \cite[Thm. 2]{MR1763385} that $\NA(M)$ is
decomposable if $G$ is abelian and $\dim M_i=1$ for all $i\in I$. In
\cite{MR2734956} it is proved that if all finite tensor powers of
$M_1\oplus\cdots\oplus M_\theta$ are direct sums of absolutely simple objects in
$\ydG$ then $\NA(M)$ is decomposable. 

Suppose that $M$ admits all reflections. Then 
\[
\cC(M)=(I,\cX_\theta(M),(r_i)_{i\in I},(A^X)_{X\in\cX_\theta(M)})
\]
is a connected semi-Cartan graph and hence the groupoid $\cW (M)=\cW (\cC (M))$
is defined.

We stress that the above reflection theory works more generally
for tuples of simple Yetter-Drinfeld modules.
From \cite[Cor.~2.4 and Thm.~2.3]{MR2732989} one obtains the
following theorem.

\begin{thm} \label{thm:C(M)finiteCartan}
	Let $\theta\in\N$, let $G$ be a group and 
	let $M=(M_1,\dots ,M_\theta )$, where each $M_i$ is a simple 
	Yetter-Drinfeld module over $G$.  Assume that $M$ admits all reflections and
	that $\cW(M)$ is finite. Then $\NA(M)$ is
	decomposable and $\cC (M)$ is a finite Cartan graph.
\end{thm}

Clearly, the same theorem holds if one starts with a tuple of absolutely
	simple Yetter-Drinfeld modules.
	Since extension of the base field of a Nichols algebra is compatible with the
	grading and with taking coinvariants,
	any reflection of a tuple of absolutely simples is again a
tuple of absolutely simples.

\medskip
As in the case of Coxeter groups, on morphisms of Weyl groupoids one defines a
length function, see \cite[\S1]{MR2734956}. The following theorem is an analog
of a PBW-decomposition for the Nichols algebras $\NA(M)$ of tuples $M$ of
absolutely simple Yetter-Drinfeld modules. 

\begin{thm}{\cite[Thm. 2.6]{MR2732989}}
	\label{thm:HS}
	Let $\theta\geq2$ and $M\in\cF_\theta^G$. Suppose that $M$ admits all
	reflections and that $\cW (M)$ is finite. Let
  $w=\id_{[M]}s_{i_1}\cdots s_{i_l}$ be a reduced decomposition
	of a longest element of $\cup _{[N]\in\cX_\theta(M)}\hom([N],[M])$. Let
	$$\beta_m=\id_{[M]}s_{i_1}\cdots s_{i_{m-1}}\alpha_{i_m} $$
	for all $m\in \{1,\dots ,l\}$, where $l$ is the length of $w$.
	Then $\roots_+^{[M]}=\{\beta_1,\dots,\beta_l\}$ and
	$\beta_k\ne\beta_m$ for all $k,m\in\{1,\dots,l\}$ with $k\ne m$. There 
	exist finite-dimensional absolutely simple subobjects
	$M_{\beta_m}\subseteq\NA(M)$ in $\ydG$ of degree $\beta_m$ for all
	$m\in\{1,\dots,l\}$ with  $M_{\beta_m}\simeq
	R_{i_{m-1}}\cdots R_{i_2}R_{i_1}(M)_{i_m}$ in $\ydG$.
	Moreover, the
	multiplication map 
	\begin{equation*}
		\NA(M_{\beta_l})\otimes\cdots\NA(M_{\beta_2})\otimes\NA(M_{\beta_1})\to\NA(M)
	\end{equation*}
	is an isomorphism of $\N_0^\theta$-graded objects in $\ydG$. 
\end{thm}

In this theorem, as everywhere else, we write $N_i$ for the $i$-th entry of a
tuple $N\in \cF_\theta^G$ (here $N=R_{i_{m-1}}\cdots R_{i_2}R_{i_1}(M)$), where
$1\le i\le \theta $.

For all $\alpha=\sum_{i=1}^\theta
n_i\alpha_i\in\Z^\theta$, we write $t^\alpha$ for $t_1^{n_1}\cdots
t_\theta^{n_\theta}\in\Z[[t_1,\dots,t_\theta]]$. 
For any $\N_0^\theta$-graded object
$X=\oplus_{\alpha\in\N_0^\theta} X_\alpha$ in $\ydG$, the (multivariate) Hilbert
series of $X$ is 
\[
\sum_{\alpha\in\N_0^\theta}(\dim X_\alpha)t^\alpha\in\Z[[t_1,\dots,t_\theta]].
\]

The Yetter-Drinfeld modules $(\ad V)^k(W)\subseteq \NA (V\oplus W)$
for $k\in \N _0$ and $V,W\in \ydG $ can be computed as a certain
subobject of $V^{\otimes k}\otimes W$ using Lemma~\ref{lem:X_n} below and hence the
$M_{\beta_m}$ in Theorem~\ref{thm:HS} can be computed effectively. This allows us to compute
the Hilbert series of $\NA(M)$.

\begin{lem}{\cite[Thm.~1.1]{MR2732989}}
  \label{lem:X_n}
  Let $V$ and $W$ be Yetter-Drinfeld modules over a group. 
  Let $\varphi _0=0$, $X_0^{V,W}=W$, and 
  \begin{align*}
		\varphi_{m+1} &= \id-c_{V^{\otimes m}\otimes W,V}\,c_{V,V^{\otimes m}
	  \otimes W}+(\id\otimes\varphi_m)c_{1,2},\\
		X_{m+1}^{V,W} &= \varphi_{m+1}(V\otimes X_m^{V,W})\subseteq V^{\otimes (m+1)}\otimes W
  \end{align*}
  for all $m\geq 0$. Then $(\ad V)^n(W)\simeq X_n^{V,W}$ for all $n\in \N _0$.
\end{lem}

The following important fact on $(\ad M_i)^m(M_j)$ for any $\theta \in \N $,
$M\in \cF_\theta^G$, $m\in \N $, and
$i,j\in \{1,\dots ,\theta \}$ is also used heavily for explicit calculations.
It is a variant of \cite[Thm. 7.2(3)]{MR2734956}. 

\begin{thm}  \label{thm:admabsimple}
	Let $\theta \in \N $ and $M\in \cF_\theta^G$. Assume that $M$ admits all
	reflections and that $\cW (M)$ is finite. Let $i,j\in \{1,\dots ,\theta \}$
	with $i\ne j$. Then $(\ad M_i)^m(M_j)\in \ydG $ is absolutely simple for
	all $0\le m\le -a_{ij}^M$ and zero for all $m>-a_{ij}^M$.
\end{thm}

\section{Main results}
\label{section:main_results}

We will need $q$-numbers and $q$-factorials in rings in different contexts.
For any ring $R$, any $m\in \N_0$ and any $q\in R$ let
$$(m)_q=\sum _{i=0}^{m-1}q^i,\qquad (m)_q^!=\prod _{i=1}^m(i)_q.$$

Let $G$ be a group. For
all $M\in \cF _\theta ^G$ let 
\[
\supp M=\supp M_1\cup\cdots\cup\supp M_\theta.
\]
Let $\Ggen _\theta ^G$ denote the subclass of $\cF _\theta ^G$ consisting of all
tuples $M$ such that $G$ is generated by $\supp M$.

\begin{defn} 
	\label{def:braidindecomposable}
	Let $\theta \in \N $.  Then $M\in \cF _\theta ^G$ is
	called \textbf{braid-indecomposable}, if there exists no decomposition
	$M'\oplus M''$ of $\oplus _{i=1}^\theta M_i$ in $\ydG $
	such that $M',M''\ne 0$ and $(\id	-c^2)(M'\otimes M'')=0$.
\end{defn}

In this work we will attach a skeleton (a kind of decorated Dynkin diagram) to
some tuples in $\cF_\theta^G$. 

\begin{defn} 
	\label{def:skeleton}
	Let $\theta \in \N_{\geq2}$ and 
	$M=(M_1,\dots ,M_\theta )\in \cF ^G_\theta $.
	Let $A=(a_{ij})_{1\leq i,j\leq\theta}$ be the Cartan matrix of $M$.
	We say that \textbf{$M$ has a skeleton} if
	\begin{enumerate}
		\item for all $1\le i\le \theta $ there exist $s_i\in \supp M_i$
			and $\sigma _i\in \chg{G^{s_i}}$ such that $M_i\simeq M(s_i,\sigma _i)$,
			and
		\item for all $1\le i<j\le \theta $
			with $a_{ij}\ne 0$ at least one of $a_{ij}$, $a_{ji}$ is $-1$.
	\end{enumerate}
	In this case the \textbf{skeleton of} $M$ is a partially oriented
	partially labeled loopless graph with $\theta $ vertices
	satisfying the following properties.
	\begin{enumerate}
		\item For all $1\le i\le \theta $, the $i$-th vertex is symbolized by
			$|\supp M_i|=\dim M_i$ points. If $\dim M_i=1$, then the vertex is labeled
			by $\sigma _i(s_i)$.
    	If $\dim M_i=2$ and there is an additional restriction on
			$p=\sigma_i(s_i's_i^{-1})$, where $\supp M_i=\{s_i,s_i'\}$, then the
			$i$-th vertex is labeled by $(p)$.
			Otherwise there is no label.
		\item For all $i,j\in \{1,\dots ,\theta \}$ with $i\ne j$
			there are $a_{ij}a_{ji}$ edges between
			the $i$-th and $j$-th vertex. The edge is oriented towards $j$
			if and only if $a_{ij}=-1$, $a_{ji}<-1$.
		\item Let $1\le i<j\le \theta $ with $a_{ij}<0$.
			If $\supp M_i$ and $\supp M_j$ commute, then
			the connection between the $i$-th and $j$-th vertex consists of continuous
			lines. Otherwise the connection consists of dashed lines. The connection
			is labeled with $\sigma _i(s_j)\sigma _j(s_i)$ if $\dim M_i=1$ or $\dim
			M_j=1$, and otherwise it is not labeled.
	\end{enumerate}
\end{defn}

\begin{rem}
	Let $i\in \{1,\dots ,\theta \}$ with $\dim M_i=1$ in
	Definition~\ref{def:skeleton}. Since the Yetter-Drinfeld modules $M_j$ are
	absolutely simple for all $j$, the support of each $M_j$ is a conjugacy class
	of $G$ and the central element $s_i$ acts by a scalar on
	each $M_j$. Thus $\sigma _i(s_j)$ and $\sigma _j(s_i)$ do not depend on the
	choice of $s_j\in \supp M_j$.
	
	We will show in Lemma~\ref{le:welldefp} that the label $(p)$ of a vertex with two
	points in Definition~\ref{def:skeleton} is well-defined.
	Therefore all labels of the skeleton of $M$ are well-defined.
\end{rem}


\begin{defn}
	A skeleton is called \textbf{simply-laced} if any two vertices are connected
	by at most one edge. A skeleton is called \textbf{connected} if the underlying
	graph is connected.  A connected skeleton with at least three vertices is
	said to be of \textbf{finite type} if it appears in Figure \ref{fig:fdNA}.
	For technical reasons we say that a skeleton of type $\alpha_2$ is of finite
	type.
\end{defn}

The main result of this paper is the following theorem.

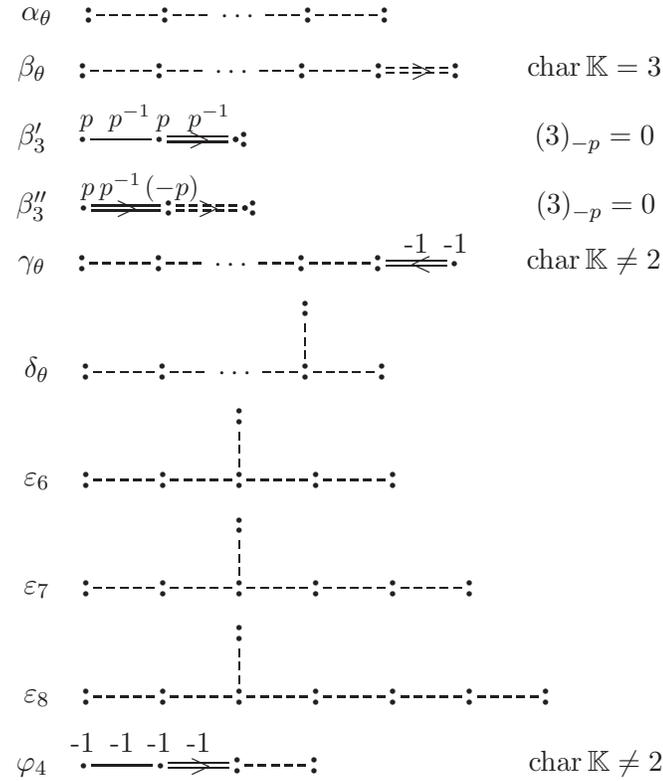
\begin{figure}
	\raisebox{8pt}{$\alpha _\theta $}
	\begin{picture}(200,20)
	\DDvertextwo{10}{8}
    \multiput(13,10)(5,0){5}{\line(1,0){3}}
	\DDvertextwo{39}{8}
    \multiput(42,10)(5,0){3}{\line(1,0){3}}
    \put(60,7){$\cdots$}
    \multiput(77,10)(5,0){3}{\line(1,0){3}}
	\DDvertextwo{93}{8}
    \multiput(96,10)(5,0){5}{\line(1,0){3}}
	\DDvertextwo{122}{8}
  \end{picture}\\
	\raisebox{8pt}{$\beta _\theta $}
  \begin{picture}(200,20)
	\DDvertextwo{10}{8}
    \multiput(13,10)(5,0){5}{\line(1,0){3}}
	\DDvertextwo{39}{8}
    \multiput(42,10)(5,0){3}{\line(1,0){3}}
    \put(60,7){$\cdots$}
    \multiput(77,10)(5,0){3}{\line(1,0){3}}
	\DDvertextwo{93}{8}
    \multiput(96,10)(5,0){5}{\line(1,0){3}}
	\DDvertextwo{122}{8}
    \multiput(125,11)(5,0){5}{\line(1,0){3}}
    \multiput(125,9)(5,0){5}{\line(1,0){3}}
		\put(134,7){$>$}
	\DDvertextwo{151}{8}
  \end{picture}
	\makebox[0pt]{\raisebox{8pt}{$\charK =3$}}\\
	\raisebox{8pt}{$\beta '_3$}
  \begin{picture}(200,25)
	  \DDvertexone{10}{10}
    \put(13,10){\line(1,0){23}}
    \put(9,15){\small $p$}
	  \put(20,15){\small $p^{-1}$}
	  \DDvertexone{39}{10}
    \put(38,15){\small $p$}
    \put(42,11){\line(1,0){23}}
    \put(42,9){\line(1,0){23}}
	  \put(50,15){\small $p^{-1}$}
    \put(50,7){$>$}
	  \DDvertexthree{68}{10}
  \end{picture}
	\makebox[0pt]{\raisebox{8pt}{$(3)_{-p}=0$}}\\
	\raisebox{8pt}{$\beta ''_3$}
  \begin{picture}(200,25)
	  \DDvertexone{10}{10}
    \put(13,11){\line(1,0){26}}
    \put(13,9){\line(1,0){26}}
    \put(9,15){\small $p$}
	  \put(16,15){\small $p^{-1}$}
    \put(22,7){$>$}
    \put(33,15){\small $(-p)$}
	  \DDvertextwo{42}{8}
    \multiput(45,11)(5,0)5{\line(1,0)3}
    \multiput(45,9)(5,0)5{\line(1,0)3}
    \put(53,7){$>$}
	  \DDvertexthree{71}{10}
  \end{picture}
	\makebox[0pt]{\raisebox{8pt}{$(3)_{-p}=0$}}\\
	\raisebox{8pt}{$\gamma _\theta $}
  \begin{picture}(200,20)
	\DDvertextwo{10}{8}
    \multiput(13,10)(5,0){5}{\line(1,0){3}}
	\DDvertextwo{39}{8}
    \multiput(42,10)(5,0){3}{\line(1,0){3}}
    \put(60,7){$\cdots$}
    \multiput(77,10)(5,0){3}{\line(1,0){3}}
	\DDvertextwo{93}{8}
    \multiput(96,10)(5,0){5}{\line(1,0){3}}
	\DDvertextwo{122}{8}
    \put(125,9){\line(1,0){23}}
    \put(125,11){\line(1,0){23}}
		\put(134,7){$<$}
    \put(132,14){-1}
	  \DDvertexone{151}{10}
    \put(147,14){-1}
  \end{picture}
	\makebox[0pt]{\raisebox{8pt}{$\charK \ne 2$}}\\
	\raisebox{8pt}{$\delta _\theta $}
  \begin{picture}(200,40)
	  \DDvertextwo{10}8
    \multiput(13,10)(5,0){5}{\line(1,0){3}}
	  \DDvertextwo{39}8
    \multiput(42,10)(5,0){3}{\line(1,0){3}}
    \put(60,7){$\cdots$}
    \multiput(77,10)(5,0){3}{\line(1,0){3}}
	  \DDvertextwo{93}8
    \multiput(93,15)(0,5)3{\line(0,1){3}}
		\DDvertextwo{93}{32}
    \multiput(96,10)(5,0)5{\line(1,0)3}
	  \DDvertextwo{122}8
  \end{picture}\\
	\raisebox{8pt}{$\varepsilon _6$}
  \begin{picture}(200,40)
	\DDvertextwo{10}{8}
    \multiput(13,10)(5,0){5}{\line(1,0){3}}
	\DDvertextwo{39}{8}
    \multiput(42,10)(5,0){5}{\line(1,0){3}}
	\DDvertextwo{68}{8}
    \multiput(68,15)(0,5){3}{\line(0,1){3}}
	\DDvertextwo{68}{32}
    \multiput(71,10)(5,0){5}{\line(1,0){3}}
	\DDvertextwo{97}{8}
    \multiput(100,10)(5,0){5}{\line(1,0){3}}
	\DDvertextwo{126}{8}
  \end{picture}\\
	\raisebox{8pt}{$\varepsilon _7$}
  \begin{picture}(200,40)
	\DDvertextwo{10}{8}
    \multiput(13,10)(5,0){5}{\line(1,0){3}}
	\DDvertextwo{39}{8}
    \multiput(42,10)(5,0){5}{\line(1,0){3}}
	\DDvertextwo{68}{8}
    \multiput(68,15)(0,5){3}{\line(0,1){3}}
	\DDvertextwo{68}{32}
    \multiput(71,10)(5,0){5}{\line(1,0){3}}
	\DDvertextwo{97}{8}
    \multiput(100,10)(5,0){5}{\line(1,0){3}}
	\DDvertextwo{126}{8}
    \multiput(129,10)(5,0){5}{\line(1,0){3}}
	\DDvertextwo{155}{8}
  \end{picture}\\
	\raisebox{8pt}{$\varepsilon _8$}
  \begin{picture}(200,40)
	  \DDvertextwo{10}{8}
    \multiput(13,10)(5,0){5}{\line(1,0){3}}
	  \DDvertextwo{39}{8}
    \multiput(42,10)(5,0){5}{\line(1,0){3}}
	  \DDvertextwo{68}{8}
    \multiput(68,15)(0,5){3}{\line(0,1){3}}
	  \DDvertextwo{68}{32}
    \multiput(71,10)(5,0){5}{\line(1,0){3}}
	  \DDvertextwo{97}{8}
    \multiput(100,10)(5,0){5}{\line(1,0){3}}
	  \DDvertextwo{126}{8}
    \multiput(129,10)(5,0){5}{\line(1,0){3}}
  	\DDvertextwo{155}{8}
    \multiput(158,10)(5,0){5}{\line(1,0){3}}
  	\DDvertextwo{184}{8}
  \end{picture}\\
	\raisebox{8pt}{$\varphi _4$}
  \begin{picture}(200,25)
		\put(5,15){-1}
	  \DDvertexone{10}{10}
    \put(13,10){\line(1,0){23}}
		\put(20,15){-1}
		\put(34,15){-1}
	  \DDvertexone{39}{10}
    \put(42,11){\line(1,0){23}}
    \put(42,9){\line(1,0){23}}
    \put(50,7){$>$}
		\put(49,15){-1}
	  \DDvertextwo{68}{8}
    \multiput(71,10)(5,0){5}{\line(1,0){3}}
	  \DDvertextwo{97}{8}
  \end{picture}
  \makebox[0pt]{\raisebox{8pt}{$\charK \ne 2$}}\\
  \caption{Skeletons of finite type with at least three vertices.}
  \label{fig:fdNA}
\end{figure}

\begin{thm}
	\label{thm:main}
	Let $\theta\in\N_{\geq3}$.  Let $G$ be a non-abelian group and 
	$M$ in $\Ggen _\theta ^G$. Assume that $M$ is braid-indecomposable.  The
	following are equivalent:
	\begin{enumerate}
		\item\label{it:main:skeleton} $M$ has a skeleton of finite type. 
		\item\label{it:main:Nichols} $\NA(M)$ is finite-dimensional.
		\item\label{it:main:groupoid} $M$ admits all reflections and the Weyl
			groupoid $\cW(M)$ of $M$ is finite.
	\end{enumerate}
\end{thm}

We record that the third property of $M$ in the theorem
is also equivalent to the finiteness of
the set of $\N _0$-graded right coideal subalgebras of $\NA (M)$ by
\cite[Thm.\,6.15]{MR3096611}.

Theorem \ref{thm:main} will be proved in Section \ref{section:main}.  The
Hilbert series of the Nichols algebras of Theorem \ref{thm:main} will be given
in Subsections \ref{subsection:ADE}, \ref{subsection:C}, \ref{subsection:B} and
\ref{subsection:F4}. In Sections~\ref{section:ADE}, \ref{section:C},
\ref{section:B}, and \ref{section:F4} we give a description of all tuples in
$\cF_\theta^G$ which
have a skeleton of finite type.

\subsection{The $ADE$ series}
\label{subsection:ADE}

The following theorem will be proved in Section \ref{section:ADE}.

\begin{thm} 
	\label{thm:ADE}
	Let $\theta\in\N_{\geq2}$. Let $G$ be a non-abelian group and  $M\in \Ggen
	_\theta ^G$.  Assume that the
	Cartan matrix $A^M$ is of finite type and the Dynkin diagram of $A^M$ is
	connected and simply-laced. Then the following hold:
	\begin{enumerate}
		\item \label{it:ADE:skeleton} $M$ has a simply-laced skeleton of finite type.
		\item \label{it:ADE:groupoid} $M$ admits all reflections and the Weyl groupoid $\cW (M)$ is
			finite with a finite root system of standard type $A_\theta$ with 
			$\theta\geq2$, $D_\theta$ with $\theta\geq4$, $E_6$, $E_7$ or $E_8$.
		\item \label{it:ADE:Nichols} $\NA(M)$ is finite-dimensional and its
			Hilbert series is 
			\[
			\cH(t)=\prod_{\alpha\in\mathbf{\Delta}_+}(1+t^\alpha)^2,
			\]
			where $\mathbf{\Delta}_+$ denotes the set of positive roots of the root system
			associated with the Cartan matrix $A^M$. The dimensions of these Nichols algebras
			are listed in Table~\ref{tab:standard}.
	\end{enumerate}
\end{thm}

\subsection{The $C$ series}
\label{subsection:C}

The following theorem will be proved in
Section~\ref{section:C}.

\begin{thm}
	\label{thm:C}
	Let $\theta\in\N_{\geq3}$, $G$ be a non-abelian
	group and $M\in \Ggen _\theta ^G$.  Assume that the Cartan matrix $A^M$ is of type
	$C_\theta$.  Then the following are equivalent:
	\begin{enumerate}
		\item \label{it:C:skeleton} $\charK\ne 2$ and $M$ has a skeleton of type $\gamma_\theta$.
		\item \label{it:C:fgroupoid} $M$ admits all reflections and the Weyl groupoid $\cW (M)$ is
			finite.
		\item \label{it:C:sgroupoid} $M$ admits all reflections and the Weyl groupoid $\cW (M)$ is
			standard.
		\item \label{it:C:NA} $\NA(M)$ is finite-dimensional.
	\end{enumerate}
	In this case, the Hilbert series of $\NA(M)$ is 
	\[
	\cH(t)=
	\prod_{\alpha\in\sroots}
	(1+t^{\alpha})^2\prod_{\alpha\in\lroots}(1+t^\alpha),
	\]
	where $\sroots$ and $\lroots$ denote the set of short positive roots
	and long positive roots of the root system associated with $\cW(M)$,
	respectively. In particular 
	\[
	\dim\NA(M)=2^{2\theta^2-\theta}.
	\]
\end{thm}

\subsection{The $B$ series}
\label{subsection:B}

Our main results in this subsection are the following two theorems.

\begin{thm}
	\label{thm:B1}
	Let $\theta\in\N_{\geq3}$. Let $G$ be a non-abelian group and  $M\in \Ggen
	_\theta ^G$.  Assume that $\dim M_1=1$ and that the Cartan matrix
	$A^M$ is of type $B_\theta$.  Then the following are
	equivalent:
	\begin{enumerate}
		\item \label{it:B1:skeleton} $\theta =3$ and $M$ has a skeleton of type $\beta '_3$.
		\item \label{it:B1:groupoid} $M$ admits all reflections and the Weyl groupoid $\cW (M)$ is
			finite.
		\item \label{it:B1:NA} $\NA(M)$ is finite-dimensional.
	\end{enumerate}
	Let $h=3$ if $\charK=2$, $h=2$ if $\charK=3$, and $h=6$ otherwise, and let
	$h'=2$ if $\charK=3$ and $h'=6$ otherwise.
	Then in the above cases the Hilbert series of $\NA(M)$ is 
 	\[
		\cH(t)=\prod _{\alpha \in\cO_1}(h)_{t^{\alpha}}\prod _{\alpha\in
		\cO_3}(2)^2_{t^{\alpha}}(3)_{t^{\alpha}}
		\prod _{\alpha \in\cO_{233}}(2)_{t^{\alpha}}(h')_{t^{\alpha}},
  \]
    where $\cO_1,\cO_3$, and $\cO_{233}$ are the sets of positive roots in the
    orbits of $\alpha_1$, $\alpha_3$, and $\alpha_2+2\alpha_3$, respectively,
    under the action of the automorphism group of the skeleton of $M$ in its
	Cartan graph, see Lemma~\ref{lem:CGtwopoints}.  In particular,
	\[
	\dim\NA(M)=h^612^4(2h')^3.
	\]
\end{thm}

\begin{thm}
	\label{thm:B2}
	Let $\theta\in\N_{\geq3}$. Let $G$ be a non-abelian group and  $M\in \Ggen
	_\theta ^G$.  Assume that $\dim M_1>1$, and
	that the Cartan matrix $A^M$ is of type $B_\theta$.  Then the following are
	equivalent:
	\begin{enumerate}
		\item \label{it:B2:skeleton} $\charK=3$ and $M$ has a skeleton of type $\beta_\theta$.
		\item \label{it:B2:groupoid} $M$ admits all reflections and the Weyl groupoid
			$\cW (M)$ is finite.
		\item \label{it:B2:sgroupoid} $M$ admits all reflections and the Weyl groupoid $\cW (M)$ is
			standard.
		\item \label{it:B2:NA} $\NA(M)$ is finite-dimensional.
			\end{enumerate}
		In this case the Hilbert series of $\NA(M)$ is 
  	\[
    	\cH(t)=\prod_{\alpha\in\sroots}
			(1+t^{\alpha}+t^{2\alpha })^2\prod_{\alpha\in\lroots}(1+t^\alpha)^2,
	  \]
	  where $\sroots$ and $\lroots$ denote the set of short positive roots
	  and long positive roots of the root system associated with $\cW(M)$,
	  respectively. In particular 
	  \[
			\dim\NA(M)=2^{2\theta (\theta -1)}3^{2\theta }.
	  \]
\end{thm}

Theorems \ref{thm:B1} and
\ref{thm:B2} will be proved in Section \ref{section:B}.

\subsection{The exceptional case $F_4$}
\label{subsection:F4}

%

The following theorem will be proved in Section~\ref{section:F4}.

\begin{thm}
	\label{thm:F4}
	Let $G$ be a non-abelian group and let $M\in \Ggen _\theta ^G$.
	Assume that the
	Cartan matrix $A^M$ is of type $F_4$.  Then the following are equivalent:
	\begin{enumerate}
		\item \label{it:F4:skeleton} $\charK\ne2$ and $M$ has a skeleton of type $\varphi_4$.
        \item \label{it:F4:fgroupoid} $M$ admits all reflections and the Weyl
            groupoid $\cW (M)$ is finite.
        \item \label{it:F4:sgroupoid} $M$ admits all reflections and the Weyl
            groupoid $\cW (M)$ is standard.
		\item \label{it:F4:Nichols} $\NA(M)$ is finite-dimensional.
	\end{enumerate}
	In this case,
	the	Hilbert series of $\NA (M)$ is
			\[
			\cH(t)=(2)_{t}^6(2)_{t^2}^5(2)_{t^3}^5(2)_{t^4}^5(2)_{t^5}^4(2)_{t^6}^3(2)_{t^7}^3(2)_{t^8}^2(2)_{t^9}(2)_{t^{10}}(2)_{t^{11}}.
			\]
			In particular $\dim\NA(M)=2^{36}$.
\end{thm}

\section{Finite Cartan graphs of rank three}
\label{section:Cartan_rank3}

In this section we collect some facts about finite Cartan graphs of rank three
which will be used for our classification. Our main reference is
\cite{MR2869179}.

\begin{lem} \label{lem:noA}
  Let $\cC =\cC (I,\cX ,r,A)$
  be a connected indecomposable finite Cartan graph with $|I|=3$.
	If $A^X$ is not of type $A_3$ for all $X\in \cX $, then up to a permutation of
	$I$ one of the following holds.
	\begin{enumerate}
		\item $\cC $ is standard of type $C_3$.
		\item $\cC $ is standard of type $B_3$.
		\item \label{it:noA:3} For each point $X$ of $\cC $, $A^X$ is one of the matrices
			\[ \begin{pmatrix} 2 & -1 & 0\\-1 & 2 & -2\\0 & -1 & 2\end{pmatrix},\quad
			\begin{pmatrix} 2 & -1 & 0\\-2 & 2 & -2\\0 & -1 & 2\end{pmatrix}. \]
		\item \label{it:noA:4} For each point $X$ of $\cC $, $A^X$ is one of the matrices
			\[ \begin{pmatrix} 2 & -1 & 0\\-1 & 2 & -1\\0 & -2 & 2\end{pmatrix},\quad
			\begin{pmatrix} 2 & -1 & 0\\-2 & 2 & -1\\0 & -2 & 2\end{pmatrix}. \]
		\item For each point $X$ of $\cC $, $A^X$ is one of the matrices
			\begin{align*}
			  &\begin{pmatrix} 2 & -1 & 0\\-1 & 2 & -1\\0 & -2 & 2\end{pmatrix},\quad
			  \begin{pmatrix} 2 & -1 & 0\\-1 & 2 & -1\\0 & -4 & 2\end{pmatrix},\quad
			  \begin{pmatrix} 2 & -1 & 0\\-1 & 2 & -2\\0 & -2 & 2\end{pmatrix},\\
			  &\begin{pmatrix} 2 & -1 & -1\\-1 & 2 & -1\\-1 & -2 & 2\end{pmatrix},\quad
			  \begin{pmatrix} 2 & 0 & -1\\0 & 2 & -1\\-1 & -2 & 2\end{pmatrix},\quad
				\begin{pmatrix} 2 & 0 & -1\\0 & 2 & -1\\-1 & -3 & 2\end{pmatrix}.
			\end{align*}
		\item For each point $X$ of $\cC $, $A^X$ is one of the matrices
			\begin{align*}
			  &\begin{pmatrix} 2 & -1 & 0\\-1 & 2 & -1\\0 & -2 & 2\end{pmatrix},\quad
			  \begin{pmatrix} 2 & -1 & 0\\-1 & 2 & -1\\0 & -3 & 2\end{pmatrix},\quad
			  \begin{pmatrix} 2 & -1 & 0\\-1 & 2 & -2\\0 & -2 & 2\end{pmatrix},\\
			  &\begin{pmatrix} 2 & -1 & 0\\-1 & 2 & -2\\0 & -1 & 2\end{pmatrix},\quad
			  \begin{pmatrix} 2 & -1 & 0\\-2 & 2 & -3\\0 & -1 & 2\end{pmatrix},\quad
			  \begin{pmatrix} 2 & -1 & 0\\-2 & 2 & -2\\0 & -1 & 2\end{pmatrix}.
			\end{align*}
	\end{enumerate}
	The six cases correspond to the set of positive roots in \cite[Appendix A]{MR2869179}
	with number $3$, $4$, $13$, $14$, $25$, and $28$, respectively.
\end{lem}

\begin{rem}
	The Cartan graphs in cases Lemma~\ref{lem:noA}(3),(4) also appeared
	in \cite[Thm.~5.4]{MR2498801}.	
\end{rem}

\begin{proof}
	Consider the list of all possible sets of positive roots in
	\cite[Appendix A]{MR2869179}.
	There are precisely $55$ such sets up to permutation of $I$ and up to a choice
	of a point of $\cC $.
	By \cite[Cor.\,2.9]{MR2869179}, the Cartan matrix of the point $X$ can be obtained from the
	set $\roots ^X_+$ of its positive roots: $\alpha_j+m\alpha_i\in \roots ^X$ for
		$m\in \Z $, $i,j\in I$ with $i\ne j$, if and only if $0\le m\le -a_{ij}^X$.
	Since the reflection $s_i^X$
	for $i\in I$ maps $\roots ^X_+\setminus \{\alpha _i\}$ bijectively to
	$\roots _+^{r_i(X)}\setminus \{\alpha _i\}$,
	one can calculate the Cartan matrices and the sets
	of positive roots in all points of $\cC $. The elementary calculations are done most
	efficiently by a computer program.
\end{proof}

For later reference we extract two easy corollaries of the lemma.

\begin{cor}
	\label{cor:columns}
  Let $\cC =\cC (I,\cX ,r,A)$
  be a connected indecomposable finite Cartan graph with $|I|=3$.
	If $A^X$ is not of type $A_3$ for all $X\in \cX $, then for all $X\in \cX $
	and for all columns of $A^X$ there is at most one entry which is
	strictly smaller than $-1$.
\end{cor}

\begin{rem} \label{rem:columns}
  The claim in Corollary~\ref{cor:columns} holds without the
	assumption in the second sentence, but we will not need this.
\end{rem}

\begin{cor}
	\label{cor:noAC}
  Let $\cC =\cC (I,\cX ,r,A)$
  be a connected indecomposable finite Cartan graph with $|I|=3$.
	If $A^X$ is not of type $A_3$ and not of type $C_3$ for all $X\in \cX $,
	then either $\cC $ is standard of type $B_3$ or
  there is a permutation of $I$ such that for all points $X$ the Cartan matrix
	$A^X$ is one of the matrices in Lemma~\ref{lem:noA}(4).
\end{cor}

\section{Cartan matrices of finite type}
\label{section:Cartan}

Recall from \cite[Thm.\,4.3]{MR823672} the classification of a class of
indecomposable real matrices. One says that a matrix $A=(a_{ij})_{i,j\in \{1,\dots ,n\}}$ is
\textbf{indecomposable}, if there are no proper subsets $I,J$ of $\{1,\dots ,n\}$
such that $I\cap J=\emptyset $, $I\cup J=\{1,\dots ,n\}$, and $a_{ij}=a_{ji}=0$
for all $i\in I$, $j\in J$. For $x,y\in \R^n$ we write $x>y$ ($x\ge y$,
respectively) if $x-y$ has only positive (non-negative, respectively) entries.

\begin{thm} \label{thm:CMclassification}
  Let $n\in \N $ and let $A$ be an indecomposable real $n\times n$-matrix such that
  $a_{ij}\le 0$ for all $i,j\in \{1,2,\dots
  ,n\}$ with $i\ne j$, and $a_{ij}=0$ whenever $a_{ji}=0$. Then $A$
  has precisely one of the following properties.
  \begin{enumerate}
    \item[(Fin)] $\det A\ne 0$; there exists $u>0$ such that $Au>0$;
      $Av\ge 0$ implies $v>0$ or $v=0$.
    \item[(Aff)] $\mathrm{corank}\,A=1$; there exists $u>0$ such that $Au=0$;
      $Av\ge 0$ implies that $Av=0$.
    \item[(Ind)] There exists $u>0$ such that $Au<0$; $Av\ge 0$, $v\ge 0$
      imply that $v=0$.
  \end{enumerate}
  Then $A$ is called of finite, affine, and indefinite type, respectively.
  Moreover, $A^t$ has the same type as $A$.
\end{thm}

Now we apply this theorem in order to prove that any connected indecomposable
finite Cartan graph has a point with a Cartan matrix
of finite type. The classification of indecomposable Cartan matrices of finite
type is well-known and can be found for example in \cite{MR823672}.

\begin{thm}
  \label{thm:finitetype}
  Let $\cC =\cC (I,\cX ,r,A)$ be a connected
	indecomposable finite Cartan graph.
	Then there exists $X\in \cX $ such that $A^X$ is of finite type.
\end{thm}

\begin{proof}
  The indecomposability of $\cC $ implies that $A^X$ is indecomposable for
  all $X\in \cX $, see \cite[Prop.~4.6]{MR2498801}.
  We give an indirect proof of the theorem.  So assume that
  for all $X\in \cX $ the Cartan matrix $A^X$ is of affine or indefinite type.

  Since $\cC $ is finite and connected,
	$\cX $ is a finite set and $\rroots X$ is
  finite for all $X\in \cX $.  Among all real roots of $\cC $ in all objects,
  choose $\alpha =\sum _{i\in I}x_i\alpha _i$ which is maximal with respect to
  $>$. Let $x=(x_i)_{i\in I}$ and let $X\in \cX $ be such that
  $\alpha \in \rroots X$. 
	Let
	$$B=\{s_{j_1}\cdots s_{j_k}^X(\alpha )\,|\,k\ge 0,\,j_1,\dots ,j_k\in
  I\}.$$
  Observe that $s_j(\alpha )=\alpha -\sum _{i\in I}a_{ji}x_i\alpha _j$ for all
  $j\in I$. Thus the maximality of $\alpha $ implies that $Ax\ge 0$.
  Since $x\ge 0$ and $x\ne 0$, $A$ is not of indefinite type.
  Then $A$ is of affine type and $Ax=0$. Consequently, $s_j^X(\alpha )=\alpha $
	for all $j\in I$. Since $\alpha $ is maximal,
		by induction on $k$ we conclude that $s_{j_1}\cdots
	s_{j_k}^X(\alpha )=\alpha $ for any $k\in \N_0$ and $j_1,\dots ,j_k\in I$.
  Therefore $B=\{\alpha \}$. On the other hand, $\alpha $ is a real root which
	implies that $s_{i_1}\cdots s_{i_k}^X(\alpha )=\alpha_i$ for some $k\in \N_0$,
	$i_1,\dots ,i_k\in I$, and then
	$s_is_{i_1}\cdots s_{i_k}^X(\alpha)=-\alpha_i\ne \alpha $. This is clearly a
contradiction.
\end{proof}

\section{Auxiliary lemmas}
\label{section:helpful}

In this section, let $G$ be a group.

\subsection{}

We first extend results of \cite{MR3276225}. We start with considerations in a
general setting. 

\begin{lem}
	\label{lem:rs}
	Let $s\in G$.  Assume that $|s^G|=2$. Let $r,\epsilon
	\in G$ be such that $rs=\epsilon sr$, $\epsilon \ne 1$.
    Then the following hold:
	\begin{enumerate}
		\item $s^G=\{s,\epsilon s\}$, $r\epsilon =\epsilon ^{-1}r$,
			and $g\epsilon =\epsilon g$, $g\epsilon s =\epsilon sg$ for all $g\in G^s$.
		\item $r^{-1}sr=rsr^{-1}=\epsilon s$ and
			$r^2,r^{-1}gr,rgr^{-1}\in G^s$ for all $g\in G^s$.
		\item $(\epsilon ^m s^n)^G=\{\epsilon ^m s^n,\epsilon ^{n-m} s^n\}$
			for all $m,n\in \Z $.
		\item Let $H$ be a subgroup of $G$ containing $r$ and $s$. Then $H$ is
				generated by $(H\cap G^s)\cup \{r\}$.
	\end{enumerate}
\end{lem}

\begin{proof}
	Since $rsr^{-1}=\epsilon s$ and $|s^G|=2$, we conclude that $s^G=\{s,\epsilon
	s\}$. Then $r\epsilon sr^{-1}=s$ and therefore $r\epsilon r^{-1}=\epsilon
	^{-1}$. Moreover, $s^G=\{s,\epsilon s\}$ implies that $g\epsilon
	sg^{-1}=\epsilon s$ for all $g\in G^s$ and hence $g\epsilon =\epsilon g$ for
	all $g\in G^s$. In particular, (1) is proven.

	(2) and (3) follow by similar arguments.

	(4) Since $|s^G|=2$, $G^s$ has index $2$ in $G$. Therefore $H\cap G^s$
	has index at most $2$ in $H$. Since $r\in H\setminus G^s$, we conclude the
claim.
\end{proof}

\begin{lem}
	\label{lem:A3:222}
	Let $r,s,\epsilon\in G$. Assume that
	$|r^G|=|s^G|=2$,
	$rs=\epsilon sr$, and
	$\epsilon \ne 1$.
	Then the following hold:
	\begin{enumerate}
		\item $r^G=\{r,\epsilon r\}$, $s^G=\{s,\epsilon s\}$, $\epsilon ^2=1$
			and $\epsilon \in Z(G)$.
		\item Let $t\in G$. Assume that $|t^G|=2$, $rt=tr$, and $st\ne ts$.
			Then $t^G=\{t,\epsilon t\}$ and $st=\epsilon ts$. 
	\end{enumerate}
\end{lem}

\begin{proof}
  (1) Lemma~\ref{lem:rs}(1) implies that $s^G=\{s,\epsilon s\}$,
	$r^G=\{r,\epsilon ^{-1}r\}$,
	$G^r$ and $G^s$ commute with $\epsilon $, and
	$r\epsilon =\epsilon ^{-1}r$.
	Thus $\epsilon ^2=1$. Since $G^s$ and $r$ generate $G$, we conclude
	 that $\epsilon \in Z(G)$.

	 (2)
	 Since $s^G=\{s,\epsilon s\}$ by (1) and since $st\ne ts$, we obtain that
	 $ts=\epsilon st$. Thus (1) with $r=t$ implies that $t^G=\{t,\epsilon t\}$ and
	 $st=\epsilon ts $.
\end{proof}

We shall also need the following lemmas. 

\begin{lem} \label{le:welldefp}
	Let $s_1,s_2\in G$ be such that $s_1\ne s_2$, and let $V\in \ydG $.
	Assume that $\dim V=2$ and that $\supp V=s_1^G=\{s_1,s_2\}$.
	Then there exist unique $p_1,p_2\in \K\setminus \{0\}$ such that
		$s_1v=p_1v$ and $s_2v=p_2v$ for all $v\in V_{s_1}$. Moreover,
	$$ s_2s_1^{-1}v=p_2p_1^{-1}v,\, s_1w=p_2w,\,
	s_2w=p_1w,\, s_1s_2^{-1}w=p_2p_1^{-1}w $$
		 for all $v\in V_{s_1}$, $w\in V_{s_2}$.
\end{lem}

\begin{proof}
  Since $s_1^G=\{s_1,s_2\}$ by assumption, there exists $r\in G$ such that
	$rs_1=s_2r$ and $rs_2=s_1r$.
	Moreover, $p_1,p_2$ exist since $s_1s_2=s_2s_1$ by Lemma~\ref{lem:rs} and
	since $\dim V_{s_i}=1$ for all $i\in \{1,2\}$.
	Then $s_1rv=rs_2v=p_2rv$ and $s_2rv=rs_1v=p_1rv$ for all $v\in V_{s_1}$.
	This implies the claim since $V_{s_2}=rV_{s_1}$.
\end{proof}

\begin{lem} 
  \label{lem:a=0}
  Let $V,W\in \ydG $ be non-zero Yetter-Drinfeld modules such that $(\id
  -c_{W,V}c_{V,W})(V\otimes W)=0$.  Then $\supp V$ and $\supp W$ commute, and
	for any $s\in \supp V$, $t\in \supp W$ there exists
	$\lambda _{st}\in \K\setminus \{0\}$ such that $sw=\lambda _{st}w$
	for all $w\in W_t$ and $tv=\lambda _{st}^{-1}v$ for
  all $v\in V_s$.
\end{lem}

\begin{proof}
	Let $s\in \supp V$, $t\in \supp W$, $v\in V_s\setminus \{0\}$, and $w\in
	W_t\setminus \{0\}$.  Then
	$$
	(\id -c_{W,V}c_{V,W})(v\otimes w)=v\otimes w-sts^{-1}v\otimes sw.
	$$
	Since $sw\in W_{sts^{-1}}$, the latter is zero if and only if $st=ts$ and
	$sw=\lambda _{st}w$, $tv=\lambda_{st}^{-1}v$
	for some $\lambda _{st}\in \K\setminus \{0\}$.
	These conditions are independent of the choice of $v$ and $w$, and therefore
	the lemma follows.
\end{proof}

We will also need a stronger claim in a more specific context.

\begin{lem} 
  \label{lem:a=0for22}
	Let $s,t,\epsilon \in G$, $\sigma \in \widehat{G^s}$, $\tau \in
	\widehat{G^t}$, and let $V,W\in \ydG $. Assume that $\epsilon \ne 1$,
	$s^G=\{s,\epsilon s\}$, $t^G=\{t,\epsilon t\}$,
	and $V\simeq M(s,\sigma )$, $W\simeq M(t,\tau )$. Then the
	following hold:
	\begin{enumerate}
		\item $\epsilon \in G^s\cup G^t$.
		\item If $G^s\ne G^t$ and $st=ts$ then $\sigma (\epsilon )=\tau(\epsilon )=1$.
		\item The following are equivalent:
			\begin{enumerate}
				\item $(\id -c_{W,V}c_{V,W})(V\otimes W)=0$,
				\item $st=ts$, 
			    $\sigma (t)\tau (s)=1$, and $\sigma (\epsilon )\tau (\epsilon )=1$.
		\end{enumerate}
	\end{enumerate}
\end{lem}

\begin{proof}
	Since $s^G=\{s,\epsilon s\}$ and $t^G=\{t,\epsilon t\}$,
	Lemma~\ref{lem:rs}(1) tells that $\epsilon \in G^s\cup G^t$.
	Note that $\epsilon $ is possibly not central if $s$ and $t$ commute.

	(2) Assume that $G^s\ne G^t$. Since both $G^s$ and $G^t$ have index two in
	$G$, there exists $r\in G^t$ with $rs=\epsilon sr$. If $st=ts$, then $s,
	\epsilon \in G^t$
	and hence $\tau (rs)=\tau(\epsilon )\tau (sr)$. Thus $\tau (\epsilon
	)=1$ and similarly $\sigma (\epsilon )=1$.

	(3) Let $v\in V_s\setminus \{0\}$, $w\in W_t\setminus \{0\}$, and let $r\in G$
  	be such that $rs=\epsilon sr$.
	Since $\K Gw=W$ and the braiding commutes with the action of $G$, we conclude that
	$(\id -c_{W,V}c_{V,W})(V\otimes W)=0$ if and only if
	$(\id -c_{W,V}c_{V,W})(v'\otimes w)=0$ for all $v'\in V_s\cup V_{\epsilon s}$.
	Since $V=\K v+\K rv$, by Lemma~\ref{lem:a=0} the latter claim is equivalent to
	\begin{align} \label{eq:three}
		st=ts,\quad v\otimes w=tv\otimes sw,\quad rv\otimes w=trv\otimes \epsilon sw.
	\end{align}
	The second equation in \eqref{eq:three} is equivalent to $\sigma (t)\tau (s)=1$.
	If $G^s=G^t$, then $r$ and $t$ do not commute. Hence $tr=r(\epsilon t)$, and
	the third equation in \eqref{eq:three} is equivalent to $\sigma (\epsilon
	t)\tau (\epsilon s)=1$.
	This implies (2).
	On the other hand, if $G^s\ne G^t$, then we may assume that $r\in G^t$. In
	that case the last equation in \eqref{eq:three} is equivalent to the second,
	and the last equation in (b) is a tautology because of (1). Thus again (2)
	holds.
\end{proof}

The following lemma is contained partially in
\cite[Lemmas 5.13, 5.15]{MR3276225}.

\begin{lem}
  \label{lem:noncommuting}
  Let $V,W\in \ydG $ be non-zero finite-dimensional objects
        such that $(\ad V)^2(W)=0$ in $\NA (V\oplus W)$.
  \begin{enumerate}
    \item If $(\ad V)(W)\not=0$ then $\supp V$ is commutative.
    \item Let $s\in \supp V$ and $t\in \supp W$. Assume that
      $(\id -c^2)(V_s\otimes W_t)\ne 0$,
			$st=ts$,
			and that
      there exists $\lambda \in \K $ such that $sw=\lambda w$
      for all $w\in W_t$.
      Then $G^t\subseteq G^s$.
    \item Let $s\in \supp V$ and $t\in \supp W$. Assume that
			$(\id -c^2)(V_s\otimes W_t)\ne 0$,
      $st=ts$,
			and that
      there exist $\lambda ,\lambda '\in \K $ such that
      $sw=\lambda w$ and $tv=\lambda 'v$
      for all $w\in W_t$, $v\in V_s$.
      Then $\dim V_s=1$.
    \item If $s\in \supp V$ and $t\in \supp W$ with $st\not=ts$,
      then $(\ad V)(W)\not=0$, $\varphi_t|_{\supp V}$ is 
      the transposition $(s\,\,t\trid s)$, $\dim V_s=1$, and
      $sv=-v$ for all $v\in V_s$.
  \end{enumerate}
\end{lem}

\begin{proof}
	(1) Let $s\in \supp V$, $t\in \supp W$ be such that $(\ad V_s)(W_t)\not=0$.
  Assume that $\supp V$ is not commutative. Since $\supp V$ is a union of
  conjugacy classes of $G$, there exists $r\in \supp V
  \setminus \{s,t^{-1}\trid s\}$ such that $rs\not=sr$.
  Then $(\ad V_r)(\ad V_s)(W_t)\not=0$ by \cite[Prop.\,5.5]{MR3276225},
  a contradiction to $(\ad V)^2(W)=0$.

        (2) 
  Let $u\in V_s$, $w\in W_t\setminus \{0\}$,
  and $\lambda \in \K ^\times $ such that $sw=\lambda w$.
  Then
  \[ (\id -c^2)(u\otimes w)=u\otimes w-tu\otimes sw=(u-\lambda tu)\otimes w.
  \]
  Thus, by assumption, there exists $v\in V_s$ such that $tv\ne \lambda^{-1}v$.

  Let $g\in G^t$, $s'=gsg^{-1}$, and $v'=gv$. Then $v'\in V_{s'}$ and $s't=ts'$.
  Moreover,
  \[ (\id -c^2)(v'\otimes w)=v'\otimes w-tv'\otimes gsg^{-1}w=g(v-\lambda tv)\otimes w
  \]
  and hence $(\id -c^2)(v'\otimes w)\ne 0$.
  Assume that $g\notin G^s$, that is, $s'\ne s$. Recall that $(\ad V)^2(W)\simeq X_2^{V,W}$
  in $\ydG $, and that $X_2^{V,W}=\varphi _2(\id \otimes \varphi _1)(V\otimes V\otimes W)$. Then
  \begin{align*}
    &\varphi _2(\id \otimes \varphi _1)(v'\otimes v\otimes w)\\
    &\quad =(\id +c_{12}-c_{23}^2c_{12}-c_{12}c_{23}^2c_{12})(v'\otimes (v-\lambda tv)\otimes w)\\
    &\quad =(\id -c_{12}c_{23}^2c_{12})(v'\otimes (v-\lambda tv)\otimes w)
      +s'(v-\lambda tv)\otimes (\id -c^2)(v'\otimes w).
  \end{align*}
  Since $s$ and $s'$ commute by (1),
  the first summand of the last expression is in $V_{s'}\otimes V_s\otimes W$, and the second is
  non-zero in $V_s\otimes V_{s'}\otimes W$. This is a contradiction to $(\ad V)^2(W)=0$.

  (3) Assume to the contrary that $v,v'\in V_s$ are linearly independent.
  By a computation similar to one in the proof of (2), we obtain that
  \begin{align*}
    (\id -c_{12}c_{23}^2c_{12})(v'\otimes (v-\lambda tv)\otimes w)
      +s(v-\lambda tv)\otimes (\id -c^2)(v'\otimes w)=0.
  \end{align*}
  Since $tv=\lambda 'v$ and $tv'=\lambda 'v'$, we conclude that
  $\lambda \lambda '\ne 1$ and
  \begin{align*}
    (1-\lambda \lambda ')(v'\otimes v-\lambda \lambda 'sv'\otimes sv
    +(1-\lambda \lambda ')sv\otimes v') \otimes w=0.
  \end{align*}
	Applying to the second tensor factor a functional $v'^*\in V_s^*$ with
	$v'^*(v)=0$, $v'^*(v')=1$, implies that $sv\in \K sv'$,
	which yields the desired contradiction.

  (4)
  Since $(\ad V_s)(W_t)\simeq (\id -c_{W_t,V_s}c_{V_s,W_t})(V_s\otimes W_t)$
  in $\ydG $ and $st\not=ts$, we conclude from
  \cite[Prop.\,5.5]{MR3276225} that $(\ad V_s)(W_t)\not=0$.
	If $\supp V=\{s,t\trid s\}$, then $\varphi _t|_{\supp V}=(s\,(t\trid s) )$.
	So assume that
  $|\supp V|\ge 3$.
  Let $r\in \supp V$ be such that $r\notin \{s,t^{-1}\trid s\}$.
  Since $(\ad V_r)(\ad V_s)(W_t)=0$ by assumption,
  \cite[Prop.\,5.5]{MR3276225} implies that $rt=tr$. Hence $\varphi _t|_{\supp V}
  =(s\,t^{-1}\trid s)$. This implies the claim on
  $\varphi_t|_{\supp V}$.

  Let now $v_1,v_2\in V_s$ and $w\in W_t$. Then
  \begin{align*}
    \varphi _2(\id \ot \varphi _1)(&v_1\ot v_2\ot w)=
    \varphi _2(v_1\ot v_2\ot w-v_1\ot sts^{-1}v_2\ot sw)\\
    =&\; (v_1\ot v_2+sv_2\ot v_1)\ot w\\
    &\; -(sv_2\ot sts^{-1}v_1 +s^2ts^{-1}v_1\ot sv_2)\ot sw\\
    &\; -(v_1\ot sts^{-1}v_2 +s^2ts^{-1}v_2\ot v_1)\ot sw\\
    &\; +(s^2ts^{-1}v_2\ot s^2ts^{-2}v_1 +s^2ts^{-1}v_1\ot s^2ts^{-1}v_2)\ot s^2w.
  \end{align*}
  Since $sw\in W_{sts^{-1}}$ and $w,s^2w\notin W_{sts^{-1}}$,
  if $(\ad V)^2(W)=0$ then the second and third line in the last expression
  have to cancel. Since
  \[ s^2ts^{-1}V_s=V_{s^2tst^{-1}s^{-2}}\]
  and $s^2tst^{-1}s^{-2}\not=s$, we conclude that
  \[ sv_2\otimes sts^{-1}v_1+v_1\otimes sts^{-1}v_2=0. \]
  In particular, $\dim V_s=1$ and $sv+v=0$ for all $v\in V_s$.
\end{proof}

\begin{lem} 
	\label{lem:dimVi=1} Let $\theta \in \N $ and $V_1,\dots
	,V_\theta $ be Yetter-Drinfeld modules over $G$. Let $i\in \{1,\dots ,\theta
	\}$ and $J\subseteq \{1,\dots ,\theta \}\setminus \{i\}$ be such that $\supp
	V_j,\supp V_k$ commute for all $j,k\in J\cup \{i\}$.  Assume that
    $G$ is generated by $\cup _{j=1}^\theta \supp V_j$,
    $V_i$ is absolutely simple, $\dim V_i<\infty $,
    and that $(\id -c_{V_j,V_i}c_{V_i,V_j})(V_i\otimes V_j)=0$ for all $j\in \{1,\dots ,\theta
	\}\setminus (J\cup \{i\})$. Then $\dim V_i=1$.
\end{lem}

\begin{proof}
        Lemma~\ref{lem:a=0} and the conditions on $\supp V_i$ imply
        that $\supp V_i$ commutes with $\supp V_j$ for all $1\le j\le \theta $.
		Since $\supp V_i$ is a conjugacy class of $G$ and $G$ is
		generated by $\cup_{j=1}^\theta\supp V_j$, we conclude that $|\supp
        V_i|=1$. Let $t\in \supp V_i$ and let $J'=J\cup \{i\}$.
        By assumption, $rs=sr$ for all $r,s\in \cup _{j\in J'}\supp V_j$,
        and hence the elements of $\cup _{j\in J'}\supp V_j$ have a common eigenspace
        $\tilde{V}$ in $\overline{\K}\otimes _\K V_i$ for some field extension
        $\overline{\K}$ of $\K$.
        Further, for all $r\in \supp V_j$ with $j\in \{1,\dots ,\theta \}
        \setminus J'$ there exists $\lambda _r\in \K$ such that $rv=\lambda
        _rv$ for all $v\in V_i$ by Lemma~\ref{lem:a=0}. Since $G$ is generated by
        $\cup _{j=1}^\theta \supp V_j$, we conclude that all elements of $G$ act by a
        constant on $\tilde{V}$. Since $V_i$ is absolutely simple,
        it follows that $\dim V_i=1$.
\end{proof}

Similar calculations as in the proof of Lemma~\ref{lem:noncommuting}
prove the following claim on braided vector spaces of
diagonal type, which will be needed in the proof of 
Lemma~\ref{lem:reduction}.

\begin{lem} 
	\label{lem:rank3ab}
        Let $g_1,g_2,g_3\in G$ and let $V\in \ydG $. Assume that $g_ig_j=g_jg_i$ for
        all $1\le i<j\le 3$, and that there exist
        $(q_{ij})_{1\le i,j\le 3}\in (\K^\times)^{3\times 3}$
        and linearly independent elements
        $v_i\in V_{g_i}$ for $i\in \{1,2,3\}$
        such that $g_iv_j=q_{ij}v_j$ for all $i,j\in \{1,2,3\}$.
        Then $(\ad v_1)(\ad v_2)(v_3)=0$
        if and only if $q_{23}q_{32}=1$ or $q_{13}q_{31}=q_{12}q_{21}=1$.
\end{lem}

\begin{proof}
        In the proof of \cite[Thm.\,1.1]{MR2732989} it was shown that
        $(\ad v_1)(\ad v_2)(v_3)=0$ if and only if
        $\varphi _2(\id \otimes \varphi _1)(v_1\otimes v_2\otimes v_3)=0$.
        Since
        \begin{align*}
                \varphi _1(v_2\otimes v_3)=&\;(1-q_{23}q_{32})v_2\otimes v_3,\\
                \varphi _2(v_1\otimes v_2\otimes v_3)=&\;
                v_1\otimes (1-q_{12}q_{21}q_{13}q_{31})v_2\otimes v_3\\
                &\;+q_{12}(1-q_{13}q_{31})v_2\otimes v_1\otimes v_3,
        \end{align*}
        the claim follows from the linear independence of $v_1,v_2,v_3$.
\end{proof}

Finally, we make an important observation on tuples with certain Cartan
matrices.

\begin{pro} \label{pro:absimple}
	Let $\theta \in \N _{\ge 2}$, $M\in \cF_\theta^G$, and $i,j\in \{1,\dots ,\theta \}$
	be such that $i\ne j$. Assume that
	$\{-a^M_{ij},-a^M_{ji}\}\in \{ \{0\},\{1\},\{1,2\}\}$. Then
	$(\ad M_i)^m(M_j)$ is absolutely simple or zero for all $m\in \N _0$.
\end{pro}

\begin{proof}
	Since $M\in \cF_\theta^G$, $(\ad M_i)^0(M_j)=M_j$ is absolutely simple.
  On the other hand,
	$(\ad M_i)^a(M_j)=R_1(M_i,M_j)_2$ for $a=-a^M_{ij}$ is absolutely simple
	by \cite[Thm.\,3.8]{MR2766176}, and
	$(\ad M_i)^m(M_j)=0$ for all $m>a$. Thus the claim holds whenever $a_{ij}\in
	\{0,-1\}$. The only remaining case is when $a^M_{ij}=-2$, $a^M_{ji}=-1$, and
	$m=1$. In this case
	$$(\ad M_i)(M_j)\simeq (\id -c_{M_j,M_i}c_{M_i,M_j})(M_i\otimes M_j)\simeq
	(\ad M_j)(M_i)$$
	which is absolutely simple by a previous argument since $a^M_{ji}=-1$.
\end{proof}

\subsection{Cartan matrices and restrictions}

Let $H\subseteq G$ be a subgroup and let $V\in \ydG $.
If $\supp V\subseteq H$, then by restricting the $G$-module structure
of $V$ to $H$ one obtains a unique Yetter-Drinfeld module $V'\in \ydH $
which we will denote by $\Res ^G_H V$.

\begin{lem} \label{lem:Gdecomp}
  Let $H$ be a subgroup of $G$. Let $X\subset G$ be a union of
  conjugacy classes of $G$ such that
  $X\cup H$ generates $G$. Then
  $$G=\langle X\rangle H=H\langle X\rangle .$$
\end{lem}

\begin{proof}
	It follows from $hx=(hxh^{-1})h$ for all $h\in H$, $x\in X$,
  since $G$ is generated by $X\cup H$.
\end{proof}

\begin{lem} 
        \label{lem:simplerestriction}
  Let $H$ be a subgroup of $G$. Let $X\subset G$ be a union of
  conjugacy classes of $G$ such that
  $X\cup H$ generates $G$.
  \begin{enumerate}
    \item Let $V$ be a simple $\K G$-module. If $xv\in \K v$
      for all $v\in V$ and all $x\in X$,
      then $V$ is a simple $\K H$-module by restriction.
    \item Let $V$ be a simple Yetter-Drinfeld module over $G$.
      Assume that $\supp V\subseteq H$.
      Let $h\in \supp V$. If $xv\in \K v$ for all $x\in X$, $v\in V_h$,
      then $\Res ^G_HV\in \ydH $ is simple.
  \end{enumerate}
\end{lem}

\begin{proof}
  (1) By Lemma~\ref{lem:Gdecomp}, $G=H\langle X\rangle $.
  Hence 
	\begin{align} \label{eq:KHv=V}
    V=\K Gv=\K H\langle X\rangle v=\K Hv
	\end{align}
  for all $v\in V\setminus \{0\}$. Therefore $V$ is a simple $\K H$-module by
  restriction.

  (2) Lemma~\ref{lem:Gdecomp} implies that $G=H\langle X\rangle $.
  Since $V$ is simple and $xv\in \K v$ for all $x\in X$ and $v\in V_h$,
	we conclude from~\eqref{eq:KHv=V}
  that $\K Hv=V$ for all $v\in V_h\setminus \{0\}$. Thus $\Res ^G_HV$ is simple.
\end{proof}
 
The last three lemmata in this subsection will be used for induction arguments.

\begin{lem} 
	\label{lem:A2comm}
	Let $\theta \in \N _{\ge 2}$ and $M\in \Ggen^G_\theta $.
	Assume that $a^M_{12}=a^M_{21}=-1$ and that $a^M_{1j}=0$ for all
	$j\in\{3,\dots,\theta\}$.
	Assume further that  $\supp M_1$ and $\supp M_2$
	commute.  Then $\dim M_1=1$ and $\dim M_2=1$.
\end{lem}

\begin{proof}
	From Lemma~\ref{lem:noncommuting}(1) we obtain that $\supp M_1$ and $\supp
	M_2$ are commutative since $a^M_{12}=a^M_{21}=-1$. Hence $\dim M_1=1$ by
	Lemma~\ref{lem:dimVi=1} with $i=1$ and $J=\{2\}$.  Let $r_1\in Z(G)$ with
	$\supp M_1=\{r_1\}$ and let $r_2\in \supp M_2$.  Since any $s_2\in \supp
	M_2$ acts by a constant on $M_1$, Lemma~\ref{lem:noncommuting}(2) with
	$V=M_2$ and $W=M_1$ implies that $G^{r_1}\subseteq G^{r_2}$. Hence
	$G^{r_2}=G$, that is, $r_2\in Z(G)$ and $\supp M_2=\{r_2\}$. Since $r_1\in
	Z(G)$ and $M_2$ is absolutely simple, there exists $\lambda '\in \K^\times $
	such that $r_1v_2=\lambda 'v_2$ for all $v_2\in M_2$. Then
	Lemma~\ref{lem:noncommuting}(3) with $V=M_2$, $W=M_1$ implies that $\dim
	M_2=1$.
\end{proof}

\begin{lem} 
	\label{lem:A2noncomm}
	Let $\theta \in \N _{\ge 2}$ and $M\in \Ggen^G_\theta $.
	Assume that $a^M_{12}=a^M_{21}=-1$ and that $a^M_{1j}=0$
	for all $j\in\{3,\dots,\theta\}$.
	Assume further that $\supp M_1$ and $\supp M_2$ do
	not commute.
	Then $|\supp M_1|=|\supp M_2|=2$ and $\dim M_1=\dim M_2=2$.
\end{lem}

\begin{proof}
	Since $a^M_{12}=a^M_{21}=-1$, Lemma~\ref{lem:noncommuting}(1) tells
	that $\supp M_1$ and $\supp M_2$ are commutative. Moreover, since
		$\supp M_1$ and $\supp M_2$ do not commute, Lemma~\ref{lem:noncommuting}(4)
	implies that $\varphi _r|_{\supp M_2}$ and
  $\varphi _s|_{\supp M_1}$ are transpositions for all $r\in \supp M_1$, $s\in
  \supp M_2$.  Let $r\in \supp M_1$ and $s\in \supp M_2$.  Then $r$ commutes
  with $\supp M_i$ for all $3\le i\le \theta $ by Lemma~\ref{lem:a=0}.  It
  follows that $\supp M_1=r^G=\{r,s\trid r\}$.
	Moreover, $\dim (M_1)_r=1$ and $\dim (M_2)_s=1$
	by Lemma~\ref{lem:noncommuting}(4).
	Since $s$ does not
  commute with any element of $r^G$, the same holds for all $s'\in s^G$.  Then
  $|\supp M_2|=2$ since $\varphi _r|_{\supp M_2}$ is a transposition.
\end{proof}


\begin{lem} 
	\label{lem:reduction}
	Let $\theta \in \N _{\ge 3}$ and $M\in \Ggen^G_\theta $.
	Assume that $a^M_{12}=a^M_{21}=a^M_{23}=-1$ and that
	$a^M_{1j}=0$ for all $j\in\{3,\dots,\theta\}$.
	Let $H=\langle \cup _{j=2}^\theta \supp
	M_j\rangle $ and $M'=(\Res _H^GM_j)_{2\le j\le \theta }$.  Then $M'\in
	\Ggen_{\theta -1}^H$. If $H$ is abelian, then $G$ is abelian.
\end{lem}

\begin{proof}
	If $\supp M_1$ and $\supp M_2$ commute, then $\dim M_1=1$ by
	Lemma~\ref{lem:A2comm}. Hence $\supp M_1$ consists of a central element of
	$G$, and the claim follows from Lemma~\ref{lem:simplerestriction}(2).

	Assume that $\supp M_1$ and $\supp M_2$ do not commute. Then $\dim
	M_1=\dim M_2=2$ and $|\supp M_1|=|\supp M_2|=2$ by
	Lemma~\ref{lem:A2noncomm}. In particular, either $\supp M_2\subseteq
		Z(H)$ or $\Res _H^GM_2\in \ydH $ is absolutely simple. Assume first that
		$\supp M_2$ does not commute with $\supp M_i$
	for some $3\le i\le \theta $. Then $\Res _H^GM_2\in \ydH $ is absolutely
	simple. Further, $\Res _H^GM_i\in \ydH $ is absolutely simple for all $i\ge 3$
	by Lemmas~\ref{lem:a=0} and \ref{lem:simplerestriction}(2).
	Then $M'\in \Ggen^H_{\theta -1}$ and $H$ is non-abelian.

	Assume that $\supp M_1$ and $\supp M_2$ do not commute, and that $\supp
	M_2$ commutes with $\supp M_3$.
	Let $r,r',s,s'\in G$ be such that $r\ne r'$,
	$s\ne s'$, and $\supp M_1=\{r,r'\}$, $\supp M_2=\{s,s'\}$. Let $t\in \supp
	M_3$ be such that $(\id -c^2)( (M_2)_s\otimes (M_3)_t)\ne 0$.  By
	Lemma~\ref{lem:a=0}, there exists $\lambda \in \K^\times $ such that
	$rw=\lambda w$ for all $w\in (M_3)_t$.  Assume that $\K $ contains all
	eigenvalues of the action of $s$ and $s'$ on $(M_3)_t$. Since $G^t$ is
		generated by $(G^t\cap G^s)\cup \{r\}$ by Lemma~\ref{lem:rs}(4), a
		joint eigenspace $W$ of $s$ and $s'$ in $(M_3)_t$
		is then invariant under the action of $G^t$.
	Since $M_3$ is absolutely simple, we conclude that
	$s$ and $s'$ act by a constant on $(M_3)_t$.  Since $rsr^{-1}=s'$, these two
	constants coincide.  By the same reason, $t$ acts by a constant on
	$M_2=(M_2)_s\oplus (M_2)_{s'}$.  Since $a^M_{23}=-1$ and $(\id -c^2)(
	(M_2)_s\otimes (M_3)_t)\ne 0$, Lemma~\ref{lem:rank3ab} implies that $\ad
	(M_2)_s\ad (M_2)_{s'}((M_3)_t)\ne 0$, which is a contradiction to
	$a^M_{23}=-1$.
\end{proof}

\subsection{Skeletons of finite type}

Here we collect two basic lemmas about skeletons and their reflections.

\begin{lem}
	\label{lem:disconnected}
	Let $J,K\subseteq\{1,\dots,\theta\}$ be disjoint non-empty subsets and let
	$i\in J$. Let $M\in\cF_\theta^G$ be such that $a^M_{ij}\in
	\Z $ for all $j\in \{1,\dots,\theta \}$. If 
	$a^{M}_{jk}=0$ for all  
	$j\in J$ and $k\in K$ then $a^{R_i(M)}_{jk}=0$ for all $j\in J$ and $k\in K$.
\end{lem}

\begin{proof}
	Suppose that $j\ne i$. Recall that $R_i(M)_j=(\ad M_i)^m(M_j)$, where
  $m=-a_{ij}^M$, and
	$(\ad M_i)^m(M_j)\simeq\varphi(M_i^{\otimes m}\otimes M_j)$ for some morphism
	$\varphi$ in $\ydG$, see Lemma~\ref{lem:X_n}.
	In particular, $R_i(M)_k=M_k$ for all $k\in K$. Moreover, $a^M_{jk}=0$ if and
	only if $c_{M_k,M_j}c_{M_j,M_k}=\id _{M_j\otimes M_k}$.
	Since $c^2$ is a natural
	isomorphism, it commutes with $\varphi\otimes\id$. This implies the claim of the
	lemma for $j\ne i$. The case where $j=i$ means that $(\id-c_{W,V}c_{V,W})(V\otimes
	W)=0$ implies that $(\id-c_{W,V^*}c_{V^*,W})(V^*\otimes W)=0$ for $V=M_i$ and
	$W=M_k$, where $k\in K$. The latter is well-known.
\end{proof}

The following lemma and the remark below will be used to simplify the
calculations of the skeletons of reflections of tuples.

\begin{lem}
	\label{lem:triples}
	Let $\theta\geq3$, $i\in\{1,\dots,\theta\}$ and 
	let $M\in\cF_\theta^G$. Suppose that $M$ has a skeleton
	and that for all $j,k\in\{1,\dots,\theta\}\setminus\{i\}$ with $j\ne k$,
	the triple $R_1(M_i,M_j,M_k)$ has a skeleton $\cS '_{jk}$.
	Then $R_i(M)$ has a skeleton $\cS '$. Moreover, $\cS '$ is uniquely determined
	such that it restricts to $\cS '_{jk}$ when
	considering only the vertices $i$, $j$, and $k$.
\end{lem}

Note that $R_1(M_i,M_j,M_k)$ means reflection on the first entry of the
triple, that is, on $M_i$.

\begin{proof}
	The definition of a skeleton of $R_i(M)$ and its existence consist of a family of
	conditions in each of which at most two entries $R_i(M)_j$, $R_i(M)_k$ with
	$j,k\in\{1,\dots,\theta\}$ are involved. Thus these conditions can be obtained
	from $R_1(M_i,M_j,M_k)$. This implies the claim.
\end{proof}

\begin{rem} \label{rem:connectedtriples}
	Let $\theta\geq3$, $i\in\{1,\dots,\theta\}$ and 
	let $M=(M_1,\dots,M_\theta)\in\cF_\theta^G$. Suppose that $M$ has a connected
	skeleton $\cS $.
  Lemma~\ref{lem:triples} can be used to obtain quickly the skeleton of $R_i(M)$
  for some $M\in \cF_\theta^G$ (if it exists).

	Assume that for all $j,k\in
	\{1,\dots ,\theta \}\setminus \{i\}$ such that $j\ne k$ and the skeleton of
	$(M_i,M_j,M_k)$ is connected, the triple $R_1(M_i,M_j,M_k)$ has a skeleton $\cS '_{jk}$.
	We show that then the conditions of Lemma~\ref{lem:triples} are fulfilled and hence
	$R_i(M)$ has a skeleton.

	Indeed, for any triple $(i,j,k)$ with $|\{i,j,k\}|=3$ one of the following possibilities
	occurs:
	\begin{enumerate}
		\item $j$ and $k$ are not connected with $i$ in $\cS $. Then $R_i(M)_j=M_j$,
			$R_i(M)_k=M_k$, and hence $R_1(M_i,M_j,M_k)$ has a skeleton $\cS '_{jk}$. In this
			skeleton, $j$ and $k$ are not connected with $i$ by
			Lemma~\ref{lem:disconnected}. Hence $\cS '_{jk}$
			coincides with the skeleton of $(M_i,M_j,M_k)$.
		\item $(M_i,M_j,M_k)$ has a connected skeleton. Then $R_1(M_i,M_j,M_k)$ has
			a connected skeleton by assumption.
		\item Precisely one of $j$ and $k$ (say $j$) is connected with the vertex $i$
			and the other is neither connected with $i$ nor with $j$.
			Then
			$R_i(M)_k=M_k$. Moreover, there exists $l\in \{1,\dots ,\theta \}\setminus
			\{i,j,k\}$ such that $(M_i,M_j,M_l)$ has a connected skeleton. Then
			$R_1(M_i,M_j,M_l)$ has a connected skeleton by assumption. Then
			$R_1(M_i,M_j,M_k)$ has a skeleton with two connected components by
			Lemma~\ref{lem:disconnected}. 
	\end{enumerate}
	This leads to the claim on the existence (and the shape) of the skeleton of
	$R_i(M)$.
\end{rem}

\section{Proof of Theorem \ref{thm:ADE}: The case $ADE$}
\label{section:ADE}

In this section we require that all assumptions in Theorem~\ref{thm:ADE} hold.
Thus let $\theta\in\N_{\geq 2}$ and let $G$ be a non-abelian group and
$M\in\Ggen_\theta^G$.
Assume that the Cartan matrix
$A^M$ of $M$ is a Cartan matrix of type $A_\theta $
with $\theta \geq 2$, or $D_\theta $ with $\theta \ge 4$, or $E_\theta $ with
$\theta \in \{6,7,8\}$.

\begin{lem} 
	\label{lem:ADE:noncomm}
	The following hold:
	\begin{enumerate}
		\item $|\supp M_i|=2=\dim M_i$ for all $i\in\{1,\dots,\theta\}$. 
		\item $\supp M_i$ does not commute with $\supp M_j$ whenever $a^M_{ij}=-1$.  
	\end{enumerate}
\end{lem}

\begin{proof}
	We proceed by induction on $\theta $. If $\theta =2$, then $A^M$ is of type
	$A_2$. If $\supp M_1$ and $\supp M_2$ commute, then
	Lemma~\ref{lem:A2comm} 
	implies that $G$ is commutative, which is a
	contradiction to our assumption. Hence $\supp M_1$ and $\supp M_2$ do not
	commute, and the lemma follows from 
	Lemma~\ref{lem:A2noncomm}.  
	Assume
	that $\theta \ge 3$.  Let $I=\{1,\dots ,\theta \}$.  By the assumptions on
	$A^M$ there exist $i,j,k\in I$ such that $a^M_{ij}=a^M_{ji}=a^M_{jk}=-1$,
	and $a^M_{il}=0$ for all $l\in I\setminus \{i,j\}$.  Let $H$ be the subgroup
	of $G$ generated by $\cup _{l\in I\setminus \{i\}}\supp M_l$. Then $M'=(\Res
	_H^G M_l)_{l\in I\setminus \{i\}}\in \Ggen_{\theta -1}^H$ by
	Lemma~\ref{lem:reduction}, 
	and $a^{M'}_{lm}=a^M_{lm}$ for all $l,m\in
	I\setminus \{i\}$. Hence, by induction hypothesis, the lemma holds for all
	$l\in I\setminus \{i\}$. In particular, $\dim M_j=2$.
	Then
	$\supp M_i$ and $\supp M_j$ do not commute and $|\supp M_i|=2=\dim M_i$
	by Lemmas~\ref{lem:A2comm} and \ref{lem:A2noncomm}.
\end{proof}

The following lemma describes the structure of the Yetter-Drinfeld modules
encoded in a skeleton of types $\alpha_\theta$, $\delta_\theta$,
$\varepsilon_6$, $\varepsilon_7$ and $\varepsilon_8$. 

\begin{lem}
	\label{lem:ADE:conditions}
	Let $N\in\cF^G_\theta$.
	The following are equivalent:
	\begin{enumerate}
		\item $N$ has a connected simply-laced skeleton of finite type.
		\item 
			There exist
			\begin{itemize}
				\item a symmetric indecomposable Cartan matrix $A\in \Z^{\theta \times
					\theta}$ of finite type,
				\item an element $\epsilon\in Z(G)$ with $\epsilon^2=1$, and
				\item for all $i\in\{1,\dots,\theta\}$, $s_i\in\supp N_i$
					a unique character $\sigma_i$ of $G^{s_i}$,
			\end{itemize}
			such that $\supp N_i=\{s_i,\epsilon s_i\}$ and $N_i\simeq M(s_i,\sigma_i)$
			for all $i\in\{1,\dots,\theta\}$, and  
			the following conditions hold:
			\begin{align}
				\label{eq:ADE:1}\sigma_i(s_j)\sigma_j(s_i)=\sigma_i(\epsilon
				)\sigma_j(\epsilon )=1
				&&&
				\text{for all $i,j$ such that $a_{ij}=0$},\\
				\label{eq:ADE:2}\sigma_i(\epsilon s_j^2)\sigma_j(\epsilon s_i^2)=1 &&& \text{for all
				$i,j$ such that $a_{ij}=-1$},\\
				\label{eq:ADE:3}\sigma_i(s_i)=-1 &&& \text{for all
				$i\in\{1,\dots,\theta\}$},\\
				\label{eq:ADE:4}s_is_j=\epsilon s_js_i &&& \text{for all
				$i,j$ such that $a_{ij}=-1$},\\
				\label{eq:ADE:5}s_is_j=s_js_i &&& \text{for all $i,j$ such that $a_{ij}=0$}.
			\end{align}
		\item Let $P=(\Res^G_H N_1,\dots ,\Res ^G_H N_\theta)$, where $H\subseteq G$
			is the subgroup generated by $\cup _{i=1}^\theta \supp N_i$. Then $H$ is
			non-abelian,
			$P\in \Ggen^H_\theta $, and $A^P$ is of type
			$A_\theta$ with $\theta\geq2$, $D_\theta$ with $\theta\geq4$, or $E_\theta$
			with $\theta\in\{6,7,8\}$.
	\end{enumerate}
\end{lem}

\begin{proof}
	The implication (1)$\Rightarrow $(3) follows from the definition of a
	simply-laced skeleton.

	We prove that (3) implies (2). Let $A=A^P(=A^N)$. Then $A$ is a symmetric
	indecomposable Cartan matrix of finite type and $|\supp N_i|=\dim N_i=2$
	for all $i\in \{1,\dots ,\theta \}$ by Lemma~\ref{lem:ADE:noncomm}.
	Moreover, Lemmas~\ref{lem:ADE:noncomm} and \ref{lem:a=0} imply that
	$\supp N_i$ commutes with $\supp N_j$, where $i\ne j$, if and only if $a_{ij}=0$.
	Let $s_i\in \supp N_i$ for all $i\in \{1,\dots ,\theta \}$. Then for all
	$i\in \{1,\dots ,\theta \}$ there exists a unique character $\sigma _i$ of
	$G^{s_i}$ such that $N_i\simeq M(s_i,\sigma _i)$. Lemma~\ref{lem:A3:222}
	implies that there exists $\epsilon \in Z(G)$ such that $\epsilon ^2=1$,
	$\supp N_i=\{s_i,\epsilon s_i\}$ for all $i\in \{1,\dots ,\theta \}$, and
	\eqref{eq:ADE:4} holds. Now \eqref{eq:ADE:3} holds by
	Lemma~\ref{lem:noncommuting}(4), and \eqref{eq:ADE:1}
	follows from Lemma~\ref{lem:a=0for22}. Finally, if $a_{ij}=-1$ then $(\ad
	N_i)(N_j)$ is absolutely simple. Therefore \eqref{eq:ADE:2} follows from
	Lemma~\ref{lem:HS:X1}.

	Finally we prove that (2) implies (1). Let $i,j\in \{1,\dots ,\theta \}$
	be such that $i\ne j$. Since $\epsilon \in Z(G)$, we conclude from
	Lemma~\ref{lem:a=0for22} and from \eqref{eq:ADE:1} and
	\eqref{eq:ADE:5}, that $(\ad N_i)(N_j)=0$ if $a_{ij}=0$.
	Finally, if $a_{ij}=-1$ then \eqref{eq:ADE:2}--\eqref{eq:ADE:4} and
	Corollary~\ref{co:2-2:A2B2CartanMatrix} imply that $a^N_{ij}=-1$.
	This proves (1).
\end{proof}

We now study some reflections. In the case of rank three one has the following
lemma.

\begin{lem}
	\label{lem:skeleton:A3}
	Let $N\in\cF_3^G$. Assume that $N$ has a skeleton $\cS $ of type
	$\alpha_3$.
	Then $\cS $ is a skeleton of $R_k(N)$ for each $k\in\{1,2,3\}$. 
\end{lem}

\begin{proof}	
	By symmetry of the skeleton of type $\alpha_3$,
	it suffices to prove the lemma for the reflections $R_1$ and
	$R_2$. Let $s_i\in G$ and $\sigma_i\in \chg {G^{s_i}}$ be as in
	Lemma~\ref{lem:ADE:conditions}(2).
	Let $(U,V,W)=R_1(M)$. Then Lemma~\ref{lem:HS:reflections} implies
	that $U\simeq M(s_1^{-1},\sigma_1^*)$, $V\simeq M(s_1s_2,\sigma')$ and
	$W=M_3$, where $\sigma'\in \chg{G^{s_1s_2}}$ with
	$\sigma'(s_1s_2)=-1$ and
	$\sigma'(h)=\sigma_1(h)\sigma_2(h)$ for all $h\in G^{s_1}\cap G^{s_2}$.
	For the proof of
	the claim we use Lemma~\ref{lem:ADE:conditions}. For $(U,V,W)$,
	Conditions~\eqref{eq:ADE:1} and \eqref{eq:ADE:5} follow from
	Lemmas~\ref{lem:disconnected} and \ref{lem:a=0for22}.
	Conditions~\eqref{eq:ADE:2} and \eqref{eq:ADE:4} for $\{i,j\}=\{1,2\}$
	and \eqref{eq:ADE:3} for $i\in \{1,2\}$ hold by Lemma~\ref{lem:HS:reflections}.
	Condition~\eqref{eq:ADE:3} for $i=3$ holds since $R_1(M)_3=M_3$.
	Thus we need to prove \eqref{eq:ADE:2} and \eqref{eq:ADE:4} for $i=2$, $j=3$.

    Clearly, \eqref{eq:ADE:4} follows easily, since $s_1s_3=s_3s_1$ and
$s_2s_3=\epsilon s_3s_2$ imply that $s_1s_2s_3=\epsilon s_3s_1s_2$.
Regarding \eqref{eq:ADE:2} we obtain the following:
	\begin{align*}
		&\sigma'(\epsilon s_3^2)\sigma_3(\epsilon (s_1s_2)^2)
        =\sigma_1(\epsilon s_3^2)\sigma _2(\epsilon s_3^2)\sigma_3(s_1^2s_2^2)
        =\sigma_1(\epsilon )\sigma_3(\epsilon )
        =1,
    \end{align*}
    where the last equation follows from Lemma~\ref{lem:a=0for22}.

	Let now $(U',V',W')=R_2(M)$. By Lemma~\ref{lem:HS:reflections},
	$$U'\simeq M(s_2s_1,\rho),\quad V'\simeq M(s_2^{-1},\sigma_2^*), \quad
      W'\simeq M(s_2s_3, \tau),
    $$
    where $\rho\in \chg{G^{s_2s_1}}$ with
	$\rho(s_2s_1)=-1$, $\rho(h)=\sigma_1(h)\sigma_2(h)$ for all $h\in
G^{s_1}\cap G^{s_2}$,
	and $\tau\in \chg{G^{s_2s_3}}$ with $\tau(s_2s_3)=-1$,
	$\tau(h)=\sigma_2(h)\sigma_3(h)$ for all $h\in G^{s_2}\cap G^{s_3}$.
    As in the first part of the proof of the Lemma, one
	needs to check the conditions of Lemma~\ref{lem:ADE:conditions} for
    $R_2(M)$.

	Conditions~\eqref{eq:ADE:2}--\eqref{eq:ADE:4} follow from
Lemma~\ref{lem:HS:reflections}. For \eqref{eq:ADE:5} we record that
$$ (s_2s_1)(s_2s_3)=s_2\epsilon s_2s_1s_3=s_2\epsilon s_2s_3s_1=s_2s_3s_2s_1
$$
since $\epsilon ^2=1$. Finally, $s_1^{-1}s_3\in G^{s_1}\cap G^{s_2}\cap G^{s_3}$
and hence we
get~\eqref{eq:ADE:1} from the calculations
	\begin{align*}
	  \rho (s_2s_3)\tau(s_2s_1)&=\rho(s_2s_1s_1^{-1}s_3)\tau (s_2s_3s_3^{-1}s_1)\\
	  &=(-1)\rho(s_1^{-1}s_3)(-1)\tau (s_3^{-1}s_1)\\
      &=\sigma_1(s_1^{-1}s_3)\sigma_2(s_1^{-1}s_3s_3^{-1}s_1)\sigma_3(s_3^{-1}s_1)=1
	\end{align*}
    and
	\[
		\rho(\epsilon)\tau(\epsilon)=\sigma_1(\epsilon)\sigma_2(\epsilon)^2\sigma_3(\epsilon)=1.
	\]
	This completes the proof.
\end{proof}

The reflections 
are studied by the following proposition.

\begin{pro}
	\label{pro:skeleton:ADE}
	Let  $N\in\cF_\theta^G$.  Suppose
	that $N$ has a skeleton $\cS $ of type $\alpha_\theta$, $\delta_\theta $ (with
	$\theta \ge 4$), $\varepsilon _6$, $\varepsilon _7$, or $\varepsilon _8$. Then
	$\cS $ is a skeleton of  $R_k(N)$ for all $k\in \{1,\dots ,\theta \}$.
\end{pro}

\begin{proof}
	For $\theta=2$ the claim follows from Lemmas~\ref{lem:ADE:conditions} and
	\ref{lem:HS:reflections}.

	Assume that $\theta\ge 3$.
	By Remark~\ref{rem:connectedtriples}, it is enough to prove that for all
	pairwise distinct $i,j,k\in\{1,\dots,\theta\}$ such that the skeleton $\cS
	_{ijk}$ of $(M_i,M_j,M_k)$ is connected, $\cS _{ijk}$ is a skeleton
	of $R_1(M_i,M_j,M_k)$. All such skeletons are of type $\alpha_3$. Hence
	the claim follows from Lemma~\ref{lem:skeleton:A3}.
\end{proof}

Now we are ready to complete the proof of Theorem \ref{thm:ADE}. 

\begin{proof}[Proof of Theorem \ref{thm:ADE}]
	\eqref{it:ADE:skeleton} holds by
	Lemma~\ref{lem:ADE:conditions}(3)$\Rightarrow $(1), and
	\eqref{it:ADE:groupoid} follows from \eqref{it:ADE:skeleton} and
	Proposition~\ref{pro:skeleton:ADE}.
	
	\eqref{it:ADE:Nichols} Theorem~\ref{thm:HS} applies because of
	\eqref{it:ADE:groupoid}. Since the Cartan graph of $M$ is standard, the root
	system of $M$ coincides with the root system associated with the Cartan matrix
	$A^M$. Hence
	$$ \cH(t)=\prod_{\alpha\in\mathbf{\Delta}_+}\cH _{\NA (M_\alpha )}(t^\alpha
	). $$
	The Nichols algebras $\NA (M_i)$ are quantum linear spaces with Hilbert series
	$(1+t)^2$, see also Theorem~\cite[Thm.~4.6(2)]{MR2732989}.
	Then the claim on the Hilbert series of $\NA (M)$ follows from
	Theorem~\ref{thm:HS}.
\end{proof}

\section{Proof of Theorem \ref{thm:C}: The case $C$}
\label{section:C}

In this section we require that all assumptions in Theorem~\ref{thm:C} hold.
Let $\theta\in\N_{\geq3}$ and let $G$ be a non-abelian group. Assume that
$M\in\Ggen_\theta^G$ and that $A^M$ is a Cartan matrix of type $C_\theta$, where
$a^M_{\theta-1,\theta}=-2$ and $a^M_{ij}=-1$ for $|i-j|=1$, $(i,j)\ne
(\theta-1,\theta )$.

\subsection{}
\label{sub:C:theta=3}

We first study some particular aspects for triples.

\begin{lem}
	\label{lem:C3:noncomm}
	Assume that $\theta =3$. Then the following hold:
	\begin{enumerate}
		\item $|\supp M_1|=|\supp M_2|=\dim M_1=\dim M_2=2$ and $\dim M_3=1$.
		\item $\supp M_1$ does not commute with $\supp M_2$.
	\end{enumerate}
\end{lem}

\begin{proof}
	Suppose that $\supp M_1$ and $\supp M_2$ commute. Then $\dim M_i=1$ for all $i\in
	\{1,2\}$ by
	Lemma~\ref{lem:A2comm}. Since $a^M_{32}=-1$,
	Lemma~\ref{lem:noncommuting}(1) implies that $\supp M_3$ is commutative.  Then
	$G$ is abelian, a contradiction. Hence $\supp M_1$ and $\supp
	M_2$ do not commute. Then Lemma~\ref{lem:A2noncomm} implies that
	$$|\supp M_1|=|\supp M_2|=\dim M_1=\dim M_2=2.$$

  Let $r\in \supp M_1$, $s_1,s_2\in \supp M_2$ with $s_1\ne s_2$, and $t\in
	\supp M_3$. Then $rt=tr$ because of $a^M_{13}=0$. Hence
	$$\supp M_3\ni s_1ts_1^{-1}=r(s_1ts_1^{-1})r^{-1}
	=s_2rtr^{-1}s_2^{-1}=s_2ts_2^{-1}.$$
	Assume that $\supp M_2$ and $\supp M_3$ do not commute.
	Then $s_1,s_2$ act on $\supp M_3$ via conjugation by the same transposition
	because of Lemma~\ref{lem:noncommuting}. Since $\supp M_3$
	is a conjugacy class of $G$, we conclude that $|\supp M_3|=2$.
	Moreover, $\dim (M_3)_t = 1$ by Lemma \ref{lem:noncommuting}(4).
	Let now $\sigma $ be a character of $G^{s_1}$ such that
	$M_2\simeq M(s_1,\sigma )$. Then Corollary~\ref{co:2-2:A2B2CartanMatrix}
	for $(M_1,M_2)$ and $(M_2,M_3)$ implies that $\sigma (s_1)=-1$ and $\sigma
	(s_1)=1$, $\charK=3$, respectively. This is clearly impossible.
	Hence $\supp M_2$ and $\supp M_3$ commute.

	Since $a^M_{32}=-1$ and $\supp M_3$ commutes
	with $\supp M_1$ and $\supp M_2$, we conclude from Lemma \ref{lem:dimVi=1}
	for $\theta =3$, $V_1=M_1$, $V_2=M_2$, $V_3=M_3$, $i=3$, $J=\{2\}$,
	that $\dim M_3=1$. 	
\end{proof}

In the following two lemmas we consider a slightly more general context, which is
motivated by Lemma~\ref{lem:C3:noncomm} and will be used crucially in the proof of
Lemma~\ref{lem:C3:more}.

Let $N\in \cF ^G_3$ and let $r\in \supp N_1$, $s\in \supp N_2$, $t\in \supp N_3$.
Assume that $|r^G|=|s^G|=2$, $t\in Z(G)$, and $rs\ne sr$.
Let $\epsilon \in G$ such that $rs=\epsilon sr$. Then $\epsilon \ne 1$.
Moreover, $r^G=\{r,\epsilon r\}$, $s^G=\{s,\epsilon s\}$, $\epsilon ^2=1$, and
$\epsilon \in Z(G)$ by Lemma~\ref{lem:A3:222}(1).
Assume further that $N_1\simeq M(r,\rho)$, $N_2\simeq M(s,\sigma)$ and
$N_3\simeq M(t,\tau)$,
where $\rho \in \chg{G^r}$, $\sigma\in \chg{G^s}$, and $\tau\in \chg G$.

\begin{lem}
	\label{lem:C3}
	The following are equivalent:
	\begin{enumerate}
		\item $A^N$ is of type $C_3$. 
		\item\label{it:C3:conditions}
			The following hold:
			\begin{align*}
				&\rho(\epsilon s^2)\sigma(\epsilon r^2)=\rho(t)\tau(r)=1, &&\rho(r)=\sigma(s)=-1,\\
				&(\tau(t)+1)(\sigma(t)\tau(st)-1)=0,&&\sigma(t)\tau(s)\ne1.
			\end{align*}
	\end{enumerate}
\end{lem}

\begin{proof}
	We first prove that (1) implies (2).
    Since $N\in \cF ^G_3$ and $A^N$ is of type $C_3$,
	Proposition~\ref{pro:absimple} implies that $(\ad N_i)^m(N_j)$ is absolutely
	simple or zero for all $m\in \N_0$ and all $i,j\in \{1,2,3\}$ with $i\ne j$.
	By Corollary~\ref{co:2-2:A2B2CartanMatrix}, $a^N_{12}=a^N_{21}=-1$ implies
	that $\rho (\epsilon s^2)\sigma (\epsilon r^2)=1$ and $\rho (r)=\sigma (s)=-1$.
	Further, from Lemma~\ref{lem:X1} and from $a^N_{13}=0$, $a^N_{23}\ne 0$ we obtain that
	$\rho(t)\tau(r)=1$, $\sigma(t)\tau(s)\ne 1$.
	Finally, since $a^N_{32}=-1$, Lemma~\ref{lem:genRosso} implies that
	$(\tau (t)+1)(\sigma (t)\tau (st)-1)=0$.

	Now assume that (2) holds. Then $a^N_{12}=a^N_{21}=-1$ by
	Corollary~\ref{co:2-2:A2B2CartanMatrix},
	$a^N_{13}=a^N_{31}=0$ by Lemma~\ref{lem:X1}, and
	$a^N_{23}=-2$ by Lemmas~\ref{lem:X2} and \ref{lem:X3}(1). Finally, $a^N_{32}=-1$
	by Lemma~\ref{lem:genRosso}. This proves (1).
\end{proof}

The classes $\wp _0^G$ and $\wp _1^G$ of pairs are introduced in
Definition~\ref{def:wp22}.
 
\begin{lem}
	\label{lem:C3:conditions}
	Suppose that $N$ admits all reflections and the Weyl groupoid of
	$N$ is finite. Then $(N_2,N_3)\in \wp_0^G$ or $(N_2,N_3)\in\wp_1^G$.
\end{lem}

\begin{proof}
  Regard $N_1$ and $N_2$ as absolutely simple Yetter-Drinfeld modules over
	$H=\langle \supp (N_1\oplus N_2)\rangle $. Then $H$ is a non-abelian
	epimorphic image of $\Gamma _2$.
	By Theorem~\ref{thm:admabsimple},
	the Yetter-Drinfeld modules $(\ad N_1)^m(N_2)$ and $(\ad N_2)^m(N_1)$ are
	absolutely simple or zero for all $m\ge 0$, and they are zero for
	some $m\in \N$. Thus Lemma~\ref{lem:HS:CartanMatrix} implies that
	\begin{equation} \label{eq:22finite}
		\begin{gathered}
	    \rho (r)^2=\sigma(s)^2=1, \quad
		  \rho (\epsilon s^2)\sigma (\epsilon r^2)=1,\quad \text{and}\\
		  \rho(r)=\sigma(s)=-1\text{ if $\charK =0$.}
	  \end{gathered}
  \end{equation}
	Moreover, Corollary~\ref{cor:finiteW} applied to $(N_2,N_3)$ implies that
	$$(N_2,N_3)\in \wp^G_i\quad \text{for some $i\in \{0,1,2,3,4\}$,}$$
	since $\epsilon ^2=1$ --- see
	also Definition~\ref{def:wp22} and Table~\ref{tab:wpH}. But
  $\sigma (s)^2=1$ implies that
	$$(N_2,N_3)\notin \wp^G_4.$$

  Consider $R_3(N)=(U,V,W)$. Since $\supp N_3=\{t\}$ and $t\in Z(G)$,
  Lemma~\ref{lem:genRosso} implies that $(U,V,W)$ satisfies the assumptions of the
  lemma. In particular, $(V,W)\notin \wp^G_4$, and hence
	$$(N_2,N_3)\notin \wp^G_2$$
  by Remark~\ref{rem:wpreflections}.

	Consider $R_2(N)=(U',V',W')$. Then \eqref{eq:22finite} and
	Lemma~\ref{lem:HS:reflections} for $(N_2,N_1)$ imply that
	$$\dim U'=\dim V'=|\supp U'|=|\supp V'|=2$$
	and $\supp U'$, $\supp V'$ do not commute.
	Moreover, Remark~\ref{rem:wpreflections} for $(N_2,N_3)$ implies that $\dim
	W'=1$.
	In particular, $(U',V',W')$ satisfies the
	assumptions of the lemma.
  Therefore
  $(V',W')\notin \wp^G _2$, and hence
	$$(N_2,N_3)\notin \wp^G _3$$
	by Remark~\ref{rem:wpreflections}. This finishes the proof of the lemma.
\end{proof}

Now we look again at our main tuple $M$.

\begin{lem}
	\label{lem:C3:more}
	Suppose that $\theta=3$, $M$ admits all reflections,
	and the Weyl groupoid of $M$ is finite. Then
	$(M_2,M_3)\in\wp_1^G(2)$ and $\charK\ne 2$.
\end{lem}

\begin{proof}
	By Lemma~\ref{lem:C3:noncomm}, $M$ satisfies the assumptions on $N$ above
	Lemma~\ref{lem:C3}. Let $r,s,t\in G$ and $\rho,\sigma ,\tau $ as there.
	Since $a^M_{23}\ne 0$, we obtain from Lemma~\ref{lem:C3:conditions} that
	$(M_2,M_3)\in \wp^G_1$.
	In particular, $\sigma (t)\tau (st)=1$ and $\tau (t)\ne
	1$. Moreover, since $A^M$ is of type $C_3$, the formulas in
	Lemma~\ref{lem:C3}\eqref{it:C3:conditions} hold.
	Let $M'=R_2(M)$.
	Since $a^M_{21}=-1$, Lemma~\ref{lem:HS:reflections} implies that
	$M'_1\simeq M(sr,\rho ')$, where
	$\rho '\in \chg{G^{sr}}$ with $\rho '(sr)=-1$, $\rho '(h)=\rho
	(h)\sigma (h)$ for all $h\in G^r\cap G^s$. Further, $M'_2\simeq
	M(s^{-1},\sigma ^*)$ and $M'_3\simeq M(\epsilon s^2t,\tau _2)$ by
	Lemma~\ref{lem:X2}, since $a^M_{23}=-2$. Since $\epsilon ^2=1$ and
	$\sigma(s)=-1$, we obtain that
	$$ \tau_2(r)=-\sigma(\epsilon r^2)\tau (r),\quad
	\tau _2(s)=\sigma (\epsilon s^2)\tau (s),\quad \tau_2(t)=\sigma (t^2)\tau (t).
	$$
	Then
	\begin{align} \label{eq:R2new}
		\rho '(\epsilon s^2t)\tau _2(sr)=\rho(\epsilon s^2t)\sigma (\epsilon s^2t)
		\sigma(\epsilon s^2)\tau (s)(-\sigma (\epsilon r^2)\tau (r))=-\sigma (t)\tau (s).
	\end{align}
	Now
	Lemma~\ref{lem:C3:conditions} for $N=(M'_2,M'_1,M'_3)$
	implies that $(M'_1,M'_3)\in \wp^G_0$ or
	$(M'_1,M'_3)\in \wp^G_1$. In the first case $\sigma (t)\tau (s)=-1$, and hence
	$\tau (t)=-1$ and
	$$(M_2,M_3)\in\wp_1^G(2).$$
	Moreover, $\charK\ne 2$ since $\tau (t)\ne 1$ by the
	first paragraph.

	Assume that $(M'_1,M'_3)\in \wp^G_1$. Since $(M'_2,M'_3)\in \wp^G_1$,
	Remark~\ref{rem:wpreflections} implies that
	$a_{13}^{M'}=a_{23}^{M'}=-2$.
	Moreover, since $A^M$ is of type $C_3$, the Cartan graph of $M$
	has no point with a Cartan matrix of type $A_3$ by Theorem~\ref{thm:ADE}.
	This is a contradiction to
	$a_{13}^{M'}=a_{23}^{M'}=-2$ because of
	Corollary~\ref{cor:columns}.
\end{proof}

\subsection{}
\label{subs:C2}

Recall the assumptions of Section~\ref{section:C}:
$\theta\in\N_{\geq3}$, $G$ is a non-abelian group, and
$M\in\Ggen_\theta^G$ such that $A^M$ is of type $C_\theta$. 

\begin{lem} 
	\label{lem:Ctheta:noncomm}
	The following hold:
	\begin{enumerate}
		\item $|\supp M_i|=2=\dim M_i$ for all $1\le i\le \theta -1$, and
			$\dim M_\theta =1$.
		\item $\supp M_i$ does not commute with $\supp M_{i+1}$ for
			$1\le i\le \theta -2$.
	\end{enumerate}
\end{lem}

\begin{proof}
	We proceed by induction on $\theta$. For $\theta=3$ the claim holds by
	Lemma~\ref{lem:C3:noncomm}.

	Assume that $\theta >3$.
	Let $H$ be the subgroup of $G$ generated by $\cup _{i=2}^\theta \supp M_i$.
	Then
	$$M'=(\Res ^G_HM_2,\dots ,\Res^G_HM_\theta )\in \Ggen^H_{\theta -1}$$
	by Lemma~\ref{lem:reduction}, and $H$ is non-abelian.
  Clearly, $A^{M'}$ is of type $C_{\theta -1}$. Then induction hypothesis yields
	the claim except for $i=1$. In particular,
	$\dim M_2=|\supp M_2|=2$. Then
	$\supp M_1$ and $\supp M_2$ do not commute by Lemma~\ref{lem:A2comm}, and
	$|\supp M_1|=2=\dim M_1$
	by Lemma~\ref{lem:A2noncomm}.
\end{proof}

Before we prove Theorem~\ref{thm:C}, we have to study skeletons of type $\gamma_\theta$.

\begin{lem}
	\label{lem:C:conditions}
	Assume that $\charK\ne 2$.
	Let $\theta\in\N_{\geq3}$ and let
	$N\in\cF^G_\theta$. The following are equivalent:
	\begin{enumerate}
		\item $N$ has a skeleton of type $\gamma_\theta$. 
		\item 
			There exists $\epsilon\in Z(G)$
			with $\epsilon^2=1$, $\epsilon \ne 1$,
			and for all $i\in\{1,\dots,\theta\}$ and all
			$s_i\in\supp N_i$ there exists a unique character $\sigma_i$ of $G^{s_i}$
			such that $\supp N_i=\{s_i,\epsilon s_i\}$ for all
			$i\in\{1,\dots,\theta-1\}$, $\supp N_\theta=\{s_\theta\}$ and $N_i\simeq M(s_i,\sigma_i)$
			for all $i\in\{1,\dots,\theta\}$, and  
			the following hold:
			\begin{align}
				\label{eq:C:1}\sigma_i(s_j)\sigma_j(s_i)=1 &
				&&\text{if $|i-j|\ge 2$ and $1\le i,j\le \theta$,}\\
				\label{eq:C:1a}\sigma_i(\epsilon)\sigma_j(\epsilon)=1 &
				&&\text{if $|i-j|\ge 2$, $i,j<\theta$,}\\
				\label{eq:C:3}\sigma_{\theta-1}(s_\theta)\sigma_\theta(s_{\theta-1})=-1,\\
				\label{eq:C:4}\sigma_i(\epsilon s_{i+1}^2)\sigma_{i+1}(\epsilon s_i^2)=1 
				&&&\text{for all $i\in \{1,\dots ,\theta -2\}$},\\
				\label{eq:C:5}\sigma_i(s_i)=-1 &&& \text{for all
				$i\in\{1,\dots,\theta\}$},\\
				\label{eq:C:6}s_is_{i+1}=\epsilon s_{i+1}s_i &&& \text{for all
					$i\in \{1,\dots ,\theta -2\}$},\\
				\label{eq:C:7}s_is_j=s_js_i &&& \text{if $j\ge i+2$ or $j=\theta$}.
			\end{align}
	\end{enumerate}
\end{lem}

\begin{proof}
	We first prove that (2) implies (1). For this, the only non-trivial task is to
	show that the Cartan matrix $A^N$ is of type $C_\theta $.
	Now $a^N_{i\,i+1}=a^N_{i+1\,i}=-1$ for all
	$i\in \{1,\dots ,\theta -2\}$ by
	Corollary~\ref{co:2-2:A2B2CartanMatrix},
	$a^N_{i\theta}=a^N_{\theta i}=0$ for $i\in \{1,\dots ,\theta -2\}$
	by Lemma~\ref{lem:X1}, $a^N_{ij}=a^N_{ji}=0$ for $i,j\in \{1,\dots ,\theta
	-1\}$ with $|i-j|>1$ by Lemma~\ref{lem:a=0for22}, and
	$a^N_{\theta -1\,\theta }=-2$ by Lemmas~\ref{lem:X2} and \ref{lem:X3}.
	Finally, $a^N_{\theta\,\theta-1}=-1$
	by Lemma~\ref{lem:genRosso}. This proves (1).

	Assume now that (1) holds. Then the claims in (2) on $\epsilon $ and
	$\supp N_i$ for $i\in
	\{1,\dots ,\theta \}$ including \eqref{eq:C:6} and \eqref{eq:C:7}
	follow from Lemma~\ref{lem:A3:222}(1). Moreover, $A^N$ is of
	type $C_\theta $, and hence
	Proposition~\ref{pro:absimple} implies that $(\ad N_i)^m(N_j)$ is absolutely
	simple or zero for all $i,j\in \{1,\dots ,\theta \}$ with $i\ne j$.
	Then \eqref{eq:C:3} holds by assumption on the skeleton and
	\eqref{eq:C:1}--\eqref{eq:C:5} follow from
	Lemmas~\ref{lem:a=0for22}, \ref{lem:X1}, \ref{lem:genRosso}, and from
	Corollary~\ref{co:2-2:A2B2CartanMatrix}.
\end{proof}

For the reflections one needs the following lemmas. 

\begin{lem}
	\label{lem:skeleton:C3}
	Let $N\in \cF ^G_3$.
	Assume that $\charK\ne 2$ and that $N$ has a skeleton $\cS $ of type $\gamma_3$.
	Then $\cS $ is a skeleton of $R_k(N)$ for all $k\in \{1,2,3\}$.
\end{lem}

\begin{proof}
	Since $N$ has a skeleton of type $\gamma _3$,
	by Lemma~\ref{lem:C:conditions} there exist
	$r,s,t,\epsilon \in G$ and $\rho \in \chg{G^r}$, $\sigma \in \chg{G^s}$,
	$\tau \in \chg{G}$ as above Lemma~\ref{lem:C3}. Moreover, $\rho ,\sigma ,\tau
	$ satisfy the equations in Lemma~\ref{lem:C:conditions}(2) with $s_1=r$,
	$s_2=s$, $s_3=t$, $\sigma_1=\rho $, $\sigma_2=\sigma $, $\sigma_3=\tau $.
	In particular, $\sigma (t)\tau (s)=\tau (t)=-1$.

	Let $(U,V,W)=R_1(N)$. Then
	Lemma~\ref{lem:HS:reflections} implies
	that $U\simeq M(r^{-1},\rho^*)$, $V\simeq M(rs,\sigma')$ and $W'=W$, where
	$\sigma'\in \chg{G^{rs}}$ with $\sigma'(rs)=-1$ and
	$\sigma'(h)=\rho(h)\sigma(h)$ for all $h\in G^r\cap G^s$. Now we use
	Lemma~\ref{lem:C:conditions} to prove that $\cS $ is a skeleton of $R_1(N)$.
	Lemma~\ref{lem:C:conditions}(2) for $N$, especially Equations~\eqref{eq:C:1} and
	\eqref{eq:C:5}, imply that $\rho (t)\tau (r)=1$ and $\rho (r)=-1$. Hence
	$\rho^*(t)\tau(r^{-1})=1$ and $\rho ^*(r^{-1})=-1$. Further, 
	$\rho^*(\epsilon (rs)^2)\sigma '(\epsilon r^{-2})=1$ by
	Lemma~\ref{lem:HS:reflections}. Finally,
	\[
	\sigma'(t)\tau(rs)=\rho(t)\sigma(t)\tau(r)\tau(s)=-1.
	\]

	Let now $(U',V',W')=R_2(N)$. Lemmas~\ref{lem:HS:reflections} and
	\ref{lem:21reflections}(1) imply that
	$U'\simeq M(sr,\rho')$, $V'\simeq M(s^{-1},\sigma^*)$
	and $W'\simeq M(\epsilon s^2t,\tau')$, where $\rho' \in \chg{G^{sr}}$
	with $\rho'(sr)=-1$, $\rho'(h)=\rho(h)\sigma(h)$ for all $h\in G^r\cap G^s$,
	and $\tau'\in \chg G$ with $\tau'(r)=-\sigma(\epsilon
	r^2)\tau(r)$, $\tau'(z)=\sigma(zr^{-1}zr)\tau(z)$ for all $z\in G^s$.
	Again we use Lemma~\ref{lem:C:conditions} to prove that $\cS $ is a skeleton
	of $R_2(N)$.
	Lemmas~\ref{lem:HS:reflections} and \ref{lem:21reflections}(1) imply
	that $\rho '(sr)=-1$, $\sigma^*(s^{-1})=-1$,
	$\tau '(\epsilon s^2t)=-1$, and
	$$\rho '(\epsilon s^{-2})\sigma ^*(\epsilon (rs)^2)=1,\quad
	\sigma ^*(\epsilon s^2t)\tau _2(s^{-1})=-1.$$
	Finally,
	\begin{align*}
		\rho'(\epsilon s^2t)\tau'(sr)&=\rho(\epsilon s^2t)\sigma(\epsilon
		s^2t)(-\sigma(\epsilon r^2)\tau(r))\sigma(sr^{-1}sr)\tau(s)\\
		&=-\rho(\epsilon s^2)\sigma(\epsilon
		r^2)\rho(t)\tau(r)\sigma(\epsilon^2s^4)\sigma(t)\tau(s)\\
		&=1.
	\end{align*}
	Thus $\cS $ is a skeleton of $R_2(N)$.

	Now let $(U'',V'',W'')=R_3(N)$. Then
	Lemmas~\ref{lem:21reflections}(5) and \ref{lem:genRosso}
	imply that $U''=U$, $V''\simeq M(st, \sigma'')$ and $W''\simeq
	M(t^{-1},\tau^*)$ where $\sigma''\in \chg{G^{st}}$ with
	$\sigma''(z)=\sigma(z)\tau(z)$ for all $z\in G^s$.
	Lemmas~\ref{lem:21reflections}(5) and \ref{lem:genRosso} imply all conditions in
	Lemma~\ref{lem:C:conditions}(2) for $(U'',V'',W'')$ except
	\eqref{eq:C:4} and
	\eqref{eq:C:1} for $i=1$, $j=3$.
	These two we obtain as follows:
	\begin{gather*}
		\rho(\epsilon s^2t^2)\sigma''(\epsilon r^2)
		=\rho(\epsilon s^2)\sigma(\epsilon r^2)\rho(t^2)\tau(\epsilon r^2)=1,\\
	  \rho(t^{-1})\tau^*(r)=\rho(t)^{-1}\tau(r)^{-1}=1.
	\end{gather*}
	Thus $\cS $ is a skeleton of $R_3(N)$.
\end{proof}

\begin{pro}
	\label{pro:skeleton:C}
	Let $\theta\geq3$ and  $N\in\cF_\theta^G$. If
	$N$ has a skeleton $\cS $ of type $\gamma_\theta$, then
	$A^N$ is of type $C_\theta $ and
	$\cS $ is a skeleton of  $R_k(N)$ for all $k\in \{1,\dots ,\theta \}$.
\end{pro}

\begin{proof}
	Proceed as in the proof of Proposition~\ref{pro:skeleton:ADE} and apply
	Lemmas~\ref{lem:skeleton:C3} and \ref{lem:skeleton:A3}.
\end{proof}

We are now ready to prove Theorem~\ref{thm:C}. 

\begin{proof}[Proof of Theorem \ref{thm:C}]
	We prove the implications
	\eqref{it:C:skeleton}$\Rightarrow $\eqref{it:C:NA}$\Rightarrow
	$\eqref{it:C:fgroupoid}$\Rightarrow$\eqref{it:C:skeleton}
	and \eqref{it:C:skeleton}$\Rightarrow
	$\eqref{it:C:sgroupoid}$\Rightarrow$\eqref{it:C:fgroupoid}.

  \eqref{it:C:skeleton}$\Rightarrow $\eqref{it:C:NA}.
	Since $M\in \Ggen ^G_\theta $ has a skeleton of type
	$\gamma _\theta $, Proposition~\ref{pro:skeleton:C} implies
	that $M$ admits all reflections and $\cW (M)$ is standard of type
  $C_\theta $.
  Moreover,
	from Lemma~\ref{lem:C:conditions} we conclude that
	$\NA (M_i)$ is
	finite-dimensional for all $i\in \{1,\dots ,\theta \}$.
	More precisely,
	$$\cH _{\NA (M_i)}(t)=(2)_t^2,\quad \cH_{\NA (M_\theta )}(t)=(2)_t$$
	for all $i\in \{1,\dots ,\theta -1\}$.
	Since the long roots are on the orbit of $\alpha_\theta $ and the short roots
	on the orbit of (any) $\alpha_i$ with $i<\theta $,
	Theorem~\ref{thm:HS}
	implies that
	$\NA (M)$ is finite-dimensional with the claimed Hilbert series.

	\eqref{it:C:NA}$\Rightarrow $\eqref{it:C:fgroupoid}.
  Since $\dim \NA (M)<\infty $, the tuple $M$ admits all
	reflections by \cite[Cor. 3.18]{MR2766176} and the Weyl groupoid is finite by
	\cite[Prop. 3.23]{MR2766176}.

	\eqref{it:C:fgroupoid}$\Rightarrow $\eqref{it:C:skeleton}.
	It is assumed that $M$ admits all
  reflections, $A^M$ is of type $C_\theta $, and $\cW (M)$ is finite. Thus
  $\cC (M)$ is a connected indecomposable finite Cartan graph by
Theorem~\ref{thm:C(M)finiteCartan}.
	For the proof of \eqref{it:C:skeleton} we just have to verify
	the conditions in Lemma~\ref{lem:C:conditions}(2)
	and that $\charK \ne 2$. Now Lemma~\ref{lem:Ctheta:noncomm} tells that
	$\supp M_i=\dim M_i=2$ for all $i\in \{1,\dots ,\theta -1\}$, $\dim M_\theta
	=1$, and $\supp M_i$ and $\supp M_{i+1}$ do not commute
	for $1\le i\le \theta -2$. Moreover, an iterated application of
	Lemma~\ref{lem:reduction} implies that
  $$(\Res^G_H M_{\theta -2},\Res^G_H M_{\theta -1},\Res^G_H M_\theta )\in 
	\Ggen^H_3,$$
	where $H$ is the (non-abelian) subgroup of $G$ generated by $\cup
	_{i=\theta-2}^\theta \supp M_i$. Therefore Lemma~\ref{lem:C3:more}
	implies that $\charK=2$ and that the conditions in 
	Lemma~\ref{lem:C:conditions}(2) hold whenever $i,j\in \{\theta-2,\theta-1,\theta \}$.
	The remaining claims in Lemma~\ref{lem:C:conditions}(2)
	follow from Lemmas~\ref{lem:a=0for22}, \ref{lem:genRosso}, and
Corollary~\ref{co:2-2:A2B2CartanMatrix}.

  \eqref{it:C:skeleton}$\Rightarrow $\eqref{it:C:sgroupoid}.
	Since $M\in \Ggen ^G_\theta $ has a skeleton of type
	$\gamma _\theta $, Proposition~\ref{pro:skeleton:C} implies
	that $M$ admits all reflections and $\cW (M)$ is standard of type
  $C_\theta $.

  \eqref{it:C:sgroupoid}$\Rightarrow $\eqref{it:C:fgroupoid}.
  This is clear, see e.\,g.~\cite[Thm.\,3.3]{MR2498801}.
\end{proof}

\section{Proof of Theorems \ref{thm:B1} and \ref{thm:B2}: The case $B$}
\label{section:B}

In the whole section let $G$ be a non-abelian group. In order to prove
Theorems~\ref{thm:B1} and \ref{thm:B2}, we collect first some information on
skeletons of type $\beta _\theta $, $\beta '_\theta $ and $\beta ''_\theta $ for
$\theta \ge 3$, on tuples in
$\cF^G_\theta $ with such skeletons, and on a
particular Cartan graph.

Extending the definition of a skeleton of type $\beta '_3$ and $\beta ''_3$, we
say that the skeletons in Figure~\ref{fig:beta'} \textbf{are of type} $\beta _\theta '$
and $\beta ''_\theta $, respectively. We will need them for the proof of
Theorem~\ref{thm:B1}. We want to stress that the skeletons of type $\beta
'_\theta $ and $\beta ''_\theta $ are of finite type if and only if $\theta=3$.

\begin{figure}
	\raisebox{8pt}{$\beta '_\theta $}
  \begin{picture}(180,25)
	  \DDvertexone{10}{10}
    \put(9,15){\small $p$}
    \put(13,10){\line(1,0){23}}
	  \put(20,15){\small $p^{-1}$}
	  \DDvertexone{39}{10}
    \put(38,15){\small $p$}
    \put(42,10){\line(1,0){14}}
	  \put(48,15){\small $p^{-1}$}
	  \put(60,7){$\cdots$}
    \put(76,10){\line(1,0){18}}
	  \put(77,15){\small $p^{-1}$}
	  \DDvertexone{97}{10}
    \put(96,15){\small $p$}
    \put(100,10){\line(1,0){23}}
	  \put(107,15){\small $p^{-1}$}
	  \DDvertexone{126}{10}
    \put(125,15){\small $p$}
    \put(129,11){\line(1,0){23}}
    \put(129,9){\line(1,0){23}}
	  \put(137,15){\small $p^{-1}$}
    \put(137,7){$>$}
	  \DDvertexthree{155}{10}
  \end{picture}
	\raisebox{8pt}{$(3)_{-p}=0$}\\
	\raisebox{8pt}{$\beta ''_\theta $}
  \begin{picture}(180,25)
	  \DDvertexone{10}{10}
    \put(13,10){\line(1,0){23}}
    \put(9,15){\small $p$}
	  \put(20,15){\small $p^{-1}$}
	  \DDvertexone{39}{10}
    \put(38,15){\small $p$}
    \put(42,10){\line(1,0){14}}
	  \put(48,15){\small $p^{-1}$}
	  \put(60,7){$\cdots$}
    \put(76,10){\line(1,0){18}}
	  \put(77,15){\small $p^{-1}$}
	  \DDvertexone{97}{10}
    \put(96,15){\small $p$}
    \put(100,9){\line(1,0){26}}
    \put(100,11){\line(1,0){26}}
	  \put(103,15){\small $p^{-1}$}
    \put(107,7){$>$}
	  \DDvertextwo{129}{8}
    \put(120,15){\small $(-p)$}
    \multiput(132,11)(5,0)5{\line(1,0)3}
    \multiput(132,9)(5,0)5{\line(1,0)3}
    \put(140,7){$>$}
	  \DDvertexthree{158}{10}
  \end{picture}
	\raisebox{8pt}{$(3)_{-p}=0$}
  \caption{Skeletons of type $\beta '_\theta $ and $\beta ''_\theta $.}
  \label{fig:beta'}
\end{figure}

For tuples with skeletons of type $\beta _\theta $, $\beta'_\theta $, and
$\beta''_\theta $,
respectively, where $\theta \ge 3$, one obtains the following.

\begin{lem}
	\label{lem:B:conditions}
	Suppose that $\charK=3$. 
	Let $\theta\in\N_{\ge 3}$ and $M\in\cF^G_\theta$.
	The following are equivalent:
	\begin{enumerate}
		\item $M$ has a skeleton of type $\beta_\theta$. 
		\item There exist $\epsilon\in Z(G)$ with $\epsilon^2=1$, $\epsilon \ne 1$,
			and for all
			$i\in\{1,\dots,\theta\}$ and $s_i\in\supp M_i$ a unique character
			$\sigma_i$ of $G^{s_i}$ such that $\supp M_i=\{s_i,\epsilon s_i\}$
			and $M_i\simeq M(s_i,\sigma_i)$ for all $i\in\{1,\dots,\theta\}$,
			and such that the following conditions hold:
			\begin{align}
				\label{eq:B:1}\sigma_i(s_j)\sigma_j(s_i)=\sigma_i(\epsilon
)\sigma_j(\epsilon )=1
				&&&
				\text{for all $i,j$ such that $|i-j|\ge 2$},\\
				\label{eq:B:2}\sigma_i(\epsilon s_{i+1}^2)\sigma_{i+1}(\epsilon s_i^2)=1 &&&
				\text{for all $i\in \{1,\dots ,\theta -1\}$},\\
				\label{eq:B:3}\sigma_i(s_i)=-1 &&& \text{for all
				$i\in\{1,\dots,\theta-1\}$},\\
				\label{eq:B:3a}\sigma_\theta(s_\theta)=1,\\
				\label{eq:B:4}s_is_{i+1}=\epsilon s_{i+1}s_i &&& \text{for all
					$i\in \{1,\dots ,\theta -1\}$},\\
				\label{eq:B:5}s_is_j=s_js_i &&& \text{for all $i,j$ such that $|i-j|\ge
					2$}.
			\end{align}
	\end{enumerate}
\end{lem}

\begin{proof}
	We first prove that (2) implies (1). By
	Definition~\ref{def:skeleton} and the assumptions in (2),
	it only remains to prove that the Cartan matrix $A^M$ is of type $B_\theta $.
	Now $a^M_{i\,i+1}=a^M_{i+1\,i}=-1$ for all
	$i\in \{1,\dots ,\theta -2\}$ by
	Corollary~\ref{co:2-2:A2B2CartanMatrix}(1),
	$a^M_{ij}=a^M_{ji}=0$ for $i,j\in \{1,\dots ,\theta \}$
	with $|i-j|>1$ by Lemma~\ref{lem:a=0for22}, and
	$a^M_{\theta -1\,\theta }=-1$,
	$a^M_{\theta \,\theta -1}=-2$
	by Corollary~\ref{co:2-2:A2B2CartanMatrix}(2).
	This proves (1).

	Assume now that (1) holds. Then the claims in (2) on $\supp M_i$ for all $i\in
	\{1,\dots ,\theta \}$ including \eqref{eq:B:4} and \eqref{eq:B:5}
	follow from Lemma~\ref{lem:A3:222}(1) and the shape of the skeleton of $M$.
	Moreover, $A^M$ is of type $B_\theta $ by (1) and the definition of a
	skeleton.
	Then \eqref{eq:B:1}--\eqref{eq:B:3a} follow from
	Lemmas~\ref{lem:a=0for22} and from
	Corollary~\ref{co:2-2:A2B2CartanMatrix}.
\end{proof}

\begin{lem}
	\label{lem:B':conditions}
	Let $\theta\in\N_{\ge 3}$ and $M\in\cF^G_\theta$.
	The following are equivalent:
	\begin{enumerate}
		\item $M$ has a skeleton of type $\beta '_\theta $, and there exist $t_1,t_2\in
			\supp M_\theta $ such that $t_1t_2\ne t_2t_1$.
		\item 
			Let $s_i\in \supp M_i$ for all $i\in \{1,\dots ,\theta \}$.
			There exists $\epsilon \in G$
			with $\epsilon^3=1$, $\epsilon \ne 1$,
			and unique characters $\sigma_i$ of $G^{s_i}$
			such that $\supp M_i=\{s_i\}$ for all
			$i\in\{1,\dots,\theta-1\}$, $\supp M_\theta=\{s_\theta,\epsilon s_\theta
			,\epsilon ^2s_\theta \}$ and $M_i\simeq M(s_i,\sigma_i)$
			for all $i\in\{1,\dots,\theta\}$, and  
			the following hold:
			\begin{align}
				\label{eq:B':1}\sigma_i(s_j)\sigma_j(s_i)=1 &
				&&\text{if $|i-j|\ge 2$,}\\
				\label{eq:B':2}\sigma_i(s_{i+1})\sigma_{i+1}(s_i)=p^{-1} &
				&&\text{for all $i\in \{1,\dots ,\theta -1\}$,}\\
				\label{eq:B':3}\sigma_i(s_i)=p &&& \text{for all
				$i\in\{1,\dots,\theta-1\}$},\\
				\label{eq:B':4}\sigma_\theta (s_\theta )=-1,\\
				\label{eq:B':5}\epsilon s_\theta =s_\theta \epsilon ^{-1},
			\end{align}
			where $p\in \K$ with $1-p+p^2=0$.
	\end{enumerate}
\end{lem}

\begin{proof}
	We first prove that (2) implies (1). According to
	Definition~\ref{def:skeleton} and the assumptions in (2),
	it only remains to prove that the Cartan matrix $A^M$ is of type $B_\theta $.
	Now $a^M_{i\,i+1}=a^M_{i+1\,i}=-1$ and
	$a^M_{ij}=0$ for all $i\in \{1,\dots ,\theta -2\}$ and all $j>i+1$ with $j\ne
	\theta $ by Lemma~\ref{lem:Rosso}.
	Further, $a^M_{\theta -1\,\theta}=-1$ and
	$a^M_{i\theta }=0$ (and hence $a^M_{\theta i}=0$) for all $i<\theta-1$
	by Lemma~\ref{lem:genRosso}.
    Finally $a^M_{\theta \,\theta -1}=-2$ because of Lemma~\ref{lem:1+3=typeB2}.
	This proves (1).

	Assume now that (1) holds. In particular, $A^M$ is of type $B_\theta $
	by the definition of a skeleton and a skeleton of type $\beta '_\theta $.
	Let $s_\theta \in \supp M_\theta $.
	Since $|\supp M_\theta |=3$, (1) and Lemma~\ref{lem:3} imply
	that there exists $\epsilon \in G$
	such that $\epsilon ^3=1$, $\epsilon \ne 1$,
  $\epsilon s_\theta =s_\theta \epsilon ^{-1}$, and
	$\supp M_\theta =\{s_\theta ,\epsilon s_\theta,\epsilon ^2s_\theta \}$.
	Then \eqref{eq:B':1} follows from Lemma~\ref{lem:genRosso}, since $a^M_{ij}=0$
	whenever $|i-j|\ge 2$. Equations~\eqref{eq:B':2} and \eqref{eq:B':3} are given
	in the skeleton. Since $a^M_{\theta\,\theta-1}=-2$, \eqref{eq:B':4}
follows from Lemma~\ref{lem:1+3=typeB2}.
\end{proof}

\begin{lem}
	\label{lem:B'':conditions}
	Let $M\in\cF^G_\theta $. The
	following are equivalent:
	\begin{enumerate}
		\item $M$ has a skeleton of type $\beta ''_\theta $.
		\item 
			Let $s_i\in \supp M_i$ for all $i\in \{1,\dots ,\theta \}$.
			There exists $\epsilon \in G$
			with $\epsilon^3=1$, $\epsilon \ne 1$,
			and $\sigma_i\in \chg{G^{s_i}}$ for all $i\in \{1,\dots ,\theta \}$
			such that $\supp M_i=\{s_i\}$ for all
			$i\in\{1,\dots,\theta-2\}$, $\supp M_{\theta -1}=\{s_{\theta -1},\epsilon
			s_{\theta -1}\}$,
			$\supp M_\theta=\{s_\theta,\epsilon s_\theta
			,\epsilon ^2s_\theta \}$, $M_i\simeq M(s_i,\sigma_i)$
			for all $i\in\{1,\dots,\theta\}$, and  
			the following hold:
			\begin{align}
				\label{eq:B'':1}\sigma_i(s_j)\sigma_j(s_i)=1 &
				&&\text{if $|i-j|\ge 2$, $i,j\le \theta $,}\\
				\label{eq:B'':2}\sigma_i(s_{i+1})\sigma_{i+1}(s_i)=p^{-1} &
				&&\text{for all $i\in \{1,\dots ,\theta -2\}$,}\\
				\label{eq:B'':3}\sigma_i(s_i)=p &&& \text{for all
				$i\in\{1,\dots,\theta-2\}$},\\
				\label{eq:B'':4}\sigma_{\theta -1}(s_{\theta -1})=\sigma_\theta (s_\theta )=-1,\\
				\label{eq:B'':4a}\sigma_{\theta-1}(\epsilon )=-p, &&&\\
				\label{eq:B'':5}\sigma_{\theta-1}(\epsilon s_\theta ^2)\sigma_\theta
				(\epsilon s_{\theta -1}^2)=1, &&&\\
				\label{eq:B'':7}s_\theta s_{\theta -1}=\epsilon s_{\theta -1}s_\theta ,\\
				\label{eq:B'':8}\epsilon s_\theta =s_\theta \epsilon ^{-1},
			\end{align}
			where $p\in \K$ with $1-p+p^2=0$.
	\end{enumerate}
\end{lem}

\begin{proof}
	Again we first prove that (2) implies (1). According to
	Definition~\ref{def:skeleton} and the assumptions in (2),
	it only remains to prove that the off-diagonal entries of $A^M$
	correspond to the integers obtained from the skeleton of type $\beta ''_\theta
	$.
	Now $a^M_{i\,i+1}=a^M_{i+1\,i}=-1$
	for all $i\in \{1,\dots ,\theta -3\}$
	by Lemma~\ref{lem:Rosso}
	and $a^M_{ij}=0$ for all $1\le i,j\le \theta $ with $j\ge i+2$
	by Lemma~\ref{lem:genRosso}.
	Also, $a^M_{\theta -2\,\theta -1}=-1$	by Lemma~\ref{lem:genRosso}.
	Moreover, $a^M_{\theta -1\,\theta -2}=-2$	by Lemmas~\ref{lem:X2} and
	\ref{lem:X3}(1).
	Finally, $a^M_{\theta -1\,\theta }=-1$ and $a^M_{\theta \,\theta -1}=-2$
	because of Lemma~\ref{lem:2+3=typeB2}.
	This proves (1).

	Assume now that (1) holds. 
	Since $s_{\theta -1}^G$ and $s_\theta ^G$ do not commute and since
	$|s_{\theta -1}^G|=2$, we obtain that
	$s_\theta s_{\theta -1}\ne s_{\theta -1}s_\theta $. Let $\epsilon \in G$
	be such
	that $s_{\theta -1}^G=\{s_{\theta -1},\epsilon s_{\theta -1}\}$. Then
	$\epsilon ^3=1$, $\supp M_\theta =\{s_\theta ,\epsilon s_\theta
	,\epsilon^2s_\theta \}$,
	and \eqref{eq:B'':7}, \eqref{eq:B'':8} hold by Lemma~\ref{lem:2+3}.
	It remains to prove \eqref{eq:B'':1}--\eqref{eq:B'':5}.

	Now \eqref{eq:B'':1} follows from Lemma~\ref{lem:genRosso}, since $a^M_{ij}=0$
	whenever $1\le i<j-1$.
	By Proposition~\ref{pro:absimple}, all $(\ad M_i)^m(M_j)$ for $i\ne j$, $m\ge
	0$, are absolutely simple or zero because of (1).
	Since $a^M_{\theta\,\theta-1}=-1$ and $a^M_{\theta ,\theta -1}=-2$,
	\eqref{eq:B'':4} and \eqref{eq:B'':5} follow from Lemma~\ref{lem:2+3=typeB2}. 
	Finally, Conditions~\eqref{eq:B'':2}, \eqref{eq:B'':3}, and \eqref{eq:B'':4a}
	are given in the skeleton.
\end{proof}

In the following three propositions we study reflections of skeletons of type
$\beta _\theta $, $\beta'_\theta $, and $\beta''_\theta $ with $\theta \ge 3$.

\begin{pro}
	\label{pro:skeleton:B}
	Let $\theta \in \N $ with $\theta \ge 3$ and let $M\in\cF _\theta ^G$.
	Assume that $M$ has a skeleton $\cS $ of type $\beta _\theta $.
	Then the Cartan matrix of $M$ is of type $B_\theta $, and
	$\cS $ is a skeleton of $R_k(M)$ for all $k\in \{1,\dots ,\theta \}$.
\end{pro}

\begin{proof}
	Following the arguments in the proof of Proposition~\ref{pro:skeleton:ADE}
	and using	Lemma~\ref{lem:skeleton:A3}, it suffices to prove the claim for
	$\theta =3$. In this case, one obtains the claim following
	the proof of Lemma~\ref{lem:skeleton:A3} and using
	Lemma~\ref{lem:B:conditions}.
\end{proof}

\begin{pro}
	\label{pro:skeleton:B'}
	Let $M\in\cF_\theta ^G$.
	Assume that $M$ has a skeleton $\cS $ of type $\beta '_\theta $.
	Then $\cS $ is a skeleton of $R_k(M)$ for $1\le k\le \theta -1$,
	and	$R_\theta (M)$ has a skeleton of type $\beta ''_\theta $.
\end{pro}

\begin{proof}
	By Remark~\ref{rem:connectedtriples}, it is enough to consider connected
	subgraphs of $\cS$ with three vertices $i_1,i_2,i_3$. If $\theta \notin
	\{i_1,i_2,i_3\}$ and $k\in \{i_1,i_2,i_3\}$,
	then Lemma~\ref{lem:B':conditions} implies that
	$M_{i_1}\oplus M_{i_2}\oplus M_{i_3}$ is a braided
	vector space of Cartan type with Cartan matrix of type $A_3$,
	and hence the tuple $R_j(M_{i_1},M_{i_2},M_{i_3})$ for $j\in \{1,2,3\}$
    has the same skeleton as $(M_{i_1},M_{i_2},M_{i_3})$.
	Thus it remains to prove the proposition for $\theta =3$ and $k\in \{1,2,3\}$.

	Assume first that $k=1$. Then $\dim M_k=1$, $a^M_{12}=-1$, and $a^M_{13}=0$.
	Hence $R_1(M)_1=M_1^*$, $R_1(M)_2\simeq M_1\otimes M_2$ by
	Lemma~\ref{lem:Rosso}, and $R_1(M)_3=M_3$.
	We now verify the conditions in Lemma~\ref{lem:B':conditions} for $R_1(M)$.
	The only non-trivial condition is \eqref{eq:B':2} for $i=2$.
	For this we obtain that
	\begin{gather*}
		\sigma_1\sigma_2(s_3)\sigma_3(s_1s_2)=
		\sigma_1(s_3)\sigma_3(s_1)\sigma_2(s_3)\sigma_3(s_2)=p^{-1},
	\end{gather*}
	and hence $\cS $ is a skeleton of $R_1(M)$.
	
	Assume now that $k=2$. Then $\dim M_k=1$ and $a^M_{21}=a^M_{23}=-1$.
	Hence $R_2(M)_1\simeq M_2\otimes M_1$, $R_2(M)_2\simeq M_2^*$, and
	$R_2(M)_3\simeq M_2\otimes M_3$ by Lemma~\ref{lem:genRosso}.
	We verify the conditions in Lemma~\ref{lem:B':conditions} for $R_2(M)$.
	We obtain that
	\begin{gather*}
		\sigma_1\sigma_2(s_2s_3)\sigma_2\sigma_3(s_2s_1)
		=\sigma_1(s_3)\sigma_3(s_1)\sigma_1(s_2)\sigma_2(s_1s_2^2s_3)\sigma_3(s_2)
		=p^{-1}p^2p^{-1}=1,\\
		\sigma_1\sigma_2(s_2^{-1})\sigma_2^*(s_2s_1)
		=(\sigma_1(s_2)\sigma_2(s_1))^{-1}\sigma_2(s_2)^{-2}=pp^{-2}=p^{-1},\\
		\sigma_2^*(s_2s_3)\sigma_2\sigma_3(s_2^{-1})
		=(\sigma_2(s_3)\sigma_3(s_2))^{-1}\sigma_2(s_2)^{-2}=pp^{-2}=p^{-1},\\
		\sigma_1\sigma_2(s_1s_2)=pp^{-1}p=p,\\
		\sigma_2^*(s_2^{-1})=p,\quad
		\sigma_2\sigma_3(s_2s_3)=pp^{-1}\sigma_3(s_3)=\sigma_3(s_3).
	\end{gather*}
	Condition~\eqref{eq:B':5} for $R_2(M)$ is clear.
	Therefore $\cS $ is a skeleton of $R_2(M)$.

	Finally, assume that $k=3$. Then $R_3(M)=(M_1,(\ad M_3)^2(M_2),M_3^*)$.
	We have to show that $R_3(M)$ has a skeleton of type $\beta''_3$.
	To do so we apply Lemma~\ref{lem:B'':conditions}.
	By Proposition~\ref{pro:R2:1+3}, $R_3(M)_2\simeq M(s',\sigma ')$
	and $M_3^*\simeq M(s_3^{-1},\sigma_3^*)$, where $s'=\epsilon s_2s_3^2$,
	$\sigma '(\epsilon )=p^{-2}=-p$ (which proves \eqref{eq:B'':4a}),
	and $\sigma'(h)=\tau(h)^2\sigma(h)$
	for all $h\in G^\epsilon \cap G^{s_3}$. Now Conditions~\eqref{eq:B'':1},
	\eqref{eq:B'':3}, \eqref{eq:B'':7} and \eqref{eq:B'':8} are clear. Moreover,
	\eqref{eq:B'':4} and \eqref{eq:B'':5} follow from the last claim of
	Proposition~\ref{pro:R2:1+3}. We verify now \eqref{eq:B'':2}:
	$$ \sigma _1(s')\sigma '(s_1)=\sigma_1(\epsilon s_2s_3^2)\sigma_2(s_1)\sigma_3
	(s_1)^2=\sigma_1(s_2)\sigma_2(s_1)(\sigma_1(s_3)\sigma_3(s_1))^2=p^{-1}$$
	and the proof is completed.
\end{proof}

For the proof of the third of three propositions we will use the following
lemma, which will also play a role in the proof of
Proposition~\ref{pro:skeleton:F4}.

\begin{lem} \label{lem:skeleton:112}
	Let $M\in \cF_3^G$. Assume that $M$ has a skeleton $\cS $ as in
	Figure~\ref{fig:skeleton:112}, where $p=-1$ if $q^2=1$.
	Then $\cS $ is a skeleton of $R_k(M)$ for all $k\in \{1,2,3\}$.
\end{lem}

\begin{figure}
  \begin{picture}(80,25)
	  \DDvertexone{10}{10}
    \put(9,15){\small $p$}
    \put(13,10){\line(1,0){23}}
	  \put(20,15){\small $p^{-1}$}
	  \DDvertexone{39}{10}
    \put(38,15){\small $p$}
    \put(42,11){\line(1,0){28}}
    \put(42,9){\line(1,0){28}}
	  \put(50,15){\small $p^{-1}$}
    \put(50,7){$>$}
	  \DDvertextwo{73}{8}
    \put(66,17){\small $(q)$}
  \end{picture}
	\caption{The skeleton in Lemma~\ref{lem:skeleton:112}.}
	\label{fig:skeleton:112}
\end{figure}
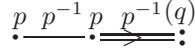

\begin{proof}
	By assumption, there exist $r,t\in Z(G)$, $s,\epsilon \in G$
	and $\rho ,\tau \in \chg G$, $\sigma \in \chg{G^s}$ such that
	$M_1\simeq M(r,\rho )$, $M_2\simeq M(t,\tau )$, $M_3\simeq M(s,\sigma)$,
	$s^G=\{s,\epsilon s\}$, and $\epsilon \ne 1$.
	By Lemma~\ref{lem:rs}, there exists $x\in G$ such that $xs=\epsilon sx$ and
	$x\epsilon =\epsilon ^{-1}x$. The skeleton contains additionally the following
	information, see Lemmas~\ref{lem:Rosso} and \ref{lem:X1}:
	\begin{align*}
		\rho (r)&=p,& \tau (t)&=p,& \sigma (\epsilon )&=q,\\
		\rho (t)\tau (r)&=p^{-1},&\rho (s)\sigma (r)&=1,&\tau (s)\sigma (t)&=p^{-1},
	\end{align*}
	and that $a^M_{32}=-2$.
	Since $a^M_{32}=-2$, Lemma~\ref{lem:X1} implies that $p\ne 1$.
	Since $a^M_{ij}a^M_{ji}\in \{0,1,2\}$ for all $i,j\in \{1,2,3\}$ with $i\ne j$,
	Proposition~\ref{pro:absimple} and Lemmas~\ref{lem:X2}, \ref{lem:X3} imply
	that the relations in one of the following three lines hold:
	\begin{gather*}
		\sigma (s)=-1,\quad \sigma (\epsilon ^2)\ne 1,\quad \sigma (\epsilon
		^2t^2)\tau (s^2)=1,\\
		\sigma (s)=-1,\quad \sigma (\epsilon ^2)=1,\\
		\sigma (s)\ne -1,\quad \sigma (\epsilon ^2)\ne 1,\quad \sigma (st)\tau
		(s)=1,\quad \sigma (\epsilon ^2s^2)=1.
	\end{gather*}
	Since $\sigma (t)\tau (st)=1$, we conclude that
	$(M_3,M_2)\in \wp _5\cup \wp _1\cup \wp _7$.

	Let now $(U,V,W)=R_1(M)$. Then $U\simeq M(r^{-1},\rho ^*)$, $V\simeq M(rt,\rho
	\tau )$, and $W=M_3\simeq M(s,\sigma )$. In particular,
	$$ \rho \tau (s)\sigma (rt)=\tau (s)\sigma (t)=p^{-1}.$$
  Using the above formulas and the definition of a skeleton,
	we conclude that $\cS $ is a skeleton of $R_1(M)$.

	Let $(U',V',W')=R_2(M)$. Then $U'\simeq M(tr,\tau \rho )$, $V'\simeq
	M(t^{-1},\tau ^*)$, and $W'\simeq M(ts,\tau \sigma )$. Then
	$$ \tau \rho (ts)\tau \sigma (tr)=\tau (t^2)\tau (s)\sigma (t)\rho (t)\tau
	(r)\rho (s)\sigma (r)=p^2p^{-1}p^{-1}=1 $$
	and $\tau \sigma (\epsilon ts)=\tau (ts)\sigma (t)\sigma (\epsilon s)=q$.
	Then Lemma~\ref{lem:21reflections}(5) implies that $\cS $ is a skeleton of
	$R_2(M)$.

	Now let $(U'',V'',W'')=R_3(M)$. Then $U''=M_1\simeq M(r,\rho )$, $V''\simeq
	M(\epsilon s^2t,\tau _2)$, and $W''\simeq M(s^{-1},\sigma ^*)$, where
	$\tau_2\in \chg G$ as in Lemma~\ref{lem:X2}.
    We record that $(s^{-1})^G=\{s^{-1},\epsilon ^{-1}s^{-1}\}$ and that
$\sigma^*(\epsilon^{-1})=\sigma(\epsilon )$. Moreover,
	$$\tau _2(r)\rho (\epsilon s^2t)=\sigma (r^2)\tau(r)\rho (s^2t)=\rho (t)\tau
	(r)=p^{-1}.$$
	Thus, if $(M_3,M_2)\in
	\wp_5$, $\wp _1$, and $\wp_7$, respectively,
	then Lemma~\ref{lem:21reflections}(2), (1), and (3), respectively, implies that
  $\cS $ is a skeleton of $R_3(M)$. Here, in the case of $\sigma (\epsilon^2)=1$
	we used (and needed) that $p=-1$ in order to identify $\cS $ as a
	skeleton of $R_3(M)$.
	This completes the proof.
\end{proof}

\begin{pro}
	\label{pro:skeleton:B''}
	Let $M\in\cF_\theta ^G$.
	Assume that $M$ has a skeleton $\cS $ of type $\beta ''_\theta $.
	Then $\cS $ is a skeleton of $R_k(M)$ for $1\le k\le \theta -1$,
	and	$R_\theta (M)$ has a skeleton of type $\beta '_\theta $.
\end{pro}

\begin{proof}
	By Remark~\ref{rem:connectedtriples}, it is enough to consider connected
	subgraphs of $\cS$ with three vertices $i_1,i_2,i_3$ and their reflections.
	If $i_1,i_2,i_3\le \theta-2$, then $M_{i_1}\oplus M_{i_2}\oplus M_{i_3}$ is a
	braided vector space of Cartan type and their reflections have the same
	skeleton. If $\{i_1,i_2,i_3\}=\{\theta -3,\theta -2,\theta -1\}$, then the
	reflections of $M_{i_1}\oplus M_{i_2}\oplus M_{i_3}$ have the same skeleton as
	$M_{i_1}\oplus M_{i_2}\oplus M_{i_3}$. Indeed,
	$(3)_{-p}=0$ by assumption and hence $p\ne 1$. Therefore $(-p)^2=1$ implies
that $p=-1$ and hence Lemma~\ref{lem:skeleton:112} applies.

	We are left to determine the skeleton of
	$R_k(M_{\theta -2},M_{\theta -1},M_\theta )$ for all $k\in \{1,2,3\}$, that is,
	to prove the claim for $\theta =3$. To do so, assume that $\theta=3$, and let
	$s_1,s_2,s_3,\epsilon \in G$ and
	$\sigma_1,\sigma_2,\sigma_3$ as in Lemma~\ref{lem:B'':conditions}.

	Let $(U,V,W)=R_1(M)$. Then $U\simeq M(s_1^{-1},\sigma^*)$, $V\simeq
	M(s_1s_2,\sigma_1\sigma_2)$, and $W\simeq M(s_3,\sigma_3)$ by
	Lemma~\ref{lem:genRosso}. Then
	$$ \sigma_1\sigma_2(\epsilon s_3^2)\sigma_3(\epsilon s_1^2s_2^2)
	=(\sigma_1(s_3)\sigma_3(s_1))^2\sigma_2(\epsilon s_3^2)\sigma_3(\epsilon
	s_2^2)=1
	$$
	and hence $\cS $ is a skeleton of $R_1(M)$ by Lemmas~\ref{lem:B'':conditions}
	and \ref{lem:21reflections}(5).

	Let $(U',V',W')=R_2(M)$. Then
	$$ V'\simeq M(s_2^{-1},\sigma^*),\quad
    (s_2^{-1})^G=\{s_2^{-1},\epsilon^{-1}s_2^{-1}\},$$
	and $U'\simeq M(\epsilon s_2^2s_1,\rho ')$ for some $\rho '\in \chg G$ by Lemmas~\ref{lem:X2},
	\ref{lem:X3}. Moreover, the skeleton of $(M_1,M_2)$ is a skeleton of $(U',V')$
	by Lemma~\ref{lem:21reflections}(1) (if $p\ne -1$) and by
	Lemma~\ref{lem:21reflections}(2) (if $p=-1$), since $\sigma_2(s_2)=-1$.
	Further, since $(3)_{\sigma_2(\epsilon )}=0$ by \eqref{eq:B'':4a}, we conclude
	from Lemma~\ref{lem:2+3=typeB2} that $W'\simeq M(\epsilon^{-1}s_2s_3,\tau ')$,
	where $\tau '\in \chg{G^{s_3}}$ with $\tau '(s_3)=\sigma_3(\epsilon
	s_2^{-1})\sigma_2(\epsilon)$ and $\tau'(h)=\sigma_2(h)\sigma_3(h)$ for all
	$h\in G^{s_2}\cap G^{s_3}$. Then
	\begin{gather*}
	  \sigma_2^*(s_2^{-1})=\sigma_2(s_2)=-1,\\
      \tau'(\epsilon^{-1}s_2s_3)=\sigma_2(\epsilon^{-1}s_2)\sigma_3(\epsilon^{-1}s_2)
      \sigma_3(\epsilon s_2^{-1})\sigma_2(\epsilon )=\sigma_2(s_2)=-1,\\
	  \sigma_2^*(\epsilon^{-1})=\sigma_2(\epsilon )=-p,\\
      \sigma_2^*(\epsilon ^{-1} (\epsilon ^{-1}s_2s_3)^2)
      \tau '(\epsilon ^{-1}s_2^{-2})
      =\sigma_2(s_2^2s_3^2)^{-1}\sigma_2\sigma_3(\epsilon s_2^2)^{-1}=1,\\
      \epsilon ^{-1}s_2s_3 \,s_2^{-1}=\epsilon ^{-2}s_3s_2s_2^{-1}
      =\epsilon ^{-1} s_2^{-1}(\epsilon^{-1}s_2s_3).
	\end{gather*}
    Therefore $\cS $ is a skeleton of $R_2(M)$ by
    Lemma~\ref{lem:B'':conditions}.

	Let $(U'',V'',W'')=R_3(M)$. Then $U''=M_1$ and
    $W''\simeq M(s_3^{-1},\sigma_3^*)$. Lemma~\ref{lem:2+3=typeB2} implies that
	$V''\simeq M(\epsilon ^{-1}s_3^2s_2,\sigma '')$, where $\sigma''\in \chg G$
    such that
		$$ \sigma''(\epsilon )=1,\quad \sigma''(s_3)=-\sigma_3(\epsilon
		s_2^{-1})\sigma_2(\epsilon),\quad \sigma ''(h)=\sigma_3(h)^2\sigma_2(h)
    $$
		for all $h\in G^{s_2}\cap G^{s_3}$. Now we verify the conditions in
		Lemma~\ref{lem:B':conditions} for $R_3(M)\in \cF^G_3$. Except
		\eqref{eq:B':3} for $i=2$ and except \eqref{eq:B':3}, everything is clear or
		can be seen directly. Since $\epsilon^{-1}s_2\in Z(G)$ by
		Lemma~\ref{lem:2+3}, for \eqref{eq:B':3}, $i=2$, we obtain that
		$$ \sigma ''(\epsilon^{-1}s_3^2s_2)=\sigma_3(\epsilon
		s_2^{-1})^2\sigma_2(\epsilon)^2
		\sigma_3(\epsilon^{-1}s_2)^2\sigma_2(\epsilon^{-1}s_2)=\sigma_2(\epsilon
		s_2)=p.$$
		Finally, for \eqref{eq:B':2} we calculate the following:
		\begin{gather*}
			\sigma_1(\epsilon^{-1}s_3^2s_2)\sigma''(s_1)=\sigma_1(s_3)^2\sigma_1(s_2)
			\sigma_3(s_1)^2\sigma_2(s_1)=p^{-1},\\
			\sigma ''(s_3^{-1})\sigma_3^*(\epsilon ^{-1}s_3^2s_2)
			=-\sigma_3(\epsilon^{-1}s_2)\sigma_2(\epsilon^{-1})\sigma_3(\epsilon
			s_2^{-1})=-\sigma_2(\epsilon^{-1})=p^{-1}.
		\end{gather*}
		Thus $R_3(M)$ has a skeleton of type $\beta'_3$.
		This completes the proof of the proposition.
\end{proof}

Before proving Theorem~\ref{thm:B1} we also need more information on the
finite Cartan graph in Lemma~\ref{lem:noA}\eqref{it:noA:4}.

\begin{lem} \label{lem:CGtwopoints}
	Let $\cC =\cC (I,\cX ,r,A)$ be the Cartan graph with $I=\{1,2,3\}$, $\cX
	=\{X,Y\}$, such that $r_1=r_2=\id $, $r_3$ is the transposition $(X\,Y)$
	and
	$$
	  A^X=\begin{pmatrix} 2 & -1 & 0 \\ -1 & 2 & -1\\ 0 & -2 & 2 \end{pmatrix},
		\quad
	  A^Y=\begin{pmatrix} 2 & -1 & 0 \\ -2 & 2 & -1\\ 0 & -2 & 2 \end{pmatrix}.
	$$
	Let $W_0\subset W(\cC )$ be the automorphism group of $X$. Then
	\begin{align*}
		\roots ^X_+=&\;\{1,2,3,12,23,123,23^2,123^2,12^23^2,12^23^3,12^23^4,12^33^4,1^22^33^4\},\\
		\roots ^Y_+=&\;\{1,2,3,12,23,12^2,123,23^2,12^23,123^2,12^23^2,12^33^2,1^22^33^2\},
	\end{align*}
  and the orbits of $\roots ^X$ with respect to the action of $W_0$ are
	\begin{gather*}
		\{\pm 1,\pm 2,\pm 12,\pm 12^23^4,\pm 12^33^4,\pm 1^22^33^4 \},\\
		\{\pm 3,\pm 23,\pm 123,\pm 12^23^3\},\quad
		\{\pm 23^2,\pm 123^2, \pm 12^23^2\},
	\end{gather*}
	where $1^a2^b3^c$ and $-1^a2^b3^c$ mean $a\alpha_1+b\alpha_2+c\alpha_3$
	and $-a\alpha_1-b\alpha_2-c\alpha_3$, respectively, for all $a,b,c\in \Z $.
\end{lem}

\begin{proof}
	It is clear from the definition that $\cC $ is a semi-Cartan graph. It is a
	Cartan graph by \cite[Thm.~5.4]{MR2498801}. The root system with number 14 in
	\cite[Appendix A]{MR2869179}, where one interchanges $\alpha_1$ and
	$\alpha_2$, has $\cC $ as a Cartan graph and corresponds to the point $Y$, see
	also the proof of Lemma~\ref{lem:noA}. {}From this one obtains easily the set
	$\roots ^X=s_3^Y(\roots ^Y)$.

	By the proof of \cite[Thm.~5.4]{MR2498801}, see \cite[Eqs.\,(5.4),(5.5)]{MR2498801},
	$W_0$ is generated as a group by
	$s_1^X$, $s_2^X$, and $t=s_3s_2s_3^X$. (Observe that in \cite{MR2498801} the
	role of $1$ and $3$ in $I$ are interchanged.) We record that
	$$ t(\alpha_1)=\alpha_1+2\alpha_2+4\alpha_3,\quad
	t(\alpha_2)=\alpha_2,\quad t(\alpha_3)=-(\alpha_2+\alpha_3).$$
	Applying successively these generators of $W_0$ to
	the elements of $\roots ^X$ one obtains the last claim of the lemma.
\end{proof}

Now we are able to prove Theorems~\ref{thm:B1} and \ref{thm:B2}. 

\begin{proof}[Proof of Theorem~\ref{thm:B1}: $\dim M_1=1$] \

	(1)$\Rightarrow $(3). Since $\theta =3$ and $M$ has a skeleton of type
	$\beta'_3$, Propositions~\ref{pro:skeleton:B'} and \ref{pro:skeleton:B''} imply
	that $M$ admits all reflections and the skeletons of $M$ and of $R_3(M)$ form
	the points of the semi-Cartan graph $\cC $ in Lemma~\ref{lem:CGtwopoints}.
	This semi-Cartan graph is a finite Cartan graph, and the positive roots of its
	points are given in Lemma~\ref{lem:CGtwopoints}. Since $M$ has a skeleton of
	type $\beta '_3$, Lemma~\ref{lem:B':conditions} implies that $\NA (M_i)$ is
	finite-dimensional for all $i\in \{1,2,3\}$. More precisely,
	$$\cH _{\NA (M_1)}(t)=\cH_{\NA (M_2)}(t)=(h)_t,\quad \cH_{\NA
    (M_3)}(t)=(2)_t^2(3)_t,$$
	where $h=3$ if $\charK=2$, $h=2$ if $\charK=3$, and $h=6$ otherwise.
	Similarly, Lemma~\ref{lem:B'':conditions} implies that
	$R_3(M)_2$ is a braided vector space of diagonal type with braiding matrix
	$$ \begin{pmatrix} -1 & -\zeta \\ -\zeta & -1 \end{pmatrix} $$
	where $\zeta =\sigma (\epsilon )$ in the notation of
	Lemma~\ref{lem:B'':conditions}. Therefore
	$$ \cH _{\NA (R_3(M)_2)}(t)=(2)_t(h')_t=
      \begin{cases}
        (2)_t^2 & \text{if $\charK =3$,}\\
        (2)_t^2(3)_{t^2} & \text{if $\charK \ne 3$,}
      \end{cases}
    $$
	where $h'=6$ if $\charK\ne 3$ and $h'=2$ if $\charK=3$.
	Now Theorem~\ref{thm:HS}, using the decomposition of $\roots _+^X$
	into $W_0$-orbits in Lemma~\ref{lem:CGtwopoints}, implies that
	$\NA (M)$ is finite-dimensional with the claimed Hilbert series.

	(3)$\Rightarrow $(2).
    Since $\dim \NA (M)<\infty $, the tuple $M$ admits all
	reflections by \cite[Cor. 3.18]{MR2766176} and the Weyl groupoid is finite by
	\cite[Prop. 3.23]{MR2766176}.

	(2)$\Rightarrow $(1). It is assumed that $\dim M_1=1$, $M$ admits all
    reflections, $A^M$ is of type $B_\theta $, and $\cW (M)$ is finite. Thus
		Theorem~\ref{thm:C(M)finiteCartan} tells that $\cC (M)$ is a connected
		indecomposable finite Cartan graph.

    Assume first that $\theta =3$. If $\cC (M)$ has a point with a Cartan matrix
    of type $A_3$ or $C_3$, then $M$ is standard of type $A_3$ and $C_3$,
    respectively, by Theorems~\ref{thm:ADE} and \ref{thm:C}.
    Since the Cartan matrix $A^M$ is of type $B_3$, from
    Corollary~\ref{cor:noAC} we conclude that either $M$ is standard of
    type $B_3$ or each point of $\cC (M)$ has one of the two Cartan matrices
		in Lemma~\ref{lem:noA}\eqref{it:noA:4}.

    Since $\dim M_1=1$, Lemma~\ref{lem:A2comm} implies that $\dim M_2=1$.
    Let $H$ be the subgroup generated by $\supp M_2\cup \supp M_3$.
    Then $H$ is non-abelian,
    $M'=(\Res ^G_H M_2,\Res ^G_H M_3)\in \Ggen_2^H$, $M'$ admits all
    reflections, and $\cW (M')$ is standard of type $B_2$ because of
    Corollary~\ref{cor:noAC}. Now \cite[Thm.\,2.1,\,Table\,1]{rank2},
		especially the claim on the support of $M'$,
    imply immediately that $\supp M_3$ is non-abelian and
		$|\supp M_3|\in \{3,4\}$. Moreover, the only possible example with
			$|\supp M_3|=4$ would be \cite[Ex.\,1.7]{rank2}.
		However, this example has a root system which is standard of type $G_2$, and
	hence a Cartan matrix of type $B_2$ is impossible if $|\supp M_3|=4$.
    On the other hand, $M'$ being standard implies that $M'\notin \wp _5$ in the
    notation of \cite[7.1,8.4]{rank2}.
    The only remaining possibility is discussed in \cite[Thm.\,8.2]{rank2}:
    There exist $r,s\in Z(G)$, $t,\epsilon \in G$, characters $\rho ,\sigma $ of
    $G$ and $\tau $ of $G^t$ such that
    $$ M_1\simeq M(r,\rho ),\quad M_2\simeq M(s,\sigma ),\quad M_3\simeq
    M(t,\tau ),$$
    and $G$ is generated by $r,s,t,\epsilon $, the relations
    $t\epsilon =\epsilon ^{-1}t$ and $\epsilon ^3=1$ hold in $G$, and
    \begin{align}
      (3)_{-\sigma (s)}=0,\quad \sigma (st)\tau (s)=1,\quad \tau (t)=-1.
    \end{align}
		Moreover, the condition $a^M_{13}=0$ is equivalent to $\rho (t)\tau (r)=1$.

		Both if $M$ is standard and if $\roots ^M_+$ is the root system of $X$ in
		Lemma~\ref{lem:CGtwopoints}, we obtain that 
		$$\roots ^M_+=\roots ^{R_1(M)}_+=\roots ^{R_2(M)}_+,\quad
		A^M=A^{R_1(M)}=A^{R_2(M)}.$$
    Since $R_1(M)\simeq (M_1^*, M_1\otimes M_2,M_3)$ and $M_1\otimes M_2\simeq
		M(rs,\rho \sigma )$, the above arguments for $M$ applied to $R_1(M)$
		imply that
		$$ (3)_{-\rho (rs)\sigma	(rs)}=0,\quad \rho (rst)\sigma (rst)\tau (rs)=1
		$$
		and hence $\rho (rs)\sigma (r)=1$. Similarly,
		$R_2(M)\simeq (M_1\otimes M_2,M_2^*,M_2\otimes M_3)$. Then
		$a^{R_2(M)}_{13}=0$ implies that
	 	$$\rho\sigma (st)\sigma\tau (rs)=1,$$
		and therefore $\rho (s)\sigma (rs)=1$. Thus $M$ has a skeleton of type
		$\beta'_3$ by Lemma~\ref{lem:B':conditions}.

		Assume now that $\theta \ge 4$.
    Since $\dim M_1=1$, Lemma~\ref{lem:A2comm} implies that $\dim M_2=1$.
		Let $H$ be the subgroup generated by $\cup _{i=2}^\theta \supp M_i$.
    Then $H$ is non-abelian,
		$M'=(\Res ^G_H M_i)_{2\le i\le \theta }\in \Ggen_{\theta -1}^H$, $M'$ admits all
		reflections, $\cW (M')$ is finite, and $A^{M'}$ is of type $B_{\theta -1}$.
		Thus it suffices to lead these assumptions to a contradiction in the case
		$\theta =4$.

		Assume that $\theta =4$. By the claim for $\theta =3$ we conclude that
		there exist $r',r,s\in Z(G)$, $t,\epsilon \in G$, and characters
		$\rho ',\rho ,\sigma $ of $G$ and $\tau $ of $G^t$ such that
		$r',r,s,t,\epsilon $ generate $G$, and the relations $\epsilon ^3=1$, $t\epsilon
		=\epsilon^{-1}t$ hold in $G$. Moreover,
		$$ M_1\simeq M(r',\rho '),\quad M_2\simeq M(r,\rho ),\quad
		M_3\simeq M(s,\sigma ),\quad M_4\simeq M(t,\tau ),$$
		and the characters satisfy the relations
		\begin{align*}
			\rho '(s)\sigma (r')&=1,& \rho '(t)\tau (r')&=1,&
			\rho (rs)\sigma	(r)&=1,& \rho (t)\tau (r)&=1,\\
			\rho (s)\sigma (rs)&=1, & (3)_{-\sigma (s)}&=0,&
			\sigma (st)\tau (s)&=1,& \tau (t)&=-1.
		\end{align*}
    Since $R_1(M)\in \Ggen^G_4$ and $\dim R_1(M)_1=1$, we conclude that
    $$M'=(R_1(M)_i)_{i\in \{2,3,4\}}\in \Ggen^H_3,$$
    where $H$ is the subgroup of
    $G$ generated by $\cup _{i=2}^4\supp R_1(M)_i$.
    We record that
    $$M'_1\simeq M_1\otimes M_2,\quad M'_2\simeq M_3,\quad M'_3\simeq M_4.$$
    We now apply Theorem~\ref{thm:main} for $\theta=3$. This is possible
		since the proof does not use results on tuples in $\cF ^G_n$,
		$n\ge 4$. Since $c_{M'_3,M'_2}c_{M'_2,M'_3}\ne \id _{M'_2\ot M'_3}$,
		according to Theorem~\ref{thm:main} and the equations $\dim M'_1=\dim M'_2=1$
		we conclude that either $c_{M'_2,M'_1}c_{M'_1,M'_2}=\id_{M'_1\ot M'_2}$
	or $(M'_1,M'_2,M'_3)$ has a skeleton of type $\beta '_3$. This implies that
    $$ \rho '\rho (s)\sigma (rr')=1 \text{ or }
		\rho '\rho (r'r)\rho '\rho (s)\sigma (rr')=1.$$
    The first case is impossible since $\rho (s)\sigma (r)\ne 1$,
    $\rho '(s)\sigma (r')=1$.
    Therefore $\rho '(r'r)\rho (r')=1$.
    
    Since $a^M_{21}=-1$, we know that $\rho (r)=-1$ or $\rho (rr')\rho '(r)=1$.
    Assume first that $\rho (rr')\rho'(r)=1$. Then
    Propositions~\ref{pro:skeleton:B'} and \ref{pro:skeleton:B''} imply that
    there is a finite Cartan graph with two points corresponding to the skeleton
		of $M$ and of $R_4(M)$, respectively, such that the Cartan
    matrices of these points are
    \begin{align*}
      \begin{pmatrix}
        2 & -1 & 0 & 0 \\
        -1 & 2 & -1 & 0 \\
        0 & -1 & 2 & -1 \\
        0 & 0 & -2 & 2
      \end{pmatrix},\quad
      \begin{pmatrix}
        2 & -1 & 0 & 0 \\
        -1 & 2 & -1 & 0 \\
        0 & -2 & 2 & -1 \\
        0 & 0 & -2 & 2
      \end{pmatrix}.
    \end{align*}
	  However, by \cite[Thm.\,5.4]{MR2498801} there is no such finite Cartan
    graph, which establishes the desired contradiction.

    Assume now that $\rho '(r)\rho (r'r)\ne 1$ and $\rho (r)=-1$.
    Since $(3)_{-\rho (r)}=0$, this implies that
    $\charK =3$. Let $M''=(R_2(M)_1,R_2(M)_3,R_2(M)_4)$ and let now $H$ be the
    subgroup of $G$ generated by $\supp M''$. Since $R_2(M)\in \Ggen^G_4$ and
    $\dim R_2(M)_2=1$, we conclude that $M''\in \Ggen^H_3$. Moreover,
    $$ M''_1\simeq M_1\otimes M_2,\quad M''_2\simeq M_2\otimes M_3,\quad
       M''_3\simeq M_4.$$
    Since
    $$\rho '\rho (rs)\rho \sigma (r'r)=\rho '(r)\rho (r'r)\ne 1,$$
    the tuple $M''$ is braid-indecomposable. 
    {}From Theorem~\ref{thm:main} for $\theta=3$
    and from the facts that $\dim M''_1=\dim M''_2=1$ and $\rho \sigma (rs)=-1$
    we conclude that
    $\rho'\rho (rs)\rho\sigma (r'r)=-1$. This immediately implies that
    $\rho '(r)\rho (r')=1$, a contradiction to $a^M_{12}\ne 0$.
    Thus $\theta \ne 4$ and the proof of the theorem is completed.
\end{proof}

\begin{proof}[Proof of Theorem~\ref{thm:B2}: $\dim M_1>1$] \

	(1)$\Rightarrow $(3),(4). Since $M\in \Ggen ^G_\theta $ has a skeleton of type
	$\beta _\theta $, Proposition~\ref{pro:skeleton:B} implies
	that $M$ admits all reflections and $\cW (M)$ is standard of type
    $B_\theta $.
    Lemma~\ref{lem:B:conditions} implies that $\NA (M_i)$ is
	finite-dimensional for all $i\in \{1,\dots ,\theta \}$. More precisely,
    $$\cH_{\NA (M_i)}(t)=(2)_t^2,\quad \cH _{\NA (M_\theta )}(t)=(3)_t^2.$$
	Now Theorem~\ref{thm:HS}
	implies that
	$\NA (M)$ is finite-dimensional with the claimed Hilbert series.

	(4)$\Rightarrow $(2).
    Since $\dim \NA (M)<\infty $, the tuple $M$ admits all
	reflections by \cite[Cor. 3.18]{MR2766176} and the Weyl groupoid is finite by
	\cite[Prop. 3.23]{MR2766176}.

	(2)$\Rightarrow $(1). It is assumed that $\dim M_1>1$, $M$ admits all
    reflections, $A^M$ is of type $B_\theta $, and $\cW (M)$ is finite. Thus
		Theorem~\ref{thm:C(M)finiteCartan} tells that $\cC (M)$ is a connected
		indecomposable finite Cartan graph.

    Since $\theta \ge 3$ and $\dim M_1>1$,
		it follows from Lemma~\ref{lem:A2comm} and Lemma~\ref{lem:A2noncomm}
    that $\supp M_1$ and $\supp M_2$ do not commute and that
    $$\dim M_1=\dim M_2=|\supp M_1|=|\supp M_2|=2.$$
		Let $H$ be the subgroup generated by $\cup _{i=2}^\theta \supp M_i$. Then
    Lemma~\ref{lem:reduction} implies that
		$M'=(\Res ^G_H M_i)_{2\le i\le \theta }\in \Ggen_{\theta -1}^H$.

		Assume first that $\theta =3$.
    Let $r\in \supp M_1$, $s\in \supp M_2$, and $t\in \supp M_3$.
		Since $a^M_{13}=0$, we conclude that
		$$ s\triangleright t=r\triangleright (s\triangleright t)
		   =(r\triangleright s)\triangleright (r\triangleright t)
			 =(r\triangleright s)\triangleright t,$$
			 where $\triangleright $ means conjugation: $s\triangleright
			 t=sts^{-1}$.
		Since $s\ne r\triangleright s$, this means that
		both elements of $\supp M_2$ act in the same way on $\supp M_3$.
		Then \cite[Thm.\,2.1]{rank2} implies that $\charK=3$, $\dim M_3=|\supp
		M_3|=2$, and that the conditions in Lemma~\ref{lem:B:conditions}(2) hold.
		Then Lemma~\ref{lem:B:conditions} implies (1).

		Assume now that $\theta >3$.
		Since $\dim M_2>1$, the claim for $\theta -1$ implies that
		$\charK=3$ and $M'$ has a skeleton of type $\beta _{\theta -1}$.
		In particular, by Lemma~\ref{lem:B:conditions} there exist $s_2,\dots
		,s_\theta ,\epsilon \in G$ such that
		$\epsilon s_i=s_i\epsilon $ and $s_i^H=\{s_i,\epsilon s_i\}$ for $2\le
		i\le \theta $, where $H\subseteq G$ is the subgroup generated by $s_2,\dots
		,s_\theta ,\epsilon $. Let $s_1\in \supp M_1$. Since $s_1s_2\ne s_2s_1$
		and $\epsilon	^2=1$, we conclude from Lemma~\ref{lem:rs} that
		$\supp M_1=\{s_1,\epsilon s_1\}$ and $s_1\epsilon =\epsilon s_1$.
		Since $G$ is generated by $s_1,\dots ,s_\theta ,\epsilon $, we conclude that
		$\epsilon \in Z(G)$. In order to prove that $M$ has a skeleton of type $\beta _\theta $,
		one has	to check conditions \eqref{eq:B:1}--\eqref{eq:B:5} in
		Lemma~\ref{lem:B:conditions} for $i=1$.
		These follow from Lemmas~\ref{lem:HS:X1}, \ref{lem:HS:X2}, and \ref{lem:a=0for22}.

		(3)$\Rightarrow $(2) is clear.
\end{proof}

\section{Proof of Theorem \ref{thm:F4}: The case $F_4$}
\label{section:F4}

In this section we require that all the assumptions of Theorem~\ref{thm:F4}
hold. Thus let $G$ be a non-abelian group
and let $M=(M_1,M_2,M_3,M_4)\in\Ggen_4^G$.
Assume that $A^M$ is a Cartan matrix of type $F_4$. More precisely,
$$a^M_{12}=a^M_{21}=a^M_{23}=a^M_{34}=a^M_{43}=-1, \quad a^M_{32}=-2,$$
and $a^M_{ij}=0$ otherwise if $i\ne j$.

\begin{lem}
	\label{lem:F4:step1}
	Let $H=\langle\cup_{i=2}^4\supp M_i\rangle$ and $M'=(\Res_H^G M_i)_{2\leq
	i\leq 4}$. Then $H$ is non-abelian, $M'\in\Ggen_3^H$ and $A^{M'}$ is of type
	$C_3$. Moreover, dim $M_1=1$.
\end{lem}

\begin{proof}
	Lemma~\ref{lem:reduction} implies that $M'\in \Ggen_3^H$ and that $H$ is
	non-abelian. Since $A^M$ is of
	type $F_4$, we conclude that $A^{M'}$ is of type $C_3$.

  Since $H$ is non-abelian, Lemma~\ref{lem:Ctheta:noncomm} for $M'$
	implies that $\dim M_2=1$. Therefore $\supp M_2$ commutes with $\supp M_1$ and
	hence $\dim M_1=1$ by Lemma~\ref{lem:A2comm}.
\end{proof}

The skeleton of type $\varphi_4$ is described in the following
lemma.

\begin{lem}
	\label{lem:F4:conditions}
	Assume that $\charK\ne2$. Let $N\in \cF^G_4$. The
	following are equivalent:
	\begin{enumerate}
		\item $N$ has a skeleton of type $\varphi_4$. 
		\item 
			There exists $\epsilon\in Z(G)$ with $\epsilon^2=1$ and for all
			$i\in\{1,\dots,4\}$ and all $s_i\in\supp M_i$ there exists a unique
			character $\sigma_i$ of $G^{s_i}$ such that $\supp M_i=\{s_i\}$ for $i\in
			\{1,2\}$, $\supp M_i=\{s_i,\epsilon s_i\}$ for $i\in \{3,4\}$,
			$M_i\simeq M(s_i,\sigma_i)$ for all
			$i\in\{1,\dots,4\}$, and  the following conditions hold:
			\begin{align}
				\label{eq:F4:1}&\sigma_1(s_1)=\sigma_2(s_2)=\sigma_3(s_3)=\sigma_4(s_4)=-1,\\
				\label{eq:F4:2}&\sigma_4(\epsilon s_3^2)\sigma_3(\epsilon s_4^2)=1,\\
				\label{eq:F4:3}&\sigma_4(s_1)\sigma_1(s_4)=\sigma_3(s_1)\sigma_1(s_3)
			  =\sigma_4(s_2)\sigma_2(s_4)=1,\\
				\label{eq:F4:4}&\sigma_3(s_2)\sigma_2(s_3)=-1,\\
				\label{eq:F4:5}&\sigma_1(s_2)\sigma_2(s_1)=-1,\\
				\label{eq:F4:6}&s_3s_4=\epsilon s_4s_3.
			\end{align}
	\end{enumerate}
\end{lem}

\begin{proof}
	Suppose that $N$ has a skeleton of type $\varphi_4$. Then $A^N$ is of type
	$F_4$. Lemma~\ref{lem:A3:222}(1) implies now the existence of $\epsilon $ such that
	\eqref{eq:F4:6} holds and the supports of $M_3,M_4$ are of the given form. 
	Since $A^{(N_3,N_4)}$ is of type $A_2$,
	Corollary~\ref{co:2-2:A2B2CartanMatrix} implies \eqref{eq:F4:2} and that
	$\sigma_4(s_4)=\sigma_3(s_3)=-1$. The remaining conditions in \eqref{eq:F4:1}
	and \eqref{eq:F4:4}, \eqref{eq:F4:5}
	hold by definition of the skeleton. Now \eqref{eq:F4:3} follows from
	Lemma~\ref{lem:genRosso} since $a_{14}^M=a_{13}^M=a_{24}^M=0$.

	The converse
	follows immediately from the definition of a skeleton of type $\varphi_4$
	using Lemmas~\ref{lem:genRosso}, \ref{lem:X2}, \ref{lem:X3},
	and Corollary~\ref{co:2-2:A2B2CartanMatrix}.
\end{proof}

Reflections of the skeleton of type $\varphi_4$ are considered in the following lemma.

\begin{pro}
	\label{pro:skeleton:F4}
	Let $M\in\cF_4^G$.  Assume that $M$ has a skeleton $\cS $ of type $\varphi_4$. Then
	$\cS $ is a skeleton of $R_k(M)$ for all $k\in \{1,2,3,4\}$.
\end{pro}

\begin{proof}
	According to Remark~\ref{rem:connectedtriples} it suffices to determine the
	skeletons of $R_k(M_{i_1},M_{i_2},M_{i_3})$, where $i_1,i_2,i_3$ correspond to
	three vertices of a connected subgraph of $\cS $ and $k\in \{1,2,3\}$.
	There are only two such subgraphs and hence the proposition follows from
	Lemmas~\ref{lem:skeleton:112} and \ref{lem:skeleton:C3}.
\end{proof}

We are now ready to prove Theorem~\ref{thm:F4}. 

\begin{proof}[Proof of Theorem~\ref{thm:F4}]
	We prove the implications
	\eqref{it:C:skeleton}$\Rightarrow $\eqref{it:C:NA}$\Rightarrow
	$\eqref{it:C:fgroupoid}$\Rightarrow$\eqref{it:C:skeleton}
	and \eqref{it:C:skeleton}$\Rightarrow
	$\eqref{it:C:sgroupoid}$\Rightarrow$\eqref{it:C:fgroupoid}.

  \eqref{it:F4:sgroupoid}$\Rightarrow $\eqref{it:F4:fgroupoid}.
  This is clear, see e.\,g.~\cite[Thm.\,3.3]{MR2498801}.

	\eqref{it:F4:skeleton}$\Rightarrow $\eqref{it:F4:sgroupoid},\eqref{it:F4:Nichols}.
	Since $M\in \Ggen ^G_4$ has a skeleton of type
	$\varphi _4$, Proposition~\ref{pro:skeleton:F4} implies
	that $M$ admits all reflections and $\cW (M)$ is standard of type
  $F_4$.
	The longest element of the Weyl
	group of type $F_4$ is
	\[
	s_1s_2s_1s_3s_2s_1s_3s_2s_3s_4s_3s_2s_1s_3s_2s_3s_4s_3s_2s_1s_3s_2s_3s_4. 
	\]
	The Nichols algebras $\NA(M_i)$ are finite-dimensional for
	$i\in\{1,\dots,4\}$ and
	\[
    \cH_{\NA(M_i)}(t)=\begin{cases}
		(2)_t & \text{if $i\in\{1,2\}$},\\
		(2)_t^2 & \text{if $i\in\{3,4\}$}.
	\end{cases}
	\]
	With respect to the Cartan matrix of type $F_4$ 
	one computes 
	\begin{align*}
		&\beta_{1}=\alpha_1,
		&&\beta_{2}=\alpha_1+\alpha_2,\\
		&\beta_{3}=\alpha_2,
		&&\beta_{4}=\alpha_1+\alpha_2+\alpha_3,\\
		&\beta_{5}=\alpha_1+2\alpha_2+2\alpha_3,
		&&\beta_{6}=\alpha_1+\alpha_2+2\alpha_3,\\
		&\beta_{7}=\alpha_2+\alpha_3,
		&&\beta_{8}=\alpha_2+2\alpha_3,\\
		&\beta_{9}=\alpha_3,
		&&\beta_{10}=\alpha_1+2\alpha_2+3\alpha_3+\alpha_4,\\
		&\beta_{11}=\alpha_1+2\alpha_2+2\alpha_3+\alpha_4,
		&&\beta_{12}=2\alpha_1+3\alpha_2+4\alpha_3+2\alpha_4,\\
		&\beta_{13}=\alpha_1+3\alpha_2+4\alpha_3+2\alpha_4,
		&&\beta_{14}=\alpha_1+\alpha_2+2\alpha_3+\alpha_4,\\
		&\beta_{15}=\alpha_1+2\alpha_2+4\alpha_3+2\alpha_4,
		&&\beta_{16}=\alpha_2+2\alpha_3+\alpha_4,\\
		&\beta_{17}=\alpha_1+2\alpha_2+3\alpha_3+2\alpha_4,
		&&\beta_{18}=\alpha_1+\alpha_2+\alpha_3+\alpha_4,\\
		&\beta_{19}=\alpha_1+2\alpha_2+2\alpha_3+2\alpha_4,
		&&\beta_{20}=\alpha_1+\alpha_2+2\alpha_3+2\alpha_4,\\
		&\beta_{21}=\alpha_2+\alpha_3+\alpha_4,
		&&\beta_{22}=\alpha_2+2\alpha_3+2\alpha_4,\\
		&\beta_{23}=\alpha_3+\alpha_4,
		&&\beta_{24}=\alpha_4.
	\end{align*}
	The long and short roots are $\beta_j$ with
	$$j\in \{1,2,3,5,6,8,12,13,15,19,20,22\}$$
	and
	$$j\in \{4,7,9,10,11,14,16,17,18,21,23,24\},$$
	respectively.
	By Theorem~\ref{thm:HS}, 
	\[
	\NA(M)\simeq\NA(M_{\beta_{24}})\otimes\cdots\otimes\NA(M_{\beta_1})
	\]
	as $\N_0^4$-graded objects in $\ydG$.  Thus a direct calculation shows that
	$\NA(M)$ is finite-dimensional with the claimed Hilbert series.
	
	\eqref{it:F4:Nichols}$\Rightarrow $\eqref{it:F4:fgroupoid}.
  Since $\dim \NA (M)<\infty $, the tuple $M$ admits all
	reflections by \cite[Cor. 3.18]{MR2766176} and the Weyl groupoid is finite by
	\cite[Prop. 3.23]{MR2766176}.

	\eqref{it:F4:fgroupoid}$\Rightarrow$\eqref{it:F4:skeleton}.
	Let $H$ be the subgroup of $G$ generated by $\cup _{i=2}^4\supp M_i$.
	Let $N=(\Res ^G_H M_i)_{i\in \{2,3,4\}}$.
	Lemma~\ref{lem:F4:step1} implies that $N\in	\cE^H_3$ and
	$A^N$ is of type $C_3$.
	Therefore, by
	Theorem~\ref{thm:C}\eqref{it:C:fgroupoid}$\Rightarrow$\eqref{it:C:skeleton},
	$N$ has a skeleton of type $\gamma _3$ and $\charK\ne 2$.
	Moreover, $\dim M_1=1$ by Lemma~\ref{lem:F4:step1}. For all $i\in \{1,2,3,4\}$
	let $s_i\in G$ and $\sigma_i\in \chg{G^{s_i}}$ such that $M_i\simeq
  M(s_i,\sigma_i)$. Then $\sigma_2(s_2)=-1$, $\sigma_2(s_3)\sigma_3(s_2)=-1$, and
	\begin{align} \label{eq:MtypeF4}
		(\sigma_1(s_1)+1)(\sigma_1(s_1s_2)\sigma_2(s_1)-1)=0,
  	\quad \sigma_1(s_3)\sigma_3(s_1)=1
  \end{align}
	by Lemma~\ref{lem:Rosso}, since $a^M_{12}=-1$ and $a^M_{13}=0$.
	We are left to show that $\sigma_1(s_1)=\sigma_1(s_2)\sigma_2(s_1)=-1$.

	Let $M'=R_2(M)$. Then $M'_1\simeq M(s_2s_1,\sigma_2\sigma_1)$, $M'_3\simeq
	M(s_2s_3,\sigma_2\sigma_3)$, and $M'_4=M_4$ by Lemma~\ref{lem:genRosso}.
	In particular, $\dim M'_1=1$, $\dim M'_3=\dim M'_4=2$, and $\supp M'_3$ and $\supp
	M'_4$ do not commute. Moreover, $a^{M'}_{14}=0$
	by Lemma~\ref{lem:disconnected}. Then $(M'_1,M'_3,M'_4)\in \cF^{H'}_3$, where
	$H'\subseteq G$ is the subgroup generated by $\supp M'_1\cup \supp M'_3\cup \supp
	M'_4$. The Weyl groupoid of $(M'_1,M'_3,M'_4)$ is
	finite by assumption.
	We apply Theorem~\ref{thm:main} for $\theta=3$, which is possible, since
	its proof for $\theta =3$ does not use anything about $\theta $-tuples
	with $\theta\geq4$.
We obtain that either $a^{M'}_{13}=0$ or the triple $(M'_1,M'_3,M'_4)$
	has a skeleton of type
	$\gamma_3$. In the second case, necessarily
	$\sigma_2\sigma_1(s_2s_3)\sigma_2\sigma_3(s_2s_1)=-1$ holds.
	Equations $\sigma_2(s_2)=\sigma_2(s_3)\sigma_3(s_2)=-1$ and
	\eqref{eq:MtypeF4} imply that
	$$ \sigma_2\sigma_1(s_2s_3)\sigma_2\sigma_3(s_2s_1)=
	 -\sigma_1(s_2)\sigma_2(s_1), $$
	 and hence in the second of the above two cases necessarily
	 $\sigma_1(s_2)\sigma_2(s_1)=1$ holds.
  Since $\sigma_1(s_2)\sigma_2(s_1)\ne 1$ because of $a^M_{12}\ne 0$ and
	Lemma~\ref{lem:genRosso},
	we conclude that the second case is impossible and hence $a^{M'}_{13}=0$. Then
	$\sigma_1(s_2)\sigma_2(s_1)=-1$, in which case $\sigma_1(s_1)=-1$ by
	\eqref{eq:MtypeF4}. Thus we are done, as said at the end of the previous
	paragraph.
\end{proof}

\section{Proof of Theorem \ref{thm:main}: The classification}
\label{section:main}

Recall that $\theta\in\N_{\geq3}$, $G$ is a non-abelian group and
$M\in\Ggen_\theta^G$ is a braid-indecomposable tuple.

\begin{proof}[Proof of Theorem \ref{thm:main}]
	(1)$\Rightarrow $(2) Assume that $M$ has a skeleton $\cS$ of finite type.
	If $M$ has a skeleton of type $\alpha _\theta $ or $\delta _\theta $ or
	$\varepsilon _\theta $, then $\dim \NA (M)<\infty $ by Theorem~\ref{thm:ADE}.
	If $M$ has a skeleton of type $\gamma _\theta $ or $\varphi _4$,
	then $\dim \NA (M)<\infty $ by Theorem~\ref{thm:C} and \ref{thm:F4},
	respectively.
	If $M$ has a skeleton of type $\beta _\theta $, then $\dim M_1>1$
	by Lemma~\ref{lem:B:conditions} and hence
	$\dim \NA (M)<\infty $ by Theorem~\ref{thm:B2}.
	If $M$ has a skeleton of type $\beta '_3$, then $\dim M_1=1$
	by Lemma~\ref{lem:B':conditions} and hence
	$\dim \NA (M)<\infty $ by Theorem~\ref{thm:B1}.
	Finally, if $M$ has a skeleton of type $\beta ''_3$, then $R_3(M)$ has a
	skeleton of type $\beta '_3$ by Proposition~\ref{pro:skeleton:B''}.
	Hence $\dim \NA (R_3(M))<\infty $. Since $R_3(R_3(M))\simeq M$ by
	\cite[Thm.\,3.12]{MR2766176}, we conclude from \cite[Thm.\,1]{MR2766176}
	that $\dim \NA (M)=\dim \NA (R_3(M))<\infty $.
  
	(2)$\Rightarrow $(3)
	Since $\dim \NA (M)<\infty $, the tuple $M$ admits all
	reflections by \cite[Cor. 3.18]{MR2766176} and the Weyl groupoid is finite by
	\cite[Prop. 3.23]{MR2766176}.

	(3)$\Rightarrow $(1)
	Recall that $M$ is braid-indecomposable.
	Suppose that $M$ admits all reflections and $\cW(M)$ is finite. Then
	$\cC (M)$ is a connected indecomposable finite Cartan graph by 
	Theorem~\ref{thm:C(M)finiteCartan}.
	Therefore by Theorem~\ref{thm:finitetype}
	there exist $k\in \N _0$ and $i_1,\dots ,i_k\in \{1,\dots ,\theta \}$
	such that $A^N$ is an indecomposable Cartan matrix of finite type
	for $N=R_{i_1}\cdots R_{i_k}(M)$. The set of all indecomposable Cartan
	matrices of finite type is well-known: They are of $ADE$ types or of type
	$B_\theta $, $C_\theta $, or $F_4$.
	By Theorems~\ref{thm:ADE}, \ref{thm:B1}, \ref{thm:B2}, \ref{thm:C}, and
	\ref{thm:F4} the tuple $N$ has a skeleton of finite type. Since
	$M\simeq R_{i_k}\cdots R_{i_1}(N)$, from
	Propositions~\ref{pro:skeleton:ADE}, \ref{pro:skeleton:C},
	\ref{pro:skeleton:F4},
	\ref{pro:skeleton:B},
	\ref{pro:skeleton:B'}, and \ref{pro:skeleton:B''} we conclude that
	$M$ has a skeleton of finite type.
\end{proof}

\appendix
\section{Reflections of a pair} 
\label{appendix:reflections}

\subsection{}

For one-dimensional Yetter-Drinfeld modules $U,V$ over a group
$H$, the
Yetter-Drinfeld modules $(\ad U)^m(V)$ and $(\ad V)^m(U)$ for $m\ge 1$ are
well-known by the theory of Nichols algebras of diagonal type. The following
lemma goes back to Rosso, see \cite[Lemma 14]{MR1632802}.

\begin{lem}[Rosso]
	\label{lem:Rosso}
	Let $H$ be a group and let $U,V\in \ydH $.
	Assume that $U\simeq M(r,\rho )$ and $V\simeq M(s,\sigma )$, where
	$r,s\in Z(H)$ and $\rho ,\sigma $ are characters of $H$. Then
	$(\ad U)^m(V)\ne 0$ for a given $m\in \N $ if and only if
	\[
		(m)^!_{\rho(r)}\prod _{i=0}^{m-1}(\rho(r^i s)\sigma(r)-1)\ne 0.
	\]
	In this case, $(\ad U)^m(V)\simeq M(r^ms,\sigma_m)$, where
	$\sigma_m$ is the character of $H$ given by
	$\sigma_m(h)=\rho(h)^m\sigma(h)$ for all $h\in H$.
\end{lem}

Rosso's lemma is a special case of a more general statement which we prove here.

\begin{lem}
	\label{lem:genRosso}
	Let $H$ be a group and let $U,V\in \ydH $.
	Assume that $U\simeq M(r,\rho )$ and $V\simeq M(s,\sigma )$, where
	$r\in Z(H)$, $s\in H$, $\rho \in \chg H$, and $\sigma $ is a
	representation of $H^s$. Assume also that $\sigma (r)$ is a constant
	automorphism of $V$. Then
	$(\ad U)^m(V)\ne 0$ for a given $m\in \N $ if and only if
	\[
		(m)^!_{\rho(r)}\prod _{i=0}^{m-1}(1-\rho(r^i s)\sigma(r))\ne 0.
	\]
	In this case, $(\ad U)^m(V)\simeq M(r^ms,\sigma_m)$, where
	$\sigma_m$ is the representation of $H^s$ given by
	$\sigma_m(h)=\rho(h)^m\sigma(h)$ for all $h\in H^s$.
\end{lem}

\begin{proof}
	By Lemma~\ref{lem:X_n},
	it suffices to prove the claim for $X_m^{U,V}$ instead of $(\ad U)^m(V)$.

	Let $u\in U\setminus \{0\}$ and $v\in V_s\setminus \{0\}$.
  For all $m\ge 1$ let
	$$\gamma _m=(m)_{\rho (r)}(1-\rho (r^{m-1}s)\sigma (r)).$$
  We prove that
	\begin{align} \label{eq:Xm}
	  X_m^{U,V}=\gamma _1\cdots \gamma _mU^{\otimes m}\otimes V
	\end{align}
	for all $m\ge 1$. Then $X_m^{U,V}=0$ if $\gamma_i=0$ for some $i\in \{1,\dots
	,m\}$, and otherwise $X_m^{U,V}\simeq M(r^ms,\sigma _m)$. Indeed,
	$$ X_m^{U,V}=\oplus _{t\in \supp V}(U^{\otimes m}\otimes V_t)$$
	in the latter case and
	$$ h(u^{\otimes m}\otimes w)=\rho (h)^m u^{\otimes m}\otimes hw $$
	for all $w\in V_s$.

	We prove by induction on $m$ that
	\begin{align}\label{eq:varphim(umv)}
		\varphi _m(u^{\otimes m}\otimes v)=\gamma _mu^{\otimes m}\otimes v
  \end{align}
  for all $m\ge 1$ and all $v\in V_s$. This clearly implies \eqref{eq:Xm}.

	Let $v\in V_s$.
	For $m=1$ we have $\varphi _1(u\otimes v)=(\id -c^2)(u\otimes v)$ and
	\begin{align*}
		c^2(u\otimes v)=c(rv\otimes u)=\sigma (r)su\otimes v=\rho (s)\sigma
		(r)u\otimes v.
	\end{align*}
	Therefore $\varphi _1(u\otimes v)=\gamma_1u\otimes v$. Assume now that
	\eqref{eq:varphim(umv)} holds for some $m\ge 1$. Then
  \begin{align*}
		\varphi _{m+1}(u^{\otimes m+1}\otimes v)&=
		u^{\otimes m+1}\otimes v-c^2(u\otimes (u^{\otimes m}\otimes v))\\
		&\qquad
		+(\id \ot \varphi _m)c_{12}(u\ot u\ot (u^{\otimes m-1}\ot v))\\
    &=\big((1-\rho (r)^m\sigma (r)\rho (r^ms))+\rho (r)\gamma _m\big)
		u^{\ot m+1}\ot v\\
		&=\gamma _{m+1}u^{\ot m+1}\ot v.
  \end{align*}
	This proves the lemma.
\end{proof}

\subsection{}

In this section we collect some auxiliary results regarding reflections of
\cite[\S4]{MR2732989}.
Let $G$ be a non-abelian group.

Let $g,h,\epsilon\in G$. Assume that $|g^G|=|h^G|=2$, $gh\ne hg$, and
	$gh=\epsilon hg$.
	By Lemma~\ref{lem:A3:222}
	the subgroup $\langle g,h,\epsilon \rangle $ of $G$ is an epimorphic
image of $\Gamma _2$. Let
$V,W\in\ydG$ with 
$V\simeq M(g,\rho)$ and
$W\simeq M(h,\sigma)$, where $\rho\in \chg{G^g}$ and $\sigma\in\chg{G^h}$. 
Let $v\in V_g\setminus\{0\}$. Then $\{v,hv\}$ is a basis of $V$. The degrees of
these basis vectors are $g$ and $\epsilon g$, respectively. Similarly let $w\in
W_h\setminus\{0\}$. Then $\{w,gw\}$ is a basis of $W$ and the degrees of these
basis vectors are $h$ and $\epsilon h$, respectively. In particular,
$\Res ^G_{\langle g,h,\epsilon \rangle} V$ and
$\Res ^G_{\langle g,h,\epsilon \rangle} W$
are absolutely simple Yetter-Drinfeld modules over $\langle g,h,\epsilon \rangle
$. Since $z$ acts on $V^{\otimes m}\otimes W^{\otimes n}$ for $z\in G^g\cap G^h$
and $m,n\in \N _0$ by $\rho (z)^m\sigma(z)^n\id $, the following claims follow
directly from the corresponding results in \cite{MR2732989}.

\begin{lem}{\cite[Lemma 4.1]{MR2732989}}\
	\label{lem:HS:X1}
	\begin{enumerate}
		\item $X_1^{V,W}\ne0$. Moreover, 
			$X_1^{V,W}$ is absolutely simple if and only if $\rho(\epsilon
			h^2)\sigma(\epsilon g^2)=1$. 
		\item Assume that $\rho(\epsilon
			h^2)\sigma(\epsilon g^2)=1$. Then $X_1^{V,W}\simeq M(gh,\widetilde{\sigma})$,
			where $\widetilde{\sigma}\in \chg{G^{gh}}=\langle \{gh\}\cup (G^g\cap
			G^h)\rangle$ with
			$\widetilde{\sigma}(gh)=-\rho (g)\sigma(h)$, and 
			$\widetilde{\sigma}(z)=\rho(z)\sigma(z)$ for all $z\in G^g\cap G^h$. 
	\end{enumerate}
\end{lem}

\begin{lem}{\cite[Lemma 4.2]{MR2732989}}
	\label{lem:HS:X2}
	Assume that $\rho(\epsilon h^2)\sigma(\epsilon g^2)=1$. 
	\begin{enumerate}
		\item $X_2^{V,W}=0$ if and only if $\rho(g)=-1$. 
		\item $X_2^{V,W}$ is absolutely simple if and only if $\rho(g)=1$ and
			$\charK\ne2$. 
	\end{enumerate}
\end{lem}

\begin{lem}{\cite[Lemma 4.3]{MR2732989}}
	\label{lem:HS:X3}
	Assume that $\rho(\epsilon h^2)\sigma(\epsilon g^2)=1$, $\rho(g)=1$ and
	$\charK\ne2$. Let $n\in \N $.
	\begin{enumerate}
		\item If $n\ge 3$ then $X_n^{V,W}=0$ if and only if $0<\charK\leq n$.  
		\item If $n\ge 1$ and $X_n^{V,W}\ne0$ then $X_n^{V,W}\simeq M(g^nh,\widetilde{\sigma})$,
			where $\widetilde{\sigma}$ is a character of 
			$G^{g^nh}=\langle \{g^nh\}\cup (G^g\cap G^h)\rangle$ with
			$\widetilde{\sigma}(g^nh)=(-1)^n\sigma(h)$ and
			$\widetilde{\sigma}(z)=\rho (z)^n\sigma(z)$ for all $z\in G^g\cap G^h$.
	\end{enumerate}
\end{lem}

With the previous calculations and exchanging $V$ and $W$ one immediately
obtains the following lemma, see \cite[Prop. 4.4]{MR2732989}. 

\begin{lem}
	\label{lem:HS:CartanMatrix}
	The Yetter-Drinfeld modules 
	$(\ad V)^m(W)$ and $(\ad W)^m(V)$ are absolutely simple or zero for
	all $m\geq0$ if and only if $\rho(\epsilon h^2)\sigma(\epsilon g^2)=1$
	and $\rho(g)^2=\sigma(h)^2=1$. In this case, the non-diagonal
	entries of the Cartan matrix $A^{(V,W)}$ are 
	\[
		a_{12}^{(V,W)}=\begin{cases}
			-1 & \text{if $\rho(g)=-1$},\\
			1-p & \text{if $\rho(g)=1$ and $\charK=p>2$},
		\end{cases}
	\]
	and otherwise $(\ad V)^m(W)\ne0$ for all $m\geq0$, and similarly
	\[
		a_{21}^{(V,W)}=\begin{cases}
			-1 & \text{if $\sigma(h)=-1$},\\
			1-p & \text{if $\sigma(h)=1$ and $\charK=p>2$},
		\end{cases}
	\]
	and otherwise $(\ad W)^m(V)\ne0$ for all $m\geq0$.
\end{lem}

\begin{cor}
	\label{co:2-2:A2B2CartanMatrix}
	Let $V,W$ be as above.
	\begin{enumerate}
		\item We have $a^{(V,W)}_{12}=a^{(V,W)}_{21}=-1$ if and only if
			$\rho (\epsilon h^2)\sigma (\epsilon g^2)=1$ and
			$\rho (g)=\sigma (h)=-1$.
		\item We have $a^{(V,W)}_{12}=-1$, $a^{(V,W)}_{21}=-2$ if and only if
			$\rho (\epsilon h^2)\sigma (\epsilon g^2)=1$, $\rho (g)=-1$, $\sigma
			(h)=1$, and $\charK =3$.
	\end{enumerate}
\end{cor}

\begin{proof}
	The if part of the claim follows directly from
	Lemma~\ref{lem:HS:CartanMatrix}.

	For the only if part observe first that $a^{(V,W)}_{12}=-1$,
	$a^{(V,W)}_{21}\ge -2$
	imply that $(\ad V)(W)$ and $(\ad W)^m(V)$ with $0\le m\le -a^{(V,W)}_{21}$
	are absolutely simple by Proposition~\ref{pro:absimple}.
	Then $\rho (\epsilon h^2)\sigma (\epsilon g^2)=1$ by Lemma~\ref{lem:HS:X1},
	and the only if parts of (1) and (2) follow from Lemmas~\ref{lem:HS:X2} and
	\ref{lem:HS:X3}.
\end{proof}

Finally to compute the reflections of the pair $(V,W)$ one has the following
lemma.

\begin{lem}{\cite[Lemma 4.5]{MR2732989}}
	\label{lem:HS:reflections}
	Assume that 
	\[
		\rho(\epsilon h^2)\sigma(\epsilon g^2)=1,\quad
		\rho(g)^2=\sigma(h)^2=1,
	\]
	and that $\rho(g)=-1$ if $\charK=0$. Let $m=1$ if $\rho (g)=-1$ and let
	$m=p-1$ if $\rho (g)=1$ and $\charK =p>0$. Let $g'=g^{-1}$ and $h'=g^mh$. Then 
	\begin{align*}
		|g'^G|=|h'^G|=2,\quad g'h'\ne h'g',\quad g'h'=\epsilon h'g',\quad
		G^g\cap G^h=G^{g'}\cap G^{h'}.
	\end{align*}
	Moreover, $R_1(V,W)=(V',W')$ with $V'\simeq M(g',\rho')$ and $W'\simeq
	M(h',\sigma')$, where $\rho'\in \chg{G^{g'}}$ and $\sigma'\in \chg{G^{h'}}$ with 
	\begin{align*}
		\rho'(\epsilon h'^2)\sigma'(\epsilon g'^2)=1,
		&&
		\rho'(g')=\rho(g),
		&&
		\sigma'(h')=\sigma(h),
	\end{align*}
	and $\rho '(z)=\rho (z)^{-1}$, $\sigma '(z)=\rho (z)^m\sigma (z)$
	for all $z\in G^g\cap G^h$.
\end{lem}

\subsection{}

Here we recall results on particular pairs of Yetter-Drinfeld modules which play
an important role in the study of skeletons of type $\beta '_\theta$ and
$\beta ''_\theta $.

By Proposition~\ref{pro:absimple}, for any pair $(U,V)\in \cF^G_2$ the
Yetter-Drinfeld modules $(\ad U)^m(V)$ and $(\ad V)^m(U)$ are absolutely simple
or zero if $a^{(U,V)}_{12}a^{(U,V)}_{21}$ is one of $0$, $1$, $2$.
Therefore Lemmas ~\ref{lem:1+3=typeB2} and \ref{lem:2+3=typeB2} below
are special cases of \cite[Prop.\,6.6]{rank2}
and \cite[Prop.\,4.12]{rank2}, respectively.

\begin{lem} \label{lem:3}
  Let $t,t'\in G$. Assume that $tt'\ne t't$, $|t^G|=3$, and
  $t^G=t'^G$. Let $\epsilon \in G$ be such that $t'=\epsilon t$.
Then $\epsilon ^3=1$, $t\epsilon =\epsilon ^{-1}t$, and
$t^G=\{t,\epsilon t,\epsilon ^2t\}$.
\end{lem}

\begin{proof}
  Since $tt'\ne t't$, we conclude that $t\epsilon \ne \epsilon t$. Therefore
	$\epsilon $ commutes neither with $t$ nor with $\epsilon t$. Let $t''\in
	t^G$ be such that $t^G=\{t,\epsilon t,t''\}$. Then
$\epsilon t\epsilon ^{-1}\notin \{t,\epsilon t\}$, and hence $\epsilon
t\epsilon^{-1}=t''$. Thus conjugation by $\epsilon $ permutes $t^G$ via
$t\mapsto t''$, $t''\mapsto t'$, $t'\mapsto t$.
Hence $\epsilon ^2t\epsilon^{-1}=\epsilon t'\epsilon^{-1}=t$.
Then $t''=\epsilon t\epsilon^{-1}=\epsilon^{-1}t$ and $\epsilon t=\epsilon
t''\epsilon^{-1}=t\epsilon^{-1}=\epsilon^{-2}t$. Thus $\epsilon^3=1$ which
implies the rest.
\end{proof}

\begin{lem} \label{lem:1+3=typeB2}
	Let $s\in Z(G)$ and $t,\epsilon \in G$ be such that $\epsilon ^3=1$, $\epsilon
\ne 1$, $t\epsilon =\epsilon ^{-1}t$, and $|t^G|=3$. Let $\sigma \in \chg G$ and $\tau \in
\chg{G^t}$ and let $U,V\in \ydG $ be such that $U\simeq M(s,\sigma )$ and $V\simeq
M(t,\tau )$. Then $a^{(U,V)}_{12}=-1$ and $a^{(U,V)}_{21}=-2$ if and only if
$$ \tau (t)=-1, \quad (3)_{-\sigma (t)\tau (s)}=0, \quad
(1+\sigma (s))(1-\sigma (st)\tau (s))=0.
$$
\end{lem}

\begin{proof}
	The assumptions imply that $\langle t,\epsilon ,s\rangle $ is a non-abelian
	epimorphic image of $\Gamma _3$.
	By Lemma~\ref{lem:genRosso}, $a^{(U,V)}_{12}=-1$ if and only if
	$\sigma(s)\tau(t)\ne 1$ and $(1+\sigma(s))(1-\sigma(st)\tau(s))=0$.
	The rest follows from
	\cite[Lemmas\,6.2,6.3]{rank2} since $a^{(U,V)}_{21}=-2$ implies that
	$(\ad V)^2(U)=R_2(U,V)_1$ is absolutely simple.
\end{proof}

\begin{pro} 
	\label{pro:R2:1+3}
	Let $s\in Z(G)$ and $t,\epsilon \in G$ be such that $\epsilon ^3=1$, $\epsilon
\ne 1$, $t\epsilon =\epsilon ^{-1}t$, and $|t^G|=3$.
Let $\sigma \in \chg G$ and $\tau \in
\chg{G^t}$ be such that
$$ \tau (t)=-1, \quad (3)_{-\sigma (t)\tau (s)}=0, \quad
\sigma (st)\tau (s)=1 $$
and let $U,V,U',V'\in \ydG $ such that $U\simeq M(s,\sigma )$,
$V\simeq M(t,\tau )$, and $(U',V')=R_2(U,V)$.
Then $U'\simeq M(s',\sigma ')$ and $V'\simeq M(t^{-1},\tau ^*)$, where
$s'=\epsilon st^2$ and
$\sigma '\in \chg{G^\epsilon }$ such that $\sigma '(\epsilon )=(\sigma
(t)\tau(s))^2$, $\sigma'(h)=\tau(h)^2\sigma(h)$ for all $h\in G^t\cap
G^\epsilon $. Moreover, $\epsilon ^G=\{\epsilon,\epsilon^{-1}\}$, $t^2\in
Z(G)$, and
$$
 \sigma'(\epsilon t^{-2})\tau^*(\epsilon s'^2)=1, \quad
 \sigma '(s')=-1,\quad
 \tau^*(t^{-1})=-1.
$$
\end{pro}

\begin{proof}
	First we prove that $\epsilon ^G=\{\epsilon ,\epsilon^{-1}\}$ and that
	$t^2\in Z(G)$. Indeed, the assumptions imply that
	$t^G=\{t,\epsilon t,\epsilon ^2t\}$ and hence $G=\langle t,\epsilon ,G^t\cap
  G^\epsilon \rangle $. Let $H$ be the subgroup of $G$ generated by
	$s,t$, and $\epsilon $. Then $\Res^G_H V$, $\Res ^G_H U \in \ydH $
	are absolutely simple. The calculation of $V^*=R_2(U,V)_2$ is standard.
	We conclude from \cite[Lemmas~6.2 and 6.3]{rank2} that
	$a^{(U,V)}_{21}=-2$. From \cite[Lemma\,6.2]{rank2} we obtain that
    $R_2(U,V)_1\simeq M(s',\sigma')$ and that the remaining claims hold.
\end{proof}

\begin{lem} \label{lem:2+3}
	Let $s,t,\epsilon \in G$ be such that $\epsilon \ne 1$,
	$st\ne ts$, $s^G=\{s,\epsilon s\}$, and $|t^G|=3$.
	Then $\epsilon ^3=1$, $s\epsilon =\epsilon s$, $t\epsilon =\epsilon
	^{-1}t$, $ts=\epsilon st$, and $t^G=\{t,\epsilon t,\epsilon ^2t\}$. Moreover,
	$\epsilon^{-1}s\in Z(G)$.
\end{lem}

\begin{proof}
	We assumed that $st\ne ts$ and $s^G=\{s,\epsilon s\}$, and hence $ts=\epsilon
	st$. Thus $\epsilon s=s\epsilon $ and $t\epsilon =\epsilon^{-1}t$ by
	Lemma~\ref{lem:rs}(1). Therefore $s^kts^{-k}=\epsilon ^{-k}t\in t^G$ for all
	$k\ge 1$, that is, $\epsilon ^2=1$ or $\epsilon ^3=1$ because of $|t^G|=3$.
	To conclude the lemma it suffices to show that $\epsilon ^3=1$ and that
	$\epsilon^{-1}s\in Z(G)$.

	Assume to the contrary that $\epsilon ^2=1$. Let $t'\in t^G\setminus
	\{t,\epsilon t\}$. Then $st'=t's$ and $\epsilon t'=t'\epsilon $. In
	particular, $t'$ commutes with $s^G$, which is a contradiction, since $t'\in
	t^G$ and $t$ does not commute with $s^G$.

	Finally, Lemma~\ref{lem:rs}(3) implies that
	$(\epsilon^{-1}s)^G=\{\epsilon^{-1}s\}$.
\end{proof}

\begin{lem} \label{lem:2+3=typeB2}
	Let $s,t,\epsilon \in G$ be as in Lemma~\ref{lem:2+3}.
	Let $\sigma \in \chg {G^s}$, $\tau \in \chg{G^t}$
	and let $U,V\in \ydG $ be such that $U\simeq M(s,\sigma )$ and
	$V\simeq M(t,\tau )$. Then $a^{(U,V)}_{12}=-1$ and $a^{(U,V)}_{21}=-2$
	if and only if
	$$ \sigma (\epsilon t^2)\tau (\epsilon s^2)=1,\quad
	   \sigma (s)=-1,\quad \tau (t)=-1.
  $$
	In this case, if $(3)_{\sigma(\epsilon)}=0$ then
	$(\ad U)(V)\simeq M(\epsilon^{-1}st,\tau ')$ and $(\ad V)^2(U)\simeq
	M(\epsilon^{-1}t^2s,\sigma')$, where $\tau'\in \chg{G^t}$ with
	$\tau'(t)=\tau(\epsilon s^{-1})\sigma(\epsilon )$, $\tau'(h)=\sigma(h)\tau(h)$
	for all $h\in G^s\cap G^t$,
	and $\sigma'\in \chg G$ with $\sigma'(\epsilon )=1$,
	$\sigma'(t)=-\tau(\epsilon s^{-1})\sigma(\epsilon)$, and $\sigma'
	(h)=\tau(h)^2\sigma(h)$ for all $h\in G^s\cap G^t$.
\end{lem}

\begin{proof}
	By Lemma~\ref{lem:2+3}, the subgroup $\langle s,t\rangle \subseteq G$
	is a non-abelian epimorphic
	image of $\Gamma _3$. Hence $U$ and $V$
	satisfy the assumptions of
	\cite[Prop.\,4.12]{rank2} when viewed as Yetter-Drinfeld modules over $\langle
	s,t\rangle $. This leads to the claim.
\end{proof}

\subsection{}

In this section we study reflections of a particular pair of Yetter-Drinfeld
modules.  Let $G$ be a group and let $s\in G$.  Assume that $|s^G|=2$. Let
$r,\epsilon \in G$ be such that $rs=\epsilon sr$, $\epsilon \ne 1$.

Let $t\in Z(G)$, $\sigma \in \chg{G^s}$, and $\tau \in \chg G$.
In particular, $\tau (\epsilon )=1$.
Let $V,W\in \ydG $ be such that
$V\simeq M(s,\sigma )$ and $W\simeq M(t,\tau )$.
We determine the Yetter-Drinfeld modules $X_m^{V,W}$ for all $m\ge 1$.

\begin{lem}
  \label{lem:X1}
  The Yetter-Drinfeld module $X_1^{V,W}$ is non-zero
  if and only if $\sigma (t)\tau (s)\ne 1$.
  In this case, $X_1^{V,W}\simeq M(st,\tau _1)$, where
  $\tau _1$ is the character of $G^s=G^{st}$ with
  $\tau _1(h)=\sigma (h)\tau (h)$ for all $h\in G^s$.
\end{lem}

\begin{proof}
  Let $v\in V_s$ and $w\in W$ with $v,w\ne 0$.
  Since $G\triangleright (s,t)=s^G\times \{t\}$,
  $(\id -c_{W,V}c_{V,W})(v\otimes w)=(1-\sigma (t)\tau (s))v\otimes w$
  generates $X_1^{V,W}$ as a $\K G$-module. This implies the claim.
\end{proof}

\begin{lem} \label{lem:X2}
  Assume that $\sigma (t)\tau (s)\ne 1$.
  Then $X_2^{V,W}\ne 0$. Moreover,
  $X_2^{V,W}$ is absolutely simple if and only if one of the following hold.
  \begin{enumerate}
    \item $\sigma (\epsilon ^2)=1$,
      $(1+\sigma (s))(1-\sigma (st)\tau (s))=0$.
    \item $\sigma (s)=-1$, $\sigma (\epsilon ^2t^2)\tau (s^2)=1$.
    \item $\sigma (st)\tau (s)=1$, $\sigma (\epsilon ^2s^2)=1$.
  \end{enumerate}
  In this case, let $\lambda =-\sigma (\epsilon )$ in case (1),
  $\lambda =\sigma (\epsilon t)\tau (s)$ in case (2), and 
  $\lambda =\sigma (\epsilon s)$ in case (3). Then
  $X_2^{V,W}\simeq M(\epsilon s^2t,\tau _2)$,
  where $\tau _2\in \chg G$ with
  $$ 
  \tau _2(r)=\lambda \sigma (r^2)\tau (r),\quad
  \tau _2(g)=\sigma (g\,r^{-1}gr)\tau (g)$$
  for all $g\in G^s$,
  and $w_2=v\otimes rv\otimes w+\lambda rv\otimes v\otimes w$
  is a basis of $X_2^{V,W}$.
\end{lem}

\begin{proof}
  Let $w_1=v\otimes w$. By the proof of Lemma~\ref{lem:X1},
  $w_1\in (X_1^{V,W})_{st}$ generates $X_1^{V,W}$ as a $\K G$-module.
  Since $s^G\times (st)^G=G\triangleright (s,st)\cup G\triangleright
  (s,\epsilon st)$, the vectors
  $\varphi _2(v\otimes w_1)$ and $\varphi _2(v\otimes rw_1)$ generate 
  the $\K G$-module $X_2^{V,W}$.

  Let $w_2'=\varphi _2(v\otimes rw_1)$.
  Since
  \begin{align*}
    &\varphi _2(v\otimes rw_1)=
    v\otimes rw_1-\epsilon stv\otimes srw_1
    +\tau (r)(\id \otimes \varphi _1)(srv\otimes v\otimes w)
    \\
    &\qquad =(1-\sigma (\epsilon ^2s^2t)\tau (s))v\otimes rw_1
    +\sigma (\epsilon s)\tau (r)(1-\sigma (t)\tau (s))rv\otimes w_1,
  \end{align*}
  we conclude that $w_2'\ne 0$ and hence $X_2^{V,W}\ne 0$.

  Assume that $X_2^{V,W}$ is absolutely simple.
  Since
  \begin{align*}
    \varphi _2(v\otimes w_1)=&\;
    v\otimes w_1-stv\otimes sw_1+(\id \otimes \varphi _1)(sv\otimes v\otimes w)
    \\
    =&\;(1+\sigma (s))(1-\sigma (st)\tau (s))v\otimes w_1,
  \end{align*}
  and $\varphi _2(v\otimes w_1)\in (X_2^{V,W})_{s^2t}$,
  $w_2'\in (X_2^{V,W})_{\epsilon s^2t}$, and
  $(s^2t)^G\ne (\epsilon s^2t)^G$ by Lemma~\ref{lem:rs}(3), we conclude that
  \begin{align} \label{eq:X2:simple1}
    (1+\sigma (s))(1-\sigma (st)\tau (s))=0.
  \end{align}
  Also, the tensors $v\ot rv\ot w$, $rv\ot v\ot w$ form a basis
  of $(V\ot V\ot W)_{\epsilon s^2t}$, and hence
  $$gu=\sigma (g\,r^{-1}gr)\tau (g)u \quad \text{for all $u\in (V\ot V\ot
  W)_{\epsilon s^2t}$, $g\in G^s$.}$$
  Since $G=G^s\cup rG^s$,
  $$\K (v\ot rv+rv\ot v)\ot w, \quad \K (v\ot rv-rv\ot v)\ot w
  $$
  are the only simple Yetter-Drinfeld submodules of $(V\ot V\ot W)_{\epsilon s^2t}$.
  Thus, $w_2'$ has to span one of these submodules, that is,
  $$ 1-\sigma (\epsilon ^2s^2t)\tau (s)=
     \lambda \sigma (\epsilon s)(1-\sigma (t)\tau (s))
  $$
  for some $\lambda \in \{1,-1\}$.
  Equivalently,
\begin{align}
  (1-\lambda \sigma (\epsilon s))(1+\lambda \sigma (\epsilon st)\tau (s))=0
\label{eq:w2'}
\end{align}
  for some $\lambda \in \{1,-1\}$.
  This and Equation~\eqref{eq:X2:simple1} imply that (1) or (2) or (3) hold,
  and $X_2^{V,W}=\K (v\ot rv\ot w+\lambda rv\ot v\ot w)$.
 
  Conversely, if one of (1), (2), (3) holds, then $X_2^{V,W}=\K w_2$
  by the above calculations, and hence $X_2^{V,W}$ is absolutely simple. The
  remaining claims also follow similarly. 
\end{proof}

\begin{lem} \label{lem:X3}
  Assume that $\sigma (t)\tau (s)\ne 1$ and that $X_2^{V,W}$ is absolutely
  simple.
  Let $\tau _3$ be the character of $G^s$ and $\tau _4$ be the character
  of $G$ with
  \begin{align*}
    \tau _3(g)=&\;\sigma (g^2 \,r^{-1}gr)\tau (g), &
    \tau _4(g)=&\;\sigma (g^2\,(r^{-1}gr)^2)\tau (g), &
    \tau _4(r)=&\;\sigma (r^4)\tau (r)
  \end{align*}
  for all $g\in G^s$. Then the following hold.
  \begin{enumerate}
    \item $X_3^{V,W}=0$ if and only if $\sigma (s)=-1$ or $\sigma (\epsilon
      ^2)\ne 1$.
    \item $X_3^{V,W}$ is absolutely simple if and only if
      $\sigma (s)\ne -1$ and $\sigma (\epsilon ^2)=1$. In this case,
      $X_3^{V,W}\simeq M(\epsilon s^3t,\tau _3)$ and $X_4^{V,W}\ne 0$.
    \item Assume that $\sigma (s)\ne -1$ and $\sigma (\epsilon ^2)=1$. Then
      $X_4^{V,W}$ is absolutely simple if and only if $(3)_{\sigma (s)}=0$.
      In this case,
      $X_4^{V,W}\simeq M(\epsilon ^2 s^4t,\tau _4)$ and $X_5^{V,W}=0$.
    \item Assume that $\sigma (\epsilon ^2)=1$ and $(3)_{\sigma (s)}=0$. Let
      $w_2$ be as in Lemma~\ref{lem:X2},
      $w_3=v\otimes w_2$,
      and $$w_4=v\otimes rw_3+\sigma (r^2)\tau (r)rv\otimes w_3.$$
      Then $w_3\in (X_3^{V,W})_{\epsilon s^3t}$,
      $w_4\in (X_4^{V,W})_{\epsilon ^2s^4t}$.
  \end{enumerate}
\end{lem}

\begin{proof}
  First we calculate that
  \begin{align*}
    \varphi _3(v\otimes w_2)=(1+\sigma (s))(1-\sigma (\epsilon ^2s^3t)\tau (s))
    v\otimes w_2.
  \end{align*}
  Hence $\varphi _3(v\otimes w_2)=0$ if and only if $\sigma (s)=-1$
  or $\sigma (\epsilon ^2s^3t)\tau (s)=1$.
  Assume that $\sigma (s)\ne -1$. Since $X_2^{V,W}$ is absolutely simple,
	Lemma~\ref{lem:X2} implies that $\sigma (st)\tau (s)=1$.
	Thus $X_3^{V,W}=0$ if and only if $\sigma (\epsilon ^2s^2)=1$.
  Since $\sigma(s)^{-1}=\sigma (t)\tau (s)\ne 1$ and $\sigma(s)\ne -1$ by
	assumption,
	Lemma~\ref{lem:X2} implies that $\sigma (\epsilon^2s^2)=1$ holds if and only
if $\sigma (\epsilon ^2)\ne 1$.

  Assume now that $\sigma (\epsilon ^2)=1$ and $\sigma (s)\ne -1$.
  Then $\sigma (st)\tau (s)=1$ by Lemma~\ref{lem:X2}.
  Let $w_3=v\ot w_2$. Then $w_3\in (V^{\ot 3}\ot W)_{\epsilon s^3t}$
  and
  $$X_3^{V,W}=\K w_3+\K rw_3 \simeq M(\epsilon s^3t,\tau _3),$$
  since $gw_2=\sigma (g\,r^{-1}gr)\tau (g)w_2$ for all $g\in G^s$ by
  Lemma~\ref{lem:X2}.
  Moreover,
  \begin{align*}
    \varphi _4(v\otimes w_3)=&\;
    (3)_{\sigma (s)}(1-\sigma (s^3))v\otimes w_3,\\
    \varphi _4(v\otimes rw_3)=&\;
    (1-\sigma (s^5))v\otimes rw_3\\
    &\;-\sigma (sr^2)\tau (r)
    (1+\sigma (s))(1-\sigma (s^2))rv\ot w_3.
  \end{align*}
  Since $V\otimes V\otimes X_2^{V,W}=X_4'\oplus X_4''$ in $\ydG $,
  where
  \begin{align*}
    X_4'=&\;v\otimes v\otimes X_2^{V,W}+rv\otimes rv\otimes X_2^{V,W},\\
    X_4''=&\;v\otimes rv\otimes X_2^{V,W}+rv\otimes v\otimes X_2^{V,W},
  \end{align*}
  similarly to an argument in the proof of Lemma~\ref{lem:X2} we conclude
  that
  $X_4^{V,W}$ is absolutely simple if and only if
  $\varphi _4(v\otimes w_3)=0$ and
  $$\varphi _4(v\ot rw_3)\in \K(v\otimes rv+\lambda rv\otimes v)\otimes w_2$$
  for some $\lambda \in \K $ with $\lambda ^2=1$.
  This is equivalent to $(3)_{\sigma (s)}=0$, since then
  $\varphi _4(v\ot rw_3)=(1-\sigma (s)^{-1})w_4$
  and $rw_4=\sigma (r^4)\tau (r)w_4$.
  The rest follows easily.
\end{proof}

Now we introduce classes of pairs of absolutely simple Yetter-Drinfeld
modules over any group $H$. They will appear naturally in
	Corollary~\ref{cor:finiteW} in the classification of specific pairs admitting
all reflections.

\begin{defn} \label{def:wp22}
  Let $H$ be a group.
	For $i\in \{0,1\}$ let $\wp ^H_{22,i}$ be the class of pairs $(V,W)$
  of Yetter-Drinfeld modules over $H$ such that
  the following hold.
  \begin{enumerate}
    \item $|\supp V|=2$, $|\supp W|=2$.
		\item There exist $s\in \supp V$, $t\in \supp W$, $\sigma \in \chg{H^s}$,
			and $\tau \in \chg{H^t}$, such that
      $V\simeq M(s,\sigma )$, $W\simeq M(t,\tau )$, and the following hold:
			\begin{enumerate}
				\item If $i=0$, then $(\id -c_{W,V}c_{V,W})(V\otimes W)=0$.
        \item If $i=1$, then $\sigma (\epsilon t^2)\tau (\epsilon s^2)=1$, and
          $\sigma (s)=\tau (t)=-1$,
					where $\epsilon \in H$ with $st=\epsilon ts$ and $\epsilon \ne 1$.
			\end{enumerate}
  \end{enumerate}
  Let $\wp ^H_i$ for $0\le i\le 8$ be the class of pairs $(V,W)$
  of Yetter-Drinfeld modules over $H$ such that
  the following hold.
  \begin{enumerate}
    \item $|\supp V|=2$, $|\supp W|=1$.
		\item There exist $s\in \supp V$, $t\in \supp W$, $\sigma \in \chg{H^s}$,
			and $\tau \in \chg H$, such that
      $V\simeq M(s,\sigma )$, $W\simeq M(t,\tau )$,
      and $\sigma $ and $\tau $ satisfy the conditions in Table~\ref{tab:wpH}.
  \end{enumerate}
	For all $n\in \N$ with $n\ge 2$ let $\wp^H_1(n)$ be the subclass of $\wp^H_1$
	of those pairs $(V,W)$, where additionally
	$\tau (t)$ is a primitive $n$-th root of $1$.
\end{defn}

\begin{table}
\caption{The classes $\wp^H_i$, $0\le i\le 8$.}
\begin{center}
  \begin{tabular}{c|c}
    $i$ & conditions on $\sigma $ and $\tau $\\
    \hline
    $0$ & $\sigma (t)\tau (s)=1$ \\
    $1$ & $\sigma (\epsilon ^2)=1$, $\sigma (s)=-1$, $\sigma (t)\tau (st)=1$,
    $\tau (t)\ne 1$ \\
    $2$ & $\sigma (\epsilon ^2)=1$, $\sigma (s)=-1$, $\tau (t)=-1$,
    $(3)_{\sigma (t)\tau (s)}=0$, $\sigma (t)\tau (s)\ne 1$ \\
    $3$ & $\sigma (\epsilon ^2)=1$, $\sigma (s)=-1$, $(3)_{\sigma (t)\tau (s)}=0$,
    $\tau (t)=-\sigma (t)\tau (s)$, $\sigma (t)\tau (s)\ne 1$ \\
    $4$ & $\sigma (\epsilon ^2)=1$, $(3)_{\sigma (s)}=0$, $\sigma (st)\tau (s)=1$,
    $\tau (t)=-1$, $\sigma (s)\ne 1$ \\
    $5$ & $\sigma (\epsilon ^2)\ne 1$, $\sigma (s)=-1$,
    $\sigma (\epsilon ^2t^2)\tau (s^2)=1$, $\sigma (t)\tau (st)=1$\\
    $6$ & $\sigma (\epsilon ^2)\ne 1$, $\sigma (s)=-1$,
    $\sigma (\epsilon ^2t^2)\tau (s^2)=1$, $\tau (t)=-1$\\
    $7$ & $\sigma (\epsilon ^2)\ne 1$, $\sigma (\epsilon ^2s^2)=1$,
    $\sigma (st)\tau (s)=1$, $\sigma (t)\tau (st)=1$\\
    $8$ & $\sigma (\epsilon ^2)\ne 1$, $\sigma (\epsilon ^2s^2)=1$,
    $\sigma (st)\tau (s)=1$, $\tau (t)=-1$
  \end{tabular}
\end{center}
\label{tab:wpH}
\end{table}

%
%
%
We point out that Lemma~\ref{lem:a=0for22} gives a characterization of pairs in
$\wp^H_{22,0}$. A characterization of the class $\wp^H_{22,1}$ was given in
Corollary~\ref{co:2-2:A2B2CartanMatrix}.

The pairs $(V,W)$ in the classes $\wp ^H_{22,j}$ for $j\in \{0,1\}$
and $\wp^H_i$ for $0\le i\le 8$ 
satisfy stronger properties. To prove them we need a lemma.

For any group $H$ and any representation $\rho $ of $H$ we write
$\const _\rho (H)$ for the normal subgroup of $H$ consisting of those $g\in H$
such that $\rho (g)$ is constant. In particular, $\const _\rho (H)=H$
if $\deg \rho =1$. The following Lemma is probably well-known. It follows
	directly from the structure theory of Yetter-Drinfeld modules over groups.

\begin{lem} \label{lem:constH}
  Let $H$ be a group and let $V\in \ydH $.
  Then the following hold.
  \begin{enumerate}
    \item For all $r\in \supp V$ there exists a representation
      $\rho _r$ of $H^r$ such that $\oplus _{s\in r^H}V_s\simeq M(r,\rho _r)$.
      These representations are unique up to isomorphism, and $\deg \rho _r=\deg
      \rho _s$ for all $r,s\in \supp V$ with $s\in r^H$.
    \item Let $r\in \supp V$, $h\in \const _{\rho _r}(H^r)$, and $g\in H$.
      Let $r'=grg^{-1}$ and $h'=ghg^{-1}$. Then
      $h'\in \const _{\rho _{r'}}(H^{h'})$ and $\rho _r(h)=\rho _{r'}(h')$.
  \end{enumerate}
\end{lem}

In the following two propositions we show that the presentation of the pairs in
	the classes $\wp _{22}^H$ and $\wp ^H_i$, $0\le i\le 8$,
	in terms of elements of the group $H$ and representations of their
	centralizers is essentially independent of choices. This simplifies much the
discussion of skeletons of tuples.

\begin{pro} \label{pro:wp22univ}
	Let $H$ be a group, $(V,W)\in \wp ^H_{22,1}$,
	and $s\in \supp V$, $t\in \supp W$.
	Let $\epsilon \in H$ be such that $st=\epsilon ts$.
\begin{enumerate}
  \item There exist unique characters $\sigma $ of $H^s$ and $\tau $ of $H^t$
    such that $V\simeq M(s,\sigma )$ and $W\simeq M(t,\tau )$.
  \item 
		$s^H=\{s,\epsilon s\}$, $t^H=\{t,\epsilon t\}$, $\epsilon ^2=1$,
    $\epsilon \in Z(H)$, $\epsilon \ne 1$.
  \item
		$\sigma (\epsilon t^2)\tau (\epsilon s^2)=1$, $\sigma (s)=\tau (t)=-1$.
\end{enumerate}
\end{pro}

\begin{proof}
  %
  By assumption, there exist $s'\in \supp V$, $t'\in \supp W$,
  $\epsilon '\in H$,
  such that $s't'=\epsilon 't's'$ and $\epsilon '\ne 1$.
	Since $|\supp V|=|\supp W|=2$ and since $\supp V$, $\supp W$
		are conjugacy classes of $H$, (2) follows from Lemma~\ref{lem:A3:222}(1).
  In particular, there exists
  $x\in \langle s,t\rangle $ such that $x\triangleright s'=s$,
  $x\triangleright t'=t$.
  Then $x\triangleright \epsilon '=\epsilon$.

  Again by assumption, there exist characters $\sigma '$ of $H^{s'}$
  and $\tau '$ of $H^{t'}$ such that $V\simeq M(s',\sigma ')$,
  $W\simeq M(t',\tau ')$,
  and
  $$ \sigma '(\epsilon 't'^2)\tau '(\epsilon 's'^2)=1, \quad
     \sigma '(s')=\tau '(t')=-1.
  $$
  Then (1) holds by Lemma~\ref{lem:constH}(1), and (3) follows
  from Lemma~\ref{lem:constH}(2) with $r=s'$, $g=x$ and $r=t'$, $g=x$,
  respectively.
\end{proof}

\begin{pro} \label{pro:wpiuniv}
  Let $H$ be a group, $i\in \Z $ with $0\le i\le 8$, $(V,W)\in \wp ^H_i$,
  and $s\in \supp V$, $t\in \supp W$. Let $\epsilon  \in H$ be such that
  $s^H=\{s,\epsilon s\}$.
\begin{enumerate}
  \item There exist unique characters $\sigma $ of $H^s$ and $\tau $ of $H$
    such that $V\simeq M(s,\sigma )$ and $W\simeq M(t,\tau )$.
  \item $\sigma $ and $\tau $ satisfy the conditions in
    Table~\ref{tab:wpH}.
  \item If $n\in \N $ and $(V,W)\in \wp^H_1(n)$, then $\tau (t)$ is a
		primitive $n$-th root of $1$.
\end{enumerate}
\end{pro}

\begin{proof}
  Similar to the proof of Proposition~\ref{pro:wp22univ}.
\end{proof}

As before, let $G$ be a group, $V,W\in \ydG $ with $|\supp V|=2$ and $|\supp
W|=1$, $s\in \supp V$, $t\in \supp W$,
$\epsilon \in G$ with $s^G=\{s,\epsilon s\}$,
$\sigma $ a character of $G^s$, and $\tau $ a character of
$G$. Assume that $V\simeq M(s,\sigma )$ and
$W\simeq M(t,\tau )$. Then $\epsilon \ne 1$.

\begin{pro} \label{pro:simpleorzero}
  Assume that $\sigma (t)\tau (s)\ne 1$.
  Then $(\ad V)^m(W)$ and $(\ad W)^m(V)$ are absolutely simple or zero for all
$m\in \N $ if and only if the following hold.
  \begin{enumerate}
    \item $\sigma (\epsilon ^2)=1$, $\sigma (s)=-1$, or\\
      $\sigma (\epsilon ^2t^2)\tau (s^2)=1$, $\sigma (s)=-1$, $\sigma (\epsilon ^2)\ne 1$, or\\
      $\sigma (\epsilon ^2s^2)=\sigma (st)\tau (s)=1$, $\sigma (\epsilon ^2)\ne 1$, or\\
      $\sigma (\epsilon ^2)=\sigma (st)\tau (s)=1$, $(3)_{\sigma (s)}=0$.
    \item $(n+1)_{\tau (t)}(1-\sigma (t)\tau (st^n))=0$
      for some $n\ge 1$.
  \end{enumerate}
  Moreover, the four possibilities in (1) are mutually exclusive.
\end{pro}

\begin{proof}
  This follows from Lemmas~\ref{lem:X1}, \ref{lem:X2}, \ref{lem:X3},
  \ref{lem:genRosso}.
\end{proof}

Proposition~\ref{pro:simpleorzero} leads to a characterization of those pairs
$(V,W)$ which have a finite Weyl groupoid. Before obtaining this
characterization, we need to conclude some technicalities. For the definitions
of $\tau _2$, $\tau _4$, and $\sigma_n$ we refer to Lemmas~\ref{lem:X2}, \ref{lem:X3},
and \ref{lem:genRosso}, respectively.

\begin{lem} 
	\label{lem:21reflections}\
	\begin{enumerate}
		\item
			Assume that $\sigma (t)\tau (s)\ne 1$, $\sigma (\epsilon ^2)=1$,
			and that $\sigma (s)=-1$. Then
			$R_1(V,W)\simeq (M(s^{-1},\sigma ^*),M(\epsilon s^2t,\tau _2))$ and
			\begin{align*}
				\sigma ^*(s^{-1})=&\;-1,&
				\sigma ^*(\epsilon ^{-2})=&\;1,\\
				\sigma ^*(\epsilon s^2t)\tau _2(s^{-1})=&\;
				\sigma (t^{-1})\tau (s^{-1}),&
				\tau _2(\epsilon s^2t)=&\;\sigma (t^2)\tau (s^2t).
			\end{align*}
		\item
			Assume that
			$\sigma (\epsilon ^2t^2)\tau (s^2)=1$, $\sigma (s)=-1$, and
			that $\sigma (\epsilon ^2)\ne 1$. Then
			$R_1(V,W)\simeq (M(s^{-1},\sigma ^*),M(\epsilon s^2t,\tau _2))$ and
			\begin{align*}
				\sigma ^*(s^{-1})=&\;-1,&
				\sigma ^*(\epsilon ^{-1})=&\;\sigma (\epsilon ),\\
				\sigma ^*(\epsilon s^2t)\tau _2(s^{-1})=&\; \sigma (t)\tau (s),&
				\tau _2(\epsilon s^2t)=&\;\tau (t).
			\end{align*}
		\item
			Assume that $\sigma (\epsilon ^2s^2)=1$,
			$\sigma (st)\tau (s)=1$, and that $\sigma (\epsilon ^2)\ne 1$. Then
			$R_1(V,W)\simeq (M(s^{-1},\sigma ^*),M(\epsilon s^2t,\tau _2))$ and
			\begin{align*}
				\sigma ^*(s^{-1})=&\;\sigma (s),&
				\sigma ^*(\epsilon ^{-1})=&\;\sigma (\epsilon ),\\
				\sigma ^*(\epsilon s^2t)\tau _2(s^{-1})=&\;\sigma (t)\tau (s),&
				\tau _2(\epsilon s^2t)=&\;\tau (t).
			\end{align*}
		\item
			Assume that $\sigma (t)\tau (s)\ne 1$, $\sigma (\epsilon ^2)=1$,
			$\sigma (st)\tau (s)=1$, and $(3)_{\sigma (s)}=0$. Then
			$R_1(V,W)\simeq (M(s^{-1},\sigma ^*),M(\epsilon ^2s^4t,\tau _4))$ and
			\begin{align*}
				\sigma ^*(s^{-1})=&\;\sigma (s),&
				\sigma ^*(\epsilon ^{-2})=&\;1,\\
				\sigma ^*(\epsilon ^2s^4t)\tau _4(s^{-1})=&\;
				\sigma (t)\tau (s),&
				\tau _4(\epsilon ^2s^4t)=&\;\tau (t).
			\end{align*}
		\item Let $n\in \N $. Assume that $\sigma (t)\tau (st^n)=1$
			and that $\tau (t^k)\ne 1$ for all $1\le k\le n$. Then
			$R_2(V,W)\simeq (M(st^n,\sigma _n),M(t^{-1},\tau ^*))$, and
			\begin{align*}
				\sigma _n(st^n)=&\;\sigma (s),&
				\sigma _n(\epsilon )=&\;\sigma (\epsilon ),\\
				\sigma _n(t^{-1})\tau ^*(st^n)=&\;\sigma (t)\tau (s),&
				\tau ^*(t^{-1})=&\;\tau (t).
			\end{align*}
		\item Assume that $(\sigma (t)\tau (s))^2\ne 1$ and that
			$\tau (t)=-1$. Then
			$R_2(V,W)\simeq (M(st,\sigma _1),M(t^{-1},\tau ^*))$, and
			\begin{align*}
				\sigma _1(st)=&\;-\sigma (st)\tau (s),&
				\sigma _1(\epsilon ^2)=&\;\sigma (\epsilon ^2),\\
				\sigma _1(t^{-1})\tau ^*(st)=&\;\sigma (t^{-1})\tau (s^{-1}),&
				\tau ^*(t^{-1})=&\;-1.
			\end{align*}
	\end{enumerate}
\end{lem}

\begin{proof}
	The claims follow from Lemmas~\ref{lem:X2}, \ref{lem:X3}, and
	\ref{lem:genRosso}.
	For example, in the first three cases one obtains that
	$X_2^{V,W}\ne 0$, $X_3^{V,W}=0$, and
	\begin{align*}
		\sigma ^*(s^{-1})=&\;\sigma (s),\quad
		\sigma ^*(\epsilon ^{-2})=\sigma (\epsilon ^2),\\
		\sigma ^*(\epsilon s^2t)\tau _2(s^{-1})=&\;
		\sigma (\epsilon ^{-2}s^{-4}t^{-1})\tau (s^{-1}),\\
		\tau _2(\epsilon s^2t)=&\;\sigma (\epsilon ^2s^4t^2)\tau (s^2t).
	\end{align*}
	The additional assumptions then imply the formulas.
\end{proof}

\begin{rem} 
  \label{rem:wpreflections}
	{}From Lemmas~\ref{lem:X1} and~\ref{lem:21reflections} we obtain the
  Cartan matrix entries and reflections of the pairs in the classes $\wp ^G_n$
	for $0\le n\le 8$. We collect these data in Table~\ref{tab:wprefl}.
  \begin{table}[h]
	  \caption{Reflections of pairs $(V,W)\in \wp_n$.}
		\label{tab:wprefl}
	  \begin{tabular}{|c|c|c|c|c|}
		  \hline 
		  $(V,W)$ & $a_{12}^{(V,W)}$ & $a_{21}^{(V,W)}$ & $R_{1}(V,W)$ &
		  $R_{2}(V,W)$\rule{0pt}{3ex}\tabularnewline
		  \hline 
		  $\wp_{0}^{G}$ & $0$ & $0$ & $\wp_{0}^{G}$ & $\wp_{0}^{G}$\rule{0pt}{3ex} \tabularnewline
		  \hline 
		  $\wp_{1}^{G}$ & $-2$ & $-1$ & $\wp_{1}^{G}$ & $\wp_{1}^{G}$\rule{0pt}{3ex} \tabularnewline
		  \hline 
		  $\wp_{2}^{G}$ & $-2$ & $-1$ & $\wp_{3}^{G}$ & $\wp_{4}^{G}$\rule{0pt}{3ex} \tabularnewline
		  \hline 
		  $\wp_{3}^{G}$ & $-2$ & $-2$ & $\wp_{2}^{G}$ & $\wp_{3}^{G}$\rule{0pt}{3ex} \tabularnewline
		  \hline 
		  $\wp_{4}^{G}$ & $-4$ & $-1$ & $\wp_{4}^{G}$ & $\wp_{2}^{G}$\rule{0pt}{3ex} \tabularnewline
		  \hline 
		  $\wp_{5}^{G}$ & $-2$ & $-1$ & $\wp_{5}^{G}$ & $\wp_{5}^{G}$\rule{0pt}{3ex} \tabularnewline
		  \hline 
		  $\wp_{6}^{G}$ & $-2$ & $-1$ & $\wp_{6}^{G}$ & $\wp_{8}^{G}$\rule{0pt}{3ex} \tabularnewline
		  \hline 
		  $\wp_{7}^{G}$ & $-2$ & $-1$ & $\wp_{7}^{G}$ & $\wp_{7}^{G}$\rule{0pt}{3ex} \tabularnewline
		  \hline 
		  $\wp_{8}^{G}$ & $-2$ & $-1$ & $\wp_{8}^{G}$ & $\wp_{6}^{G}$\rule{0pt}{3ex} \tabularnewline
		  \hline 
	  \end{tabular}
  \end{table}
\end{rem}


\begin{cor} \label{cor:finiteW}
  The following are equivalent.
  \begin{enumerate}
    \item The pair $(V,W)$ admits all reflections and 
      $\cW(V,W)$ is finite.
    \item $(V,W)\in \wp^G _i$ for some $0\le i\le 8$.
  \end{enumerate}
  If $(V,W)\in \wp^G_0$, then $(V,W)$ is standard of type $A_1\times A_1$.
	If $(V,W)\in \wp^G_i$ with $i\in \{1,5,6,7,8\}$, then $(V,W)$ is standard of type $C_2$.
  If $(V,W)\in \wp^G_i$ with $2\le i\le 4$,
  then $\rroots{(V,W)}_+$
  can be obtained from \cite[Lemma\,8.5]{MR3276225}.
\end{cor}

\begin{proof}
  (2)$\Rightarrow $(1) Since $(V,W)\in \wp^G _i$ for some $0\le i\le 8$, the
  pair $(V,W)$ admits all reflections by Remark~\ref{rem:wpreflections}.
  Moreover, the Weyl groupoid $\cW(V,W)$ is finite since the set of roots of
  $(V,W)$ is finite.

  (1)$\Rightarrow $(2)
  Assume that $(V,W)$ admits all reflections and that $\cW(V,W)$ is
  finite. Then $(\ad V)^m(W)$ and $(\ad W)^m(V)$ are absolutely simple or zero
	for all $m\ge 1$ by Theorem~\ref{thm:admabsimple}.
  Lemmas~\ref{lem:X2}, \ref{lem:X3}, \ref{lem:genRosso}
  imply that all reflections of $(V,W)$
  are pairs $(V',W')$ of absolutely simple Yetter-Drinfeld modules,
  such that there exist $s',\epsilon '\in G$, $t'\in Z(G)$, and characters
  $\sigma '$ of $G^{s'}$
	and $\tau '$ of $G$ with $\epsilon '\ne 1$, $s'^G=\{s',\epsilon 's'\}$,
  $V\simeq M(s',\sigma ')$, $W\simeq M(t',\tau ')$.
  By Theorem~\ref{thm:finitetype}, there exists an object
  $(V',W')$ of $\cW(V,W)$ with a Cartan matrix of finite type.
	By Remark~\ref{rem:wpreflections}, the reflections $R_1$ and $R_2$
  induce permutations of the classes $\wp^G _i$ with $0\le i\le 8$. Hence
  it suffices to show that $(V,W)\in \wp^G _i$ for some $0\le i\le 8$
	if the Cartan matrix $A^{(V,W)}$ is of finite type.

  Assume that $A^{(V,W)}$ is of finite type different from $A_1\times A_1$.
  Then $\sigma (t)\tau (s)\ne 1$, and
  we obtain that $a^{(V,W)}_{12}\le -2$ by Lemma~\ref{lem:X2}.
  Further, $a^{(V,W)}_{12}\in \{-2,-4\}$ by Lemma~\ref{lem:X3}.
  Hence $a^{(V,W)}_{12}=-2$ and $a^{(V,W)}_{21}=-1$.
  Then
  $$\sigma (\epsilon ^2)=1,\quad \sigma (s)=-1$$
  or
  $$\sigma (\epsilon ^2t^2)\tau (s^2)=1,\quad
    \sigma (s)=-1,\quad
	  \sigma (\epsilon ^2)\ne 1$$
  or
  $$\sigma (\epsilon ^2s^2)=1,\quad \sigma (st)\tau (s)=1,\quad
    \sigma (\epsilon ^2)\ne 1
  $$
  by Lemma~\ref{lem:X3}, and
  $$(\tau (t)+1)(1-\sigma (t)\tau (st))=0$$
  by Lemma~\ref{lem:genRosso}.
  By the same lemmas, $R_1(V,W)\simeq (M(s^{-1},\sigma ^*),M(\epsilon s^2t,\tau _2))$ and
  $R_2(V,W)\simeq (M(st,\sigma _1),M(t^{-1},\tau ^*))$.

  If $\sigma (\epsilon ^2)\ne 1$, then $(V,W)\in \wp _i$ for some $5\le i\le 8$.
  So assume that $\sigma (\epsilon ^2)=1$ and $\sigma (s)=-1$.

  If $\sigma (t)\tau (st)=1$, then $(V,W)\in \wp^G _1$.  Assume now that $\tau
  (t)=-1$ and $(\sigma (t)\tau (s))^2\ne 1$.  Then
  Lemma~\ref{lem:21reflections}(6) for $(V,W)$ and
  Proposition~\ref{pro:simpleorzero} for $R_2(V,W)$ implies that $(3)_{\sigma
  _1(st)}=(3)_{\sigma (t)\tau (s)}=0$, since $\sigma _1(st)=\sigma (t)\tau
  (s)\ne -1$.  Then $(V,W)\in \wp^G_2$. This completes the proof.
\end{proof}

\section{Rank two classification} 

\label{appendix:rank2}

In this appendix we collect the main results of \cite{MR3276225, MR3272075, rank2}.
The results are presented in the terminology of this paper. Many of the examples
will be described using Definition~\ref{def:skeleton}. However, to include all
the Nichols algebras found in \cite{MR3276225, MR3272075, rank2}, one needs to add
some additional diagrams. 

\subsection{}
\label{G2}
We first describe the examples related to the group $\Gamma_2$ of
\cite[\S1.1]{rank2}. For the Nichols algebras of dimension $64$ one has the
following skeleton:
\[
\begin{picture}(100,20)
	\DDvertextwo{10}{8}
	\multiput(13,10)(5,0){5}{\line(1,0){3}}
	\DDvertextwo{39}{8}
\end{picture}
\]
In characteristic three, the pair of Yetter-Drinfeld modules which yields
Nichols algebras of dimension $1296$ has the following skeleton:
\[
  \begin{picture}(100,20)
	  \DDvertextwo{10}{8}
    \multiput(13,9)(5,0){5}{\line(1,0){3}}
    \multiput(13,11)(5,0){5}{\line(1,0){3}}
		\put(22,7){$>$}
	  \DDvertextwo{39}{8}
  \end{picture}
  \makebox[0pt]{\raisebox{8pt}{$\charK =3$}}\\
\]

\subsection{}
\label{G3}
Let us review the examples related to the group $\Gamma_3$,
see~\cite[\S1.4]{rank2}.  For the Nichols algebras of dimension $2304$ related
to the group $\Gamma_3$,~\cite[Example 1.11, \S1.4]{rank2}, one has the
following diagrams related by reflections:
\begin{equation*}
	\begin{picture}(100,20)
		\DDvertexone{10}{10}
		\put(10,10){\circle{4}}
		\put(13,9){\line(1,0){23}}
		\put(13,11){\line(1,0){23}}
		\put(22,7){$>$}
		\DDvertexthree{40}{10}
	\end{picture}
	\begin{picture}(100,20)
		\DDvertextwo{10}{8}
		\multiput(13,9)(5,0){5}{\line(1,0){3}}
		\multiput(13,11)(5,0){5}{\line(1,0){3}}
		\put(22,7){$>$}
		\DDvertexthree{39}{10}
	\end{picture}
\end{equation*}
We remark that the diagram on the left is not a skeleton in the sense of
Definition~\ref{def:skeleton} because the simple Yetter-Drinfeld module
$M(s_1,\sigma_1)$ is constructed with a two-dimensional representation
$\sigma_1$. This situation is described with a double circle at the left vertex
of the diagram.

The examples of dimensions $10368$, $5184$ or $1152$ can be described with the
following skeleton:
\[
\begin{picture}(100,22)
	\DDvertexone{10}{10}
	\put(9,15){$p$}
	\put(20,15){$p^{-1}$}
	\put(13,9){\line(1,0){23}}
	\put(13,11){\line(1,0){23}}
	\put(22,7){$>$}
	\DDvertexthree{40}{10}
\end{picture}
\makebox[0pt]{\raisebox{8pt}{$(3)_{-p}=0$}}\\
\]
We remark that in this case, an extra assumption on the value of
$p=\sigma_1(s_1$) is needed.

The examples of dimension $2239488$ related to the group $\Gamma_3$ 
of~\cite[Example 1.9, \S1.4]{rank2} can be described
with the following diagrams related by reflections:
\begin{equation*}
	\begin{picture}(100,20)
		\DDvertextwo{10}{8}
		\multiput(13,7)(5,0){5}{\line(1,0){3}}
		\multiput(13,9)(5,0){5}{\line(1,0){3}}
		\multiput(13,11)(5,0){5}{\line(1,0){3}}
		\multiput(13,13)(5,0){5}{\line(1,0){3}}
		\put(18,7){$\big<$}
		\put(27,7){$\big>$}
		\DDvertexthree{39}{10}
	\end{picture}
	\begin{picture}(100,22)
		\DDvertexone{10}{10}
		\put(8,15){$1$}
		\put(13,9){\line(1,0){23}}
		\put(13,11){\line(1,0){23}}
		\put(22,7){$>$}
		\DDvertexthree{40}{10}
	\end{picture}
	\begin{picture}(100,22)
		\DDvertexone{10}{10}
		\put(8,15){$1$}
		\put(13,7){\line(1,0){23}}
		\put(13,9){\line(1,0){23}}
		\put(13,11){\line(1,0){23}}
		\put(13,13){\line(1,0){23}}
		\put(22,7){$\big>$}
		\DDvertexthree{40}{10}
	\end{picture}
\end{equation*}
The diagram on the left is not a skeleton in the sense of
Definition~\ref{def:skeleton} since it has a double arrow. This double arrow
means that the Cartan matrix of the pair satisfies
$a_{12}^{(M_1,M_2)}=a_{21}^{(M_1,M_2)}=-2$. 

\subsection{}
\label{T}
Nichols algebras related to the group $T$ have dimension $1259712$ over fields
of characteristic two and $80621568$ otherwise, see~\cite[\S1.3]{rank2}. In this
case one has the following skeleton:
\[
\begin{picture}(100,22)
	\DDvertexone{10}{10}
	\put(9,15){$p$}
	\put(20,15){$p^{-1}$}
	\put(13,8){\line(1,0){23}}
	\put(13,10){\line(1,0){23}}
	\put(13,12){\line(1,0){23}}
	\put(22,7){$>$}
	\DDvertexT{40}{10}
\end{picture}
\makebox[0pt]{\raisebox{8pt}{$(3)_{-p}=0$}}\\
\]
The dots on the right vertex describe the structure of the support of
$M(s_2,\sigma_2)$ which is isomorphic (as a quandle) to the tetrahedron quandle.
Further, the assumption $(3)_{-p}=0$, where $p=\sigma_1(s_1)$, is needed.
	
\subsection{}
\label{G4}
Nichols algebras related to the group $\Gamma_4$ have dimension $65536$ over
fields of characteristic two and $262144$ otherwise, see \cite[\S1.2]{rank2}.
In this case one has the following skeleton:
\[
\begin{picture}(100,20)
	\DDvertextwo{10}{8}
	\multiput(13,9)(5,0){5}{\line(1,0){3}}
	\multiput(13,11)(5,0){5}{\line(1,0){3}}
	\put(22,7){$>$}
	\DDvertexsq{39}{8}
\end{picture}
\]
The four dots in the right vertex mean that the support of $M_2(s_2,\sigma_2)$
is isomorphic (as a quandle) to the dihedral quandle $\D_4$.

\subsection{}
With these diagrams, the classification of finite-dimensional Nichols algebras
admiting a finite root system of rank two, \cite[Theorem 2.1]{rank2}, can be
reformulated as follows.

\begin{thm*}
	Let $G$ be a non-abelian group and $M$ in $\Ggen_2^G$. Assume that $M$ is
	braid-indecomposable.  The following are equivalent:
	\begin{enumerate}
		\item $M$ has a skeleton appearing in
			~\eqref{G2}--\eqref{G4}.
		\item $\NA(M)$ is finite-dimensional.
		\item $M$ admits all reflections and $\cW(M)$ is finite.
	\end{enumerate}
\end{thm*}


\subsection*{Acknowledgement}

Leandro Vendramin was supported by Conicet and the Alexander von Humboldt
Foundation. Part of this work was done during his visit to ICTP (Trieste).
We also thank the referee for his numerous comments and suggestions.

\bibliographystyle{abbrv}
\bibliography{refs}

\end{document}